\documentclass{article}
\usepackage[utf8]{inputenc}
\usepackage[all,cmtip]{xy}
\usepackage{amssymb,amsmath,amsthm}
\usepackage{mathtools, mathdots}
\usepackage{fancyhdr} 
\usepackage{mathrsfs, calligra, bbm}
\usepackage[top=1.5in, bottom=1.5in, left=1in, right=1in]{geometry}
\usepackage{hyperref}
\usepackage{tikz-cd}
\usetikzlibrary{decorations.pathmorphing}
\usepackage{dirtytalk}
\usepackage{tikz}
\usepackage{enumitem}
\usepackage[T1]{fontenc}
\usepackage{graphicx}
\usepackage{subcaption}
\usepackage{array}
\usepackage{graphicx}
\usepackage[toc,page]{appendix}
\usepackage{stfloats}
\usepackage{CJKutf8}
\usepackage[strict=true, style=british]{csquotes}
\usepackage[british]{babel}
\usepackage[backend=biber,style=alphabetic,sorting=anyt, maxbibnames=99,maxcitenames=99,maxalphanames=99,
]{biblatex}
\newcounter{oldtocdepth}

\let\emptyset\varnothing

\DeclareMathAlphabet{\mathpzc}{OT1}{pzc}{m}{it}

\newtheorem{Definition}{Definition}[subsection]
\newtheorem{Theorem}[Definition]{Theorem}
\newtheorem{Lemma}[Definition]{Lemma}
\newtheorem{Proposition}[Definition]{Proposition}
\newtheorem{Corollary}[Definition]{Corollary}
\newtheorem{Example}[Definition]{Example}

\newtheorem{Conjecture}[]{Conjecture}
\newtheorem{Assumption}[]{Assumption}

\newtheorem*{Theorem*}{Theorem}

\theoremstyle{remark}
\newtheorem{Remark}[Definition]{Remark}
\newtheorem*{Remark*}{Remark}

\newcommand\scalemath[2]{\scalebox{#1}{\mbox{\ensuremath{\displaystyle #2}}}}

\DeclareMathOperator\A{\mathbf{A}}

\DeclareMathOperator\C{\mathbf{C}}

\DeclareMathOperator\F{\mathbf{F}\!}

\DeclareMathOperator\Q{\mathbf{Q}}
\DeclareMathOperator\R{\mathbf{R}}
\DeclareMathOperator\T{\mathbf{T}}

\DeclareMathOperator\Z{\mathbf{Z}}

\DeclareMathOperator\bbA{\mathbb{A}}

\DeclareMathOperator\bbG{\mathbb{G}}

\DeclareMathOperator\bbT{\mathbb{T}}

\DeclareMathOperator\calA{\mathcal{A}}

\DeclareMathOperator\calC{\mathcal{C}}
\DeclareMathOperator\calD{\mathcal{D}}
\DeclareMathOperator\calE{\mathcal{E}}
\DeclareMathOperator\calF{\mathcal{F}}

\DeclareMathOperator\calK{\mathcal{K}}

\DeclareMathOperator\calO{\mathcal{O}}

\DeclareMathOperator\calU{\mathcal{U}}
\DeclareMathOperator\calV{\mathcal{V}}
\DeclareMathOperator\calW{\mathcal{W}}
\DeclareMathOperator\calX{\mathcal{X}}
\DeclareMathOperator\calY{\mathcal{Y}}
\DeclareMathOperator\calZ{\mathcal{Z}}

\DeclareMathOperator\scrE{\mathscr{E}}
\DeclareMathOperator\scrF{\mathscr{F}}
\DeclareMathOperator\scrG{\mathscr{G}}
\DeclareMathOperator\scrH{\mathscr{H}}

\DeclareMathOperator\scrO{\mathscr{O}}

\DeclareMathOperator\scrT{\mathscr{T}}

\DeclareMathOperator\scrZ{\mathscr{Z}}

\DeclareMathOperator\frakL{\mathfrak{L}}

\DeclareMathOperator\fraka{\mathfrak{a}}

\DeclareMathOperator\frakl{\mathfrak{l}}
\DeclareMathOperator\frakm{\mathfrak{m}}
\DeclareMathOperator\frakn{\mathfrak{n}}

\DeclareMathOperator\frakp{\mathfrak{p}}
\DeclareMathOperator\frakq{\mathfrak{q}}

\DeclareMathOperator\GL{GL}

\DeclareMathOperator\SU{SU}

\DeclareMathOperator\End{End}
\DeclareMathOperator\Hom{Hom}

\DeclareMathOperator\ad{ad}

\DeclareMathOperator\adj{adj}

\DeclareMathOperator\coker{coker}

\DeclareMathOperator\cts{cts}

\DeclareMathOperator\Fil{Fil}
\DeclareMathOperator{\fin}{fin}

\DeclareMathOperator\Gal{Gal}

\DeclareMathOperator\id{id}

\DeclareMathOperator\image{image}

\DeclareMathOperator\Iw{Iw}

\DeclareMathOperator\ord{ord}

\DeclareMathOperator\opp{opp}
\DeclareMathOperator\Par{par}

\DeclareMathOperator\pr{pr}

\DeclareMathOperator\rig{rig}
\DeclareMathOperator\Res{Res}

\DeclareMathOperator\Shi{\text{\begin{CJK}{UTF8}{min}し\end{CJK}}}

\DeclareMathOperator\Spa{Spa}

\DeclareMathOperator\supp{supp}

\DeclareMathOperator\tot{tot}

\DeclareMathOperator\tr{tr}
\DeclareMathOperator\trans{^{\mathtt{t}}\!}

\DeclareMathOperator\wt{wt}

\DeclareMathOperator\Ab{{\mathsf{Ab}}}

\DeclareMathOperator\Alg{{\mathsf{Alg}}}

\DeclareMathOperator\Groups{{\mathsf{Group}}}

\DeclareMathOperator\Sets{{\mathsf{Sets}}}

\DeclareMathOperator\bfalpha{\boldsymbol{\alpha}}
\DeclareMathOperator\bfbeta{\boldsymbol{\beta}}
\DeclareMathOperator\bfgamma{\boldsymbol{\gamma}}
\DeclareMathOperator\bfdelta{\boldsymbol{\delta}}

\DeclareMathOperator\bfomega{\boldsymbol{\omega}}

\DeclareMathOperator\bfupsilon{\boldsymbol{\upsilon}}

\DeclareMathOperator\llbrack{\![\![\!}
\DeclareMathOperator\rrbrack{\!]\!]}

\makeatletter
\renewcommand{\maketitle}{\bgroup\setlength{\parindent}{0pt}
\begin{flushleft}
  \LARGE{\textbf{\@title}}
  
  \vspace{4mm}
  
  \large{\textsc{\@author}}
  
  \vspace{4mm}
\end{flushleft}\egroup
}
\makeatother

\title{On $p$-adic adjoint $L$-functions for Bianchi cuspforms: the $p$-split case}
\author{Pak-Hin Lee and Ju-Feng Wu}
\date{}

\begin{document}

\maketitle

{\footnotesize
\paragraph{Abstract.} We construct a Hecke-equivariant pairing on the overconvergent cohomology of Bianchi threefolds. Applying the strategy of Kim and Bellaïche, we use this pairing to construct $p$-adic adjoint $L$-functions for Bianchi cuspforms and show that it detects the ramification locus of the cuspidal Bianchi eigenvariety over the weight space. Combining results of Barrera Salazar--Williams, we show a non-vanishing result of this $p$-adic adjoint $L$-function at certain points. Finally, we obtain a formula relating this pairing with the adjoint $L$-values of the corresponding cuspidal Bianchi eigenforms (of level 1).
}

\tableofcontents

\section{Introduction}

\subsection{Background}

In the late 20th century, K. Doi and H. Hida (\cite{Doi-Ohta, DHI}) studied congruences between cuspforms. It was later realised by Hida that the adjoint $L$-value of a cuspform governs its congruences with other cuspforms (\cite{ Hida-CongCuspForm, Hida-CongCuspForm2, Hida-CongHecke}). We refer the readers to \cite{Hida-adjoint} for a summary of the history and an exposition of Hida's works. 

While Hida studied \emph{ordinary} families of modular forms, a generalisation (called \emph{finite-slope} families of modular forms) was proposed by R. Coleman (\cite{Coleman-overconvergent, Coleman-family}). Together with B. Mazur, they later discovered that finite-slope families of modular forms can be patched into a geometric object -- the \emph{eigencurve} (\cite{Coleman--Mazur}). As a consequence, the study of congruences between modular forms is then translated to the study of the geometry of the eigencurve. 

Naturally inspired by the works of Hida, Coleman, and Mazur, one would wonder the relation between the adjoint $L$-values and the geometry of the eigencurve. Works in this direction were first carried out by W. Kim in his Ph.D. thesis (\cite{Kim}) and rewritten by J. Bellaïche in \cite{Eigen}.

Let us briefly discuss the method of Kim and Bellaïche. Let $f$ be a cuspidal eigenform of level $\Gamma_0(N)$ and weight $k$. The works of G. Shimura and Hida provide a formula \[
    \langle f, f \rangle_{\Gamma_0(N),k} = (*) \cdot  L(\mathrm{ad}^0 f, 1),
\]
where $L(\mathrm{ad}^0 f, s)$ is the adjoint $L$-function attached to $f$, $\langle \cdot, \cdot \rangle_{\Gamma_0(N),k}$ is the Petersson inner product, and $(*)$ is an explicit factor. Let $p$ be a prime number that does not divide $N$ and let $f^{(p)}$ be a $p$-stabilisation of $f$. We further assume $f^{(p)}$ has non-critical slope.  Inspired by the formula above, Kim and Bellaïche constructed a (\emph{twisted}) pairing $[\cdot, \cdot]_{k}$ on the space of \emph{overconvergent modular symbols}. By using a general framework in commutative algebra, they showed that $[\cdot, \cdot]_{k}$ defines a function $L_{\calU, h}^{\adj}$, locally at the point $x_f$ attached to $f^{(p)}$ in the (cuspidal) eigencurve.\footnote{ Following the terminology in \cite{Eigen}, we have to assume $x_f$ is a \emph{good point}.} Then, they show that \begin{enumerate}
    \item[(i)] The weight map from the eigencurve to the weight space is étale at $x_f$ if and only if $L_{\calU, h}^{\adj}(x_f) \neq 0$. 
    \item[(ii)] There is an explicit relation between $[\cdot, \cdot]_{k}$ and $\langle \cdot, \cdot \rangle_{\Gamma_0(N),k}$.
\end{enumerate}
In other words, the pairing  $[\cdot, \cdot]_{k}$ exhibits the relation between the adjoint $L$-value and the geometry of the (cuspidal) eigencurve.

It turns out that the eigencurve admits generalisations to other automorphic forms (\cite{Buzzard-eigen, Ash-Stevens, Urban-eigen, Hansen-PhD}). It is then a natural question to ask whether the method of Kim and Bellaïche can be generalised to these cases. Indeed, the generalisation to $p$-adic families of Hilbert cuspforms is provided by Balasubramanyam--Bergdall--Longo in \cite{BL} while the Siegel case is worked out by the second-named author in \cite{Wu-pairing}. The present paper concerns another generalisation, namely to $p$-adic families of Bianchi cuspforms. In this case, we encounter the distinctive feature that automorphic forms contribute to cohomology in two degrees and the adjoint pairing is defined on two different cohomology groups. We shall provide a more detailed description in the next subsection.

\subsection{Main results}\label{subsection: main results}
Let $K$ be an imaginary quadratic number field and let $p$ be an odd prime that splits in $K$. Let $\frakn$ be an ideal in $\calO_K$ such that $p\not\in \frakn$. Let $G := \Res_{\Z}^{\calO_K}\GL_2$. In this paper, we consider two Bianchi threefolds: \begin{align*}
    Y & := \text{ the Bianchi threefold of level }\Gamma_0(\frakn)\\
    Y_{\Iw_G} & := \text{ the Bianchi threefold above $Y$ with an extra Iwahori-level at $p$}.
\end{align*}
Here, the levels are given by compact open subgroups of $G(\A_{\Q})$, whose definitions can be found in \S \ref{subsection: Bianchi threefold}. We remark that there is a natural map $Y_{\Iw_G} \rightarrow Y$.

Given $k = (k_1, k_2)\in \Z_{> 0}^2$ with $k_1=k_2$,\footnote{ When considering the cohomology of the Bianchi threefolds, one can work with any $k\in \Z_{\geq 0}^2$. We restrict to the parallel $k\in \Z_{>0}^2$ as it is assumed in our main result below (see Lemma \ref{Lemma: concentration of the cohomology degree}). That is, in the case of parallel $k\in \Z_{>0}^2$, we used the fact that the degree-0 Eisenstein part vanishes (Theorem \ref{Theorem: ESH isomorphism}).} let $V_k$ be the irreducible representation of (parallel) weight $k$ of $G$ and let $V_k^{\vee}$ be its dual. Then, the representation $V_k^{\vee}$ defines a local system on both $Y$ and $Y_{\Iw_G}$ and so one obtains a commutative diagram of cohomology groups \[
    \begin{tikzcd}
        H_c^t(Y, V_k^{\vee}) \arrow[r]\arrow[d] & H^t(Y, V_k^{\vee})\arrow[d]\\
        H_c^t(Y_{\Iw_G}, V_k^{\vee}) \arrow[r] & H^t(Y_{\Iw_G}, V_k^{\vee}),
    \end{tikzcd}
\] where the horizontal maps are natural maps from the compactly supported cohomology to Betti cohomology, and the vertical maps are induced by the natural map $Y_{\Iw_G} \rightarrow Y$. We denote by $H_{\Par}^t$ the image of $H_c^t$ in $H^t$ and obtain a natural map \[
    H_{\Par}^t(Y, V_k^{\vee}) \rightarrow H_{\Par}^t(Y_{\Iw_G}, V_k^{\vee}). 
\] 
Note that this map doesn't preserve Hecke eigenforms. Indeed, given a Hecke eigenclass $[\mu]\in H_{\Par}^t(Y, V_k^{\vee})$, its image in $H_{\Par}^t(Y_{\Iw_G}, V_k^{\vee})$ decomposes into a linear combination of four Hecke eigenforms $[\mu]_{(i,j)}^{(p)}$ -- the \emph{$p$-stabilisations} of $[\mu]$ -- where $(i,j)\in (\Z/2\Z)^2$ (see Proposition \ref{Proposition: p-stabilisation}).

To construct Bianchi eigenvarieties, we have to enlarge the coefficient $V_k^{\vee}$ to \emph{analytic distributions} $D_{\kappa_{\calU}}^r$, where $\kappa_{\calU}$ is a $p$-adic family of (parallel) weights. In particular, by studying $H_{\Par}^t(Y_{\Iw_G}, D_{\kappa_{\calU}}^r)$, one can construct the (cuspidal) parallel Bianchi eigenvariety $\calE_{\circ}$ together with a weight map \[
    \wt: \calE_{\circ} \rightarrow \calW_{\circ},
\] where $\calW_{\circ}$ is the (parallel) $p$-adic weight space.\footnote{ Our results work in a more general situation. We restrict to the case of parallel Bianchi eigenvariety to simplify the discussion in the introduction. For more general results, we refer the readers to \S \ref{subsection: Hecke and eigenvariety} and \S \ref{section: pairing}. 
}

To generalise the works of Kim and Bellaïche, we first exhibit a pairing \[
    D_{\kappa_{\calU}}^r \times D_{\kappa_{\calU}}^r \rightarrow R_{\calU},
\]
which then induces a Hecke-equivariant pairing
\[
    [\cdot, \cdot]_{\kappa_{\calU}}: H_{\Par}^1(Y_{\Iw_G}, D_{\kappa_{\calU}}^r) \times H_{\Par}^2(Y_{\Iw_G}, D_{\kappa_{\calU}}^r) \rightarrow R_{\calU}.
\]
Our first result reads as follows.

\begin{Theorem}[Corollary \ref{Corollary: non-vanishing result for p-adic adjoint L-function} and Corollary \ref{Corollary: non-vanishing result for p-adic adjoint L-function; ordinary case}]\label{Theorem: nonvanishing of the p-adic adjoint L-function; intro}
    Let $f$ be a cuspidal Bianchi eigenform of weight $k$ on $Y_{\Iw_G}$. Suppose $f$ is non-critical and is a regular $p$-stabilised newform (in the sense of \cite{BW}) Let $\frakm_f$ be the maximal ideal in the Hecke algebra defined by $f$. 
    \begin{enumerate}
        \item[(i)] Let $(R_{\calU},\kappa_{\calU})$ be a sufficiently small (parallel) $p$-adic family of weights that contains $k$. Let $h\in \Q_{\geq 0}$ such that $(\calU, h)$ is slope-adapted. The pairing $[\cdot, \cdot]_{\kappa_{\calU}}$ induces a non-degenerate pairing \[
            [\cdot, \cdot]_{\kappa_{\calU}, \frakm_{f}} : H_{\Par}^1(Y_{\Iw_G}, D_{\kappa_{\calU}}^r)^{\leq h}_{\frakm_f} \times H_{\Par}^2(Y_{\Iw_G}, D_{\kappa_{\calU}}^r)^{\leq h}_{\frakm_f} \rightarrow R_{\calU, \frakm_k}.
        \]
        \item[(ii)] Let $x_{f}$ be the point in $\calE_{\circ}$ defined by $f$. Assume that $x_f$ varies a family over $\calU$ (Assumption \ref{Assumption: vary in 1-dimensional family}).\footnote{ Since $\calU$ is an affinoid open in the parallel weight space (in this introduction) and we assumed Assumption \ref{Assumption: vary in 1-dimensional family} is satisfied, it is expected that $x_f$ is a twist of base-change point if $f$ is non-ordinary (Conjecture \ref{Conjecture: Calegari--Mazur conjecture}). We refer the readers to \cite{Calegari-Mazur} and \cite[\S 5.3]{BW} for more discussions in this direction.} Then, $\wt$ is étale at $x_f$.
        \item[(iii)] The pairing  $[\cdot, \cdot]_{\kappa_{\calU}}$ defines a $p$-adic adjoint $L$-function $L_{\calU, h}^{\adj}$ on a (sufficiently small) neighbourhood $\calX \subset \calE_{\circ}$ of $x_f$ such that for any point $y$ in $\calX$, $L_{\calU, h}^{\adj}(y) \neq 0$ if and only if $\wt$ is étale at $y$. 
        \item[(iv)] Consequently, we have a non-vanishing result \[
            L_{\calU, h}^{\adj}(x_f) \neq 0.
        \]
        \item[(v)] Similar statements hold for ordinary families (in the sense of Definition \ref{Definition: ordinary part}). 
    \end{enumerate}
\end{Theorem}

To make the link between our adjoint pairing and the adjoint $L$-value, we now fix a cuspidal Bianchi eigenform $f$ of level $1$.\footnote{ Such an assumption is unnecessary. We impose it just to simplify the notation in computation. Readers may find it already complicated in this case (see the computation in \S \ref{subsection: pairing and p-stabilisations}).} For any prime ideal $\frakq \subset \calO_K$, we fix a uniformiser $\varpi_{\frakq}\in \calO_{K, \frakq}$ and consider the Hecke operator $T_{\bfupsilon_{\frakq}}$ defined by $\bfupsilon_{\frakq} = \begin{pmatrix}1 & \\ & \varpi_{\frakq}\end{pmatrix}$ (see \S \ref{subsection: Hecke operators for classical cohomology} for the definition). Note that $p \calO_K= \frakp\overline{\frakp}$, thus $\calO_{K, \frakp} \simeq \calO_{K, \overline{\frakp}} \simeq \Z_p$; so we particularly choose $\varpi_{\frakp} = \varpi_{\overline{\frakp}} = p$. Denote by $\lambda_{f, \frakq} = \lambda_{\frakq}$ the $T_{\bfupsilon_{\frakq}}$-eigenvalue of $f$ and consider the Hecke polynomial \[
    P_{\frakq}(X) := X^2 - \lambda_{\frakq}X + q_{\frakq}^{k+1},
\] where $q_{\frakq}$ is the cardinality of the residue field of $\calO_{K, \frakq}$. Let $\alpha_{\frakq}^{(0)}$ and $\alpha_{\frakq}^{(1)}$ be the two roots of $P_{\frakq}$. The adjoint $L$-function attached to $f$ is defined to be the Euler product \[
    L(\mathrm{ad}^0 f, s) := \prod_{\frakq} \left( \left(1-\frac{\alpha_{\frakq}^{(0)}}{\alpha_{\frakq}^{(1)}}q_{\frakq}^{-s}\right) \left(1-q_{\frakq}^{-s}\right) \left(1-\frac{\alpha_{\frakq}^{(1)}}{\alpha_{\frakq}^{(0)}}q_{\frakq}^{-s} \right) \right)^{-1},
\]
which converges when $\Re(s) $ is sufficiently large. 

Combining with Urban's formula in \cite{Urban-PhD}, we have the following formula for the adjoint $L$-value.

\begin{Theorem}[Proposition \ref{Proposition: pairing and adjoint L-value}, Corollary \ref{Corollary: nonvanishing Euler factor}, and Corollary \ref{Corollary: interpolation property of p-adic adjoint L-function}]\label{Theorem: formula; intro}
    Let $f$ a cuspidal Bianchi eiegnform of level 1 as above and let $[\mu_{f, 1}]$ and $[\mu_{f, 2}]$ be its image in the parabolic cohomology of degree 1 and 2 respectively via the Eichler--Shimura--Harder isomorphism (\cite{Harder}). Suppose moreover that $\lambda_{f, \frakp} = \lambda_{\frakp}$ and $\lambda_{f, \overline{\frakp}} = \lambda_{\overline{\frakp}}$ are nonzero real numbers and both $P_{\frakp}$ and $P_{\overline{\frakp}}$ are irreducible over $\R$ (see Assumption \ref{Assumption: Hecke eigenvalues at p are nonzero}).
    \begin{enumerate}
        \item[(i)] Let $(i, j)\in (\Z/2\Z)^2$, then \[
            \left[ [\mu_{f, 1}]_{(i,j)}^{(p)}, [\mu_{f, 2}]_{(i,j)}^{(p)}\right]_k = \frac{(p+1)^2\Theta(\lambda_{f,\frakp}, \lambda_{f,\overline{\frakp}}, i,j)D_{\infty}(k+1, \Phi_{\infty})\mathrm{disc}(K)}{16 \pi} \times  L(\mathrm{ad}^0 f, 1),
        \] where $D_{\infty}(k+1, \Phi_{\infty})$ is an explicit factor in \cite{Urban-PhD} and $\Theta(\lambda_{f,\frakp}, \lambda_{f,\overline{\frakp}}, i,j)$ is an explicit factor in \eqref{eq: [p-stab f, p-stab f]}.
        \item[(ii)] For $i=1,2$, suppose $[\mu_{f, t}]_{(i,j)}^{(p)}$ is non-critical, then $\Theta(\lambda_{f,\frakp}, \lambda_{f,\overline{\frakp}}, i,j)$ is nonzero. 
        \item[(iii)] Same assumption as in (ii), let $x_{f_{(i,j)}^{(p)}}$ be the point of $\calE_{\circ}$ defined by $[\mu_{f, 1}]_{(i,j)}^{(p)}$. We also assume that Assumption \ref{Assumption: vary in 1-dimensional family} is satisfied. Then, \[
            L_{\calU, h}^{\adj}\left(x_{f_{(i,j)}^{(p)}}\right) = (*) L(\ad^0 f, 1),
        \] where $(*)$ is a non-zero factor depending on $f$.
        \item[(iv)] If $g$ is another cuspidal Bianchi eigenform that satisfies the same assumptions for $f$ as above. Let $x_{g_{(i,j)}^{(p)}}$ be the point of $\calE_{\circ}$ defined by $[\mu_{g, 1}]_{(i,j)}^{(p)}$. Suppose $x_{g_{(i,j)}^{(p)}}$ lies in a (sufficiently small) neighbourhood $\calX\subset \calE_0$ of $x_{f_{(i,j)}^{(p)}}$, then we have a similar formula \[
            L_{\calU, h}^{\adj}\left(x_{g_{(i,j)}^{(p)}} \right) = (*) L(\ad^0 g, 1),
        \]
        where $(*)$ is now a non-zero factor depending on $g$. 
    \end{enumerate}
\end{Theorem}

\begin{Remark}\label{Remark: Assumption 2; intro}
    Let us comment on Assumption \ref{Assumption: Hecke eigenvalues at p are nonzero} in the statement above: \begin{enumerate}
        \item[(i)] The assumption that both $\lambda_{f,\frakp}$ and $\lambda_{f,\overline{\frakp}}$ are nonzero real numbers is a technical assumption in our computation. We are only able to link $[\cdot, \cdot]_k$ with the adjoint $L$-value under this assumption. 
        \item[(ii)] The assumption that $P_{\frakp}$ and $P_{\overline{\frakp}}$ are irreducible over $\R$ is in the spirit of the Ramanujan conjecture. We refer the readers to Remark \ref{Remark: explanation of Assumption 2} for a more detailed discussion. 
    \end{enumerate}
\end{Remark}

\begin{Remark}\label{Remark: factors in the formula; intro}
    Readers can see that, although $\Theta(\lambda_{f,\frakp}, \lambda_{f,\overline{\frakp}}, i,j)$ is explicit, it is difficult to determine when it is nonzero since its formula is complicated. 
    The second assertion in Theorem \ref{Theorem: formula; intro} assures that the trivial zero phenomenon should not happen when $f$ is regular, which is as expected. Although it is certainly interesting to study what will happen when $\Theta(\lambda_{f,\frakp}, \lambda_{f,\overline{\frakp}}, i,j) = 0$, we do not pursue in this direction in the present paper. 
\end{Remark}

\begin{Remark}\label{Remark: further directions; intro}
    We make some further remarks regarding future directions: \begin{enumerate}
        \item[(i)] Having results for the $p$-split case, one may immediately wonder the situation when $p$ is inert in $K$. We have some ideas to tackle this question. In fact, such an idea suggested a formalism for $\GL_2$ over general number fields where $p$ is unramified. We wish to study this general situation in the near future. 
        \item[(ii)] In \cite{Lee-PhD}, the first-named author studied a construction of the \emph{twisted} $p$-adic adjoint $L$-function of a base-change form. It is a natural question to ask how it is related to the work of the present paper. Indeed, one would expect the $p$-adic adjoint $L$-function of a base-change form $f$ to be decomposed as a product of the twisted $p$-adic adjoint $L$-function constructed in \emph{loc. cit.} and the $p$-adic adjoint $L$-function constructed in \cite{Kim, Eigen}. We wish to continue our study in this direction. 
\end{enumerate}
\end{Remark}

\subsection{Leitfaden of the paper}
The present paper is organised as follows: 

In \S \ref{section: preliminaries on Bianchi forms}, we start our discussion with the cohomology groups of Bianchi threefolds. We then spell out the action of the Hecke operators when the coefficients of the cohomology groups are classical. Using this information, we discuss the construction of $p$-stabilisations of Hecke-eigenclasses in Proposition \ref{Proposition: p-stabilisation}. The final part of this section reviews the relation between the cohomology groups of Bianchi threefolds and Bianchi cuspforms. Later in the paper, we use such relation to pass from one world to the other. 

In \S \ref{section: eigenvariety}, we introduce the overconvergent cohomology groups for Bianchi threefolds and use them to construct (cuspidal) Bianchi eigenvarieties. This section starts with some discussion on $p$-adic analysis, which is supposed to be well-known to experts. In the following subsection, we briefly discuss the weight space in our case and specify two special kinds of ($p$-adic families of) weights that we shall be considering: \emph{small weights} and \emph{affinoid weights}. We remark that these ideas come from \cite{Hansen-Iwasawa, CHJ}. Then, we study the overconvergent cohomology groups in \S \ref{subsection: overconvergent cohomology} and the action of Hecke operators on them in \S \ref{subsection: Hecke and eigenvariety}. Let us mention that our study on $p$-adic analysis allows us to construct eigenvarieties by only using (modified) distribution spaces.\footnote{ Compared with the construction in \cite{Hansen-PhD}, D. Hansen constructed eigenvarieties by first using analytic functions to construct \emph{spectral varieties} and then using distributions to construct eigenvarieties. } We should also point out that due to the obstruction described in \cite[Lemma 4.2]{BW}, we have to restrict ourselves to $p$-adic families of Bianchi forms varying over a (nice) one-dimensional weight space.\footnote{ More precisely, if $\calU$ is a two-dimensional affinoid weight space and $(\calU, h)$ is slope-adapted, then \cite[Lemma 4.2]{BW} shows that $H_c^1(Y_{\Iw_G}, D_{\kappa_{\calU}}^r)^{\leq h} =0$. However, as we have seen in \S \ref{subsection: main results}, the construction of the $p$-adic adjoint $L$-function involves a construction of a Hecke-equivariant pairing between $H^1_{\Par}$ and $H^2_{\Par}$. As we need the information from $H^1_{\Par}$, we cannot work over such a two-dimensional weight space $\calU$ and have to restrict ourselves to working with $p$-adic families varying over a (nice) one-dimensional weight space. Similar restrictions can be found in \cite[\S 4.2]{BW}. 
} In \S \ref{subsection: parallel eigenvarieties}, we show that the families that vary over the parallel weight space particularly satisfy the imposed restriction. In \S \ref{subsection: control theorems}, we discuss the control theorems. By virtue of the use of small weights, we are then able to realise ordinary families in finite-slope families. In particular, we prove a Hida-style control theorem via Stevens's control theorem on overconvergent cohomology groups (Theorem \ref{Theorem: Hida's control theorem}).

The construction of the \emph{adjoint pairing} for cuspidal Bianchi forms is provided in \S \ref{section: pairing}. This is a generalisation of the case for modular forms (\cite{Kim, Eigen}); and the basic idea comes from a previous work of the second-named author (\cite{Wu-pairing}).  Applying the formalism in \cite[\S 9.1]{Eigen}, we construct a $p$-adic adjoint $L$-function $L_{\calU, h}^{\adj}$ (varying in finite-slope families) via the information we gain from the adjoint pairing. Moreover, such a formalism also implies that this $p$-adic adjoint $L$-function detects the ramification locus of the weight map. In some situation, we are able to combine with the results in \cite{BW} to deduce a non-vanishing result of $L_{\calU, h}^{\adj}$ (Corollary \ref{Corollary: non-vanishing result for p-adic adjoint L-function}). The final part of the section shows that the same strategy also works for ordinary families.

Finally, \S \ref{section: adjoint L-values} is dedicated to the link between the pairing and the adjoint $L$-values of cuspidal Bianchi eigenforms. We follow the strategy in \cite[\S 9.5]{Eigen} to compute the relation between the classical pairing and our adjoint pairing. Compared with the case for modular forms, we encountered computational difficulties with $p$-stabilisations: while an elliptic eigenform only has two $p$-stabilisations, a Bianchi eigenform has four $p$-stabilisations. Nevertheless, the computations presented in \S \ref{subsection: pairing and p-stabilisations} show that, using the construction of the $p$-stabilisation in Proposition \ref{Proposition: p-stabilisation} together with a certain reasonable assumption (Assumption \ref{Assumption: Hecke eigenvalues at p are nonzero}), one can still calculate an explicit formula. In \S \ref{subsection: adjoint L-value formula}, we combine the formula of the adjoint $L$-value in \cite{Urban-PhD} with our computation and obtain a formula for the adjoint $L$-values (Proposition \ref{Proposition: pairing and adjoint L-value} and Corollary \ref{Corollary: interpolation property of p-adic adjoint L-function}).

We remark that we defined Hecke operators in terms of cohomological correspondences. Such a definition should be well-known to experts but we were not able to find a reference. Thus, we provide a gentle study on the basic ingredients we need for this viewpoint in \S \ref{section: cohomological correspondences for topological spaces}.

\subsection*{Acknowledgements}
We thank Chris Williams for insightful discussions about the overconvergent cohomology of Bianchi threefolds and for his valuable suggestions on an early draft of this paper. We thank Luis Santiago Palacios Moyano for enlightening conversations on the $p$-stabilisations of Bianchi forms. We thank James Rawson for his interest in our work; with him, we had several interesting conversations regarding the geometry of Bianchi eigenvarieties. We also thank John Bergdall for his comments and discussion on the previous version of this work and Ho Leung Fong for pointing out a mistake in \S \ref{section: cohomological correspondences for topological spaces} in the previous version. Finally, we thank the anonymous referee for valuable suggestions and corrections. While working on this project, P.-H.L. was supported by EPSRC Standard Grant EP/S020977/1 and J.-F.W. was supported by the  ERC Consolidator grant `Shimura varieties and the BSD conjecture' and the Irish Research Council under grant number IRCLA/2023/849 (HighCritical).

\subsection*{Notation and conventions}
Throughout this article, we fix the following notation and conventions: \begin{itemize}
    \item Let $K$ be an imaginary quadratic number field. We denote by $\calO_K$ the ring of integers of $K$. 
    \item Let $p\geq 3$ be a prime number. We assume $p$ splits in $\calO_K$ and write $p\calO_K = \frakp\overline{\frakp}$. In particular, we have the completions $\calO_{K, \frakp} = \Z_p = \calO_{K, \overline{\frakp}}$ and  $K_{\frakp} = \Q_p = K_{\overline{\frakp}}$.
    \item The ring of adèles of $\Q$ (resp., $K$) is denoted by $\A_{\Q}$ (resp., $\A_{K}$). Note that $\A_{\Q} = \R\times\A_{\Q, \fin}$ (resp., $\A_{K} = \C \times \A_{K, \fin}$), where $\A_{\Q, \fin}$ (resp., $\A_{K, \fin}$) is the ring of finite adèles. Moreover, we write $\widehat{\calO}_K := \calO_K \otimes_{\Z}\widehat{\Z}$. 
    \item We fix an algebraic closure $\overline{\Q}_p$ for $\Q_p$ and denote by $\C_p$ the $p$-adic completion of $\overline{\Q}_p$. We also fix an (algebraic) isomorphism $\C_p \simeq \C$ and so we can view any finite extension of $\Q_p$ as a subfield of $\C$. 
    \item In principle, capital letters in calligraphic font (\emph{e.g.}, $\calX$, $\calY$, $\calZ$) stand for adic spaces over $\Spa(\Z_p, \Z_p)$; capital letters in script font (\emph{e.g.}, $\scrE$, $\scrF$, $\scrG$) stand for sheaves on some geometric object, which shall be clear in the context. 
\end{itemize}

\section{Preliminaries on Bianchi cuspforms}\label{section: preliminaries on Bianchi forms}
In this section, we recall some basic knowledge of Bianchi cuspforms. In \S \ref{subsection: Bianchi threefold}, we study Bianchi threefolds and their cohomology groups. In \S \ref{subsection: Hecke operators for classical cohomology}, we describe the Hecke operators in terms of cohomological correspondences. The basic ingredients we need for this are discussed in \S \ref{section: cohomological correspondences for topological spaces}. We then close this section with a recollection of the relation between Bianchi modular forms and the cohomology groups of Bianchi threefolds. 

\subsection{The Bianchi threefolds and their cohomology groups}\label{subsection: Bianchi threefold}

Throughout this paper, we consider the algebraic group \[
    G := \Res^{\calO_K}_{\Z}\GL_{2/\calO_K}.
\]
That is, for any ring $R$, \[
    G(R) := \GL_2(R\otimes_{\Z}\calO_K).
\] 
In particular, since $p$ splits in $K$, we have \[
    G(\Z_p) = \GL_2(\Z_p \otimes_{\Z}\calO_K) = \GL_2(\Z_p) \times \GL_2(\Z_p).
\]

Let $B_{\GL_2/\calO_K}$ be the upper triangular Borel of $\GL_{2/\calO_K}$  and so \[
    B_G := \Res_{\Z}^{\calO_K}B_{\GL_2/\calO_K}
\]
is a Borel subgroup of $G$. In particular, we have \[
    B_G(\Z_p) = B_{\GL_2}(\Z_p) \times B_{\GL_2}(\Z_p).
\]
Here, $B_{\GL_2}$ is the upper triangular Borel of $\GL_2$ defined over $\Q$. Similarly, let $N_{\GL_2/\calO_K}$ (resp., $T_{\GL_2/\calO_K}$) be the unipotent radical (resp., torus) associated with $B_{\GL_2/\calO_K}$, then \[
    N_G = \Res^{\calO_K}_{\Z} N_{\GL_2/\calO_K} \quad (\text{resp., }T_G = \Res^{\calO_K}_{\Z} T_{\GL_2/\calO_K})
\] is the unipotent radical (resp., torus) associated with $B_G$.

We further consider \[
    \Iw_G := \left\{ \bfgamma\in G(\Z_p): (\bfgamma \mod p) \in B_G(\F_p) \right\} = \Gamma_0(p) \times \Gamma_0(p),
\]
where \[
    \Gamma_0(p) := \{\bfgamma\in \GL_2(\Z_p): (\bfgamma\mod p)\in B_{\GL_2}(\F_p)\}.
\]
Moreover, we have the Iwahori decomposition \[
    \Iw_G = N_{G, 1}^{\opp}B_G(\Z_p),
\]
where \[
    N_{G, 1}^{\opp} = \left\{(\bfgamma_1, \bfgamma_2)\in G(\Z_p): \bfgamma_i = \begin{pmatrix} 1 & \\ c_i & 1\end{pmatrix} \text{ with }p|c_i\right\}.
\]
In particular, there is a topological isomorphism \[
    N_{G, 1}^{\opp} \simeq \Z_p^2.
\]
In what follows, we fix a choice of such isomorphism.

Let $\frakn\subset \calO_K$ be an ideal such that both $\frakp$ and $\overline{\frakp}$ do not divide $\frakn$. We define the compact open subgroup $\Gamma_0(\frakn)\subset G(\widehat{\Z})$ by \[
    \Gamma_0(\frakn) := \left\{ \bfgamma = \begin{pmatrix} a & b \\ c & d\end{pmatrix} \in \GL_2(\widehat{\calO}_K) = G(\widehat{\Z}): (\bfgamma \mod \frakn)\in B_G(\widehat{\calO}_K/\frakn)\right\}.
\]
We write $\Gamma_0(\frakn) = \prod_{\ell: \text{prime}}'\Gamma_0(\frakn)_{\ell}$ where each $\Gamma_0(\frakn)_{\ell}$ is a compact open subgroup of $G(\Z_{\ell})$. 
Then, the Bianchi threefold of level $\Gamma_0(\frakn)$ is the locally symmetric space \[
    Y = Y_{\Gamma_0(\frakn)} = G(\Q)\backslash G(\A_{\Q})/\Gamma_{\infty}\Gamma_0(\frakn) = \GL_2(K)\backslash \GL_2(\A_K) /\Gamma_{\infty}\Gamma_0(\frakn),
\]
where $\Gamma_{\infty} := \SU_2(\C)\C^\times$.

Let $\Gamma := \Iw_G \prod_{\ell\neq p}' \Gamma_0(\frakn)_{\ell} \subset G(\widehat{\Z})$. In this paper, we will also consider the Bianchi threefold of level $\Gamma$, \emph{i.e.}, \[
    Y_{\Iw_G} = Y_{\Gamma} = G(\Q) \backslash G(\A_{\Q}) / \Gamma_{\infty}\Gamma.
\]
We shall often abuse the terminology and call $Y$ the Bianchi threefold of tame-level at $p$ and $Y_{\Iw_G}$ the Bianchi threefold of Iwahori level.

In what follows, we shall study the cohomology groups of $Y$ and $Y_{\Iw_G}$. To discuss them, we first introduce the coefficients of these cohomology groups. 

For any $k = (k_1, k_2)\in \Z^{2}$, we can view $k$ as a character of $T_G$ via \[
    k = (k_1, k_2): T_G \rightarrow \Res_{\Z}^{\calO_K}\bbG_{m/\calO_K}, \quad \begin{pmatrix}a & \\ & d\end{pmatrix} \mapsto a^{k_1}\sigma(a)^{k_2}, 
\]
where $\sigma\in \Gal(K/\Q)$ is the non-trivial element. We can moreover extend $k$ to $B_G$ by requiring $k(N_G) = \{1\}$. 

Given $k  = (k_1, k_2)\in \Z_{\geq 0}^2$, we consider \[
    V_k := \left\{ \phi: G \rightarrow \Res_{\Z}^{\calO_K}\bbA_{/\calO_K}: \phi(\bfgamma\bfbeta) = k(\bfbeta)\phi(\bfgamma)\,\,\forall (\bfgamma, \bfbeta)\in G \times B_G \right\}
\] and equip it with a right $G$-action by the left translation. For latter use, we consider the $\Z_p$-realisation $V_k(\Z_p)$ of $V_k$. In other words, \begin{align*}
    V_k(\Z_p) & = \left\{ \text{polynomial functions }\phi: G(\Z_p) \rightarrow \Res_{\Z}^{\calO_K}\bbA_{/\calO_K}(\Z_p): \begin{array}{c}
         \phi(\bfgamma\bfbeta) = k(\bfbeta)\phi(\bfgamma)  \\
         \forall (\bfgamma, \bfbeta)\in G(\Z_p) \times B_G(\Z_p) 
    \end{array}\right\} \\
    & = \left\{ \text{polynomial functions }\phi: \GL_2(\Z_p) \times \GL_2(\Z_p) \rightarrow \Z_p\times \Z_p: \begin{array}{c}
        \phi((\bfgamma_1, \bfgamma_2)(\bfbeta_1, \bfbeta_2)) = k_1(\bfbeta_1)k_2(\bfbeta_2)\phi(\bfgamma_1, \bfgamma_2)  \\
        \forall (\bfgamma_i, \bfbeta_i)\in \GL_2(\Z_p) \times B_{\GL_2}(\Z_p) 
    \end{array} \right\}.
\end{align*}
The following lemma allows one to go from $V_k$ to the more traditional coefficient system when considering the cohomology of Bianchi threefolds.

\begin{Lemma}\label{Lemma: Vk and polynomials}
    Given $k = (k_1, k_2)\in \Z_{\geq 0}^2$, let \[
        P_{k_i}(\Z_p) := \{ \phi\in \Z_p[X] : \deg\phi\leq k_i\}.
    \]
    Consider $P_{k_1}(\Z_p) \otimes_{\Z_p}P_{k_2}(\Z_p)$ and equip it with a right $G(\Z_p)$-action by \[
        \phi_1(X_1)\otimes \phi_2(X_2)|_{k} \left(\begin{pmatrix}a_1 & b_1\\ c_1 & d_1\end{pmatrix}, \begin{pmatrix}a_2 & b_2\\ c_2 & d_2\end{pmatrix}\right)  = (a_1+b_1X)^{k_1}\phi_1\left(\frac{c_1+d_1X_1}{a_1+b_1X_1}\right) \otimes (a_2+b_2X)^{k_2}\phi_2\left(\frac{c_2+d_2X_2}{a_2+b_2X_2}\right).
    \]
    Then, we have an isomorphism of $\Z_p$-modules \[
        V_k(\Z_p) \simeq P_{k_1}(\Z_p)\otimes_{\Z_p}P_{k_2}(\Z_p)
    \] that is compatible with the right action of $G(\Z_p)$.
\end{Lemma}
\begin{proof}
    From the description above, one sees an identification \[
        V_k(\Z_p) = V_{k_1}(\Z_p) \otimes_{\Z_p} V_{k_2}(\Z_p),
    \]
    where \[
        V_{k_i}(\Z_p) = \left\{ \text{polynomial functions }\phi: \GL_2(\Z_p) \rightarrow \Z_p: \phi(\bfgamma\bfbeta) = k_i(\bfbeta)\phi(\bfgamma)\,\,\forall (\bfgamma, \bfbeta)\in \GL_2(\Z_p) \times B_{\GL_2}(\Z_p)\right\}.
    \]
    Hence, it is enough to show that \[
       V_{k_i}(\Z_p)  \simeq P_{k_i}(\Z_p)
    \]
    for $i=1, 2$. However, such an identification is easily given by \[
        \phi \mapsto \left(X \mapsto \phi\begin{pmatrix}1\\ X & 1\end{pmatrix}\right).
    \]
    One also easily checks that this morphism is $G(\Z_p)$-equivariant. 
\end{proof}

\begin{Remark}
    We remark that our formalism is slightly different from the (rather) classical orientation in \cite{Stevens-RAMS, Eigen} and \cite{Williams-PhD, BW}. The aforementioned works consider $P_{k}(\Z_p)$ as a left $\GL_2(\Z_p)$-module via \[
        \begin{pmatrix}a & b\\ c & d\end{pmatrix} \phi(X) = (a+cX)^k\phi\left(\frac{b+dX}{a+cX}\right).
    \]
    In particular, the variable $X$ appears in a generic element $\begin{pmatrix}1 & X\\ & 1\end{pmatrix}$ of the unipotent radical $N_{\GL_2}$ and the left $\GL_2(\Z_p)$-action is given by the matrix multiplication on the right. We make such a modification in order to be more cohesive with the formalism used in \cite{Hansen-PhD} when discussing the overconvergent cohomology groups, which also allows one to tackle higher-rank groups in general.
\end{Remark}

We retain the assumption that $(k_1, k_2)\in \Z_{\geq 0}^{2}$ and consider \[
    V_{k}^{\vee}(\Z_p) := \Hom_{\Z_p}(V_{k}(\Z_p), \Z_p) \simeq \Hom_{\Z_p}(V_{k_1}(\Z_p), \Z_p)\otimes_{\Z_p}\Hom_{\Z_p}(V_{k_2}(\Z_p), \Z_p).
\]
Here, the last isomorphism follows from \cite[Chapter 2, \S 4, Proposition 4]{Bourbaki-Algebra1-3}.
We equip $V_k^{\vee}(\Z_p)$ with the induced left $G(\Z_p)$-action.

\begin{Remark}\label{Remark: alg. rep. is self-dual}
    We remark that there is a $\Q_p$-isomorphism \[
        V_k^{\vee}(\Q_p) := V_k^{\vee}(\Z_p)\otimes_{\Z_p}\Q_p \rightarrow V_k(\Q_p) := V_k(\Z_p)\otimes_{\Z_p}\Q_p, \quad \mu \mapsto \phi_{\mu},
    \]
    where the construction of $\phi_{\mu}$ is given as follows. Given $\left( \begin{pmatrix} 1 & \\ X_1 & 1 \end{pmatrix}, \begin{pmatrix} 1 & \\ X_2 & 1 \end{pmatrix}\right)\in N_{G}^{\opp}(\Q_p)$, the function \begin{align*}
        \phi_{X_1, X_2}: \left(\begin{pmatrix}1 \\ X_1' & 1\end{pmatrix}, \begin{pmatrix}1 & \\ X_2' & 1\end{pmatrix}\right) & \mapsto \left( \begin{pmatrix}1 & X_1'\end{pmatrix}\begin{pmatrix} & -1\\ 1 & \end{pmatrix}\begin{pmatrix}1\\ X_1\end{pmatrix}\right)^{k_1}\left(\begin{pmatrix}1 & X_2'\end{pmatrix}\begin{pmatrix} & -1\\ 1 & \end{pmatrix}\begin{pmatrix}1\\ X_2\end{pmatrix}\right)^{k_2}\footnotemark \\
        & = (X_1'-X_1)^{k_1}(X_2'-X_2)^{k_2}
    \end{align*}\footnotetext{Let us briefly explain the matrix $\begin{pmatrix} & -1\\ 1\end{pmatrix}$. Let $\calO_K^2$ be the free $\calO_K$-module of rank $2$. It admits a symplectic pairing given by the formula \[
        \left(\begin{pmatrix}a\\c\end{pmatrix}, \begin{pmatrix}b\\d\end{pmatrix}\right) \mapsto \begin{pmatrix}b & d\end{pmatrix} \begin{pmatrix} & -1\\ 1\end{pmatrix}\begin{pmatrix}a\\c\end{pmatrix} = ad-bc = \det\begin{pmatrix}a & b\\c& d\end{pmatrix}.
    \] Note that $\GL_{2/\calO_K}$ preserves this symplectic pairing up to a unit (given by the determinant). This pairing then induces the self-duality on $K^2$ and so it induces the self-duality on irreducible algebraic representations of $G$ as explained in this remark.}
    defines a function in $V_k(\Q_p)$. 
    Then, $\phi_{\mu}$ is defined as \[
        \phi_{\mu}: \left( \begin{pmatrix} 1 & \\ X_1 & 1 \end{pmatrix}, \begin{pmatrix} 1 & \\ X_2 & 1 \end{pmatrix}\right) \mapsto \mu(\phi_{X_1, X_2}).
    \]
    The proof of showing this morphism is an isomorphism is similar to \cite[Lemma IV.1.3]{Eigen}. We leave it to the readers to unwind the formulae and compare these two cases. 
\end{Remark}

Because of the left $G(\Z_p)$-action on $V_k^{\vee}(\Z_p)$ (as well as $V_k^{\vee}(\Q_p)$), it defines a local system on $Y$ and $Y_{\Iw_G}$. Consequently, one can consider the cohomology groups \[
    \begin{array}{cc}
        H^i(Y, V_{k}^{\vee}(\Z_p)), & H^i(Y_{\Iw_G}, V_k^{\vee}(\Z_p)), \\
        H_c^i(Y, V_{k}^{\vee}(\Z_p)), & H_c^i(Y_{\Iw_G}, V_k^{\vee}(\Z_p)). 
    \end{array}
\]
In fact, we have a commutative diagram \begin{equation}\label{eq: commutative diagram when changing levels}
    \begin{tikzcd}
        H^i(Y, V_k^{\vee}(\Z_p)) \arrow[r] & H^i(Y_{\Iw_G}, V_k^{\vee}(\Z_p))\\
        H_c^i(Y, V_{k}^{\vee}(\Z_p)) \arrow[r]\arrow[u] & H_c^i(Y_{\Iw_G}, V_k^{\vee}(\Z_p))\arrow[u],
    \end{tikzcd}
\end{equation}
where the horizontal arrows are induced from the natural projection $Y_{\Iw_G} \rightarrow Y$ while the vertical arrows are given by the natural morphisms from the compactly supported cohomology mapping into the usual (Betti) cohomology.

In what follows, we shall be studying the \textbf{\textit{parabolic cohomology groups}} \begin{align*}
    H_{\Par}^i(Y, V_k^{\vee}(\Z_p)) & := \image \left( H_c^i(Y, V_k^{\vee}(\Z_p)) \rightarrow H^i(Y, V_k^{\vee}(\Z_p)) \right),\\
    H_{\Par}^i(Y_{\Iw_G}, V_k^{\vee}(\Z_p)) & := \image \left( H_c^i(Y_{\Iw_G}, V_k^{\vee}(\Z_p)) \rightarrow H^i(Y_{\Iw_G}, V_k^{\vee}(\Z_p)) \right).
\end{align*}
For later discussion, we also consider the \textbf{\textit{Eisenstein cohomology groups}} \begin{align*}
    H_{\mathrm{Eis}}^i(Y, V_k^{\vee}(\Z_p)) &:= \coker\left( H_c^i(Y, V_k^{\vee}(\Z_p)) \rightarrow H^i(Y, V_k^{\vee}(\Z_p)) \right),\\
    H_{\mathrm{Eis}}^i(Y_{\Iw_G}, V_k^{\vee}(\Z_p)) & := \coker \left( H_c^i(Y_{\Iw_G}, V_k^{\vee}(\Z_p)) \rightarrow H^i(Y_{\Iw_G}, V_k^{\vee}(\Z_p)) \right).
\end{align*} Similar notations apply to the coefficient $V_k^{\vee}(\Q_p)$. By \eqref{eq: commutative diagram when changing levels}, we have natural morphisms \[
    H_{\Par}^i(Y, V_k^{\vee}(\Z_p)) \rightarrow H_{\Par}^i(Y_{\Iw_G}, V_k^{\vee}(\Z_p))\quad \text{ and }\quad H_{\Par}^i(Y, V_k^{\vee}(\Q_p)) \rightarrow H_{\Par}^i(Y_{\Iw_G}, V_k^{\vee}(\Q_p)).
\]
We remark that one can compute $H^i$, $H^i_c$, $H^i_{\Par}$, and $H_{\mathrm{Eis}}^i$ using the method discussed in \S \ref{subsection: overconvergent cohomology} below.

\subsection{Hecke operators}\label{subsection: Hecke operators for classical cohomology}
Let us begin with Hecke operators away from $\Gamma_0(\frakn)$. That is, for any prime number $\ell$ with $\ell\not\in \frakn$, we consider the spherical Hecke algebra \[
    \bbT_{\ell} := \Z_p[G(\Z_{\ell}) \backslash G(\Q_{\ell}) /G(\Z_{\ell})].
\]
For any $\bfdelta\in G(\Q_{\ell})$, we have a diagram of real manifolds \[
    \begin{tikzcd}
         & Y_{\bfdelta^{-1}\Gamma_0(\frakn)\bfdelta \cap \Gamma_0(\frakn)} \arrow[r, "\cong"', "\bfdelta"]\arrow[ld, "\pr_2"'] & Y_{\Gamma_0(\frakn) \cap \bfdelta \Gamma_0(\frakn) \bfdelta^{-1}} \arrow[rd, "\pr_1"]\\
         Y_{\Gamma_0(\frakn)} &&& Y_{\Gamma_0(\frakn)}
    \end{tikzcd},
\]
where $Y_{\bfdelta^{-1}\Gamma_0(\frakn)\bfdelta \cap \Gamma_0(\frakn)}$ and $Y_{\Gamma_0(\frakn) \cap \bfdelta \Gamma_0(\frakn) \bfdelta^{-1}}$ are defined similarly as in the previous section. The two projections $\pr_1$ and $\pr_2$ are natural projections. One observes that we have the following facts 
\begin{enumerate}
    \item[$\bullet$] The morphism $\pr_1: Y_{\Gamma_0(\frakn) \cap \bfdelta \Gamma_0(\frakn) \bfdelta^{-1}} \rightarrow Y_{\Gamma_0(\frakn)}$ is a covering space of degree $\#\left( \Gamma_0(\frakn) / (\Gamma_0(\frakn) \cap \bfdelta \Gamma_0(\frakn) \bfdelta^{-1})\right)$.\footnote{ In fact, similar statement holds also for $\pr_2: Y_{\bfdelta^{-1}\Gamma_0(\frakn)\bfdelta \cap \Gamma_0(\frakn)} \rightarrow Y_{\Gamma_0(\frakn)}$, but we do not need it. } 
    \item[$\bullet$] Given $k = (k_1, k_2)\in \Z_{\geq 0}^2$, the sheaf $\pr_1^{-1} V_k^{\vee}(\Z_p)$ (resp.,  $\pr_1^{-1}V_k^{\vee}(\Q_p)$) agrees with $V_k^{\vee}(\Z_p)$ (resp., $V_k^{\vee}(\Q_p)$) on $Y_{\Gamma_0(\frakn) \cap \bfdelta \Gamma_0(\frakn) \bfdelta^{-1}}$. A similar statement also holds for $\pr_2$. 
    \item[$\bullet$] Given $k = (k_1, k_2)\in \Z_{\geq 0}^2$, the sheaf $\bfdelta^{-1} V_k^{\vee}(\Z_p)$ (resp.,  $\bfdelta^{-1}V_k^{\vee}(\Q_p)$) agrees with $V_k^{\vee}(\Z_p)$ (resp., $V_k^{\vee}(\Q_p)$) on $Y_{\bfdelta^{-1}\Gamma_0(\frakn)\bfdelta \cap \Gamma_0(\frakn)} $.
\end{enumerate}
Combining these facts with the discussion in \S \ref{section: cohomological correspondences for topological spaces}, one defines the Hecke operator $T_{\bfdelta}$ as the composition of morphisms of cohomology groups \[
    \begin{tikzcd}
        T_{\bfdelta}:  H^i(Y, V_{k}^{\vee}(\Z_p)) \arrow[r, "\pr_2^{-1}"] & H^i(Y_{\bfdelta^{-1}\Gamma_0(\frakn)\bfdelta \cap \Gamma_0(\frakn)}, V_k^{\vee}(\Z_p)) \arrow[r, "\bfdelta^{-1}", "\cong"'] & H^i(Y_{\Gamma_0(\frakn) \cap \bfdelta \Gamma_0(\frakn) \bfdelta^{-1}}, V_k^{\vee}(\Z_p)) \arrow[ld, "\cong"', out=355,in=175] \\
        & H^i(Y, \pr_{1, *}\pr_1^{-1}V_{k}^{\vee}(\Z_p)) \arrow[r, "\tr"] & H^i(Y, V_k^{\vee}(\Z_p))
    \end{tikzcd}
\] 
and \[
    \begin{tikzcd}
        T_{\bfdelta}:  H_c^i(Y, V_{k}^{\vee}(\Z_p)) \arrow[r, "\pr_2^{-1}"] & H_c^i(Y_{\bfdelta^{-1}\Gamma_0(\frakn)\bfdelta \cap \Gamma_0(\frakn)}, V_k^{\vee}(\Z_p)) \arrow[r, "\bfdelta^{-1}", "\cong"'] & H_c^i(Y_{\Gamma_0(\frakn) \cap \bfdelta \Gamma_0(\frakn) \bfdelta^{-1}}, V_k^{\vee}(\Z_p)) \arrow[ld, "\cong"', out=355,in=175] \\
        & H_c^i(Y, \pr_{1, *}\pr_1^{-1}V_{k}^{\vee}(\Z_p)) \arrow[r, "\tr"] & H_c^i(Y, V_k^{\vee}(\Z_p)).
    \end{tikzcd}
\]

We remark that a similar formalism applies when replacing $V_k^{\vee}(\Z_p)$ with $V_k^{\vee}(\Q_p)$. Moreover, the formalism also applies when replacing $Y = Y_{\Gamma_0(\frakn)}$ with $Y_{\Iw_G}$. However, for computational convenience, we need the following explicit description for Hecke operators at $p$. 

We begin with a useful lemma.

\begin{Lemma}\label{Lemma: double coset decomposition; Hecke operators at p}
    \begin{enumerate}
        \item[(i)] We have the double coset decomposition \[
            \Gamma_0(p) \begin{pmatrix}1 & \\ & p\end{pmatrix} \Gamma_0(p) = \bigsqcup_{c=0}^{p-1} \begin{pmatrix}1 & \\ pc & p\end{pmatrix} \Gamma_0(p).
        \] 
        \item[(ii)] We have the double coset decomposition \[
            \GL_2(\Z_p)\begin{pmatrix}1 & \\ & p\end{pmatrix} \GL_2(\Z_p) = \left(\bigsqcup_{c=0}^{p-1} \begin{pmatrix}1 & \\ pc & p\end{pmatrix} \GL_2(\Z_p) \right) \sqcup \begin{pmatrix}p & \\ & 1\end{pmatrix}\GL_2(\Z_p).
        \]
    \end{enumerate}
\end{Lemma}
\begin{proof}
    The proof for (ii) is similar to (i) and so we only show (i) and leave (ii) to the readers. 

    Let $\Gamma' = \begin{pmatrix}1 & \\ & p\end{pmatrix} \Gamma_0(p) \begin{pmatrix}1 & \\ &  p\end{pmatrix}^{-1} \cap \Gamma_0(p)$. Observe there is an isomorphism \[
        \Gamma_0(p)/ \Gamma' \xrightarrow{\simeq }  \Gamma_0(p) \begin{pmatrix}1 & \\ & p\end{pmatrix}\Gamma_0(p)/\Gamma_0(p), \quad \bfgamma \Gamma'  \mapsto \bfgamma \begin{pmatrix}1 & \\ & p\end{pmatrix} \Gamma_0(p).
    \] Hence, we have to compute the coset representatives for $\Gamma_0(p)/\Gamma'$. 

    For any $\begin{pmatrix}a & b \\ c & d\end{pmatrix}\in \Gamma_0(p)$, we have \[
        \begin{pmatrix}1 & \\ & p\end{pmatrix}\begin{pmatrix}a & b \\ c & d\end{pmatrix} \begin{pmatrix}1 & \\ & p\end{pmatrix}^{-1} = \begin{pmatrix} a & p^{-1}b\\ pc & d\end{pmatrix}.
    \] Thus, \[
        \Gamma' = \left\{ \begin{pmatrix}a & b \\ c & d\end{pmatrix} \in \Gamma_0(p) : p^2|c \right\}.
    \] In other words, we have the coset representatives \[
        \Gamma_0(p)/\Gamma' = \left\{ \begin{pmatrix}1 & \\ pc & 1\end{pmatrix}: c=0, 1, ..., p-1\right\}.
    \]
    By multiplying $\begin{pmatrix}1 & \\ & p\end{pmatrix}$ from the right with the coset representatives for $\Gamma_0(p)/\Gamma'$, one obtains the desired coset decomposition. 
\end{proof}

Consider  \[
    \bfupsilon_{p} := \left( \begin{pmatrix}1 &  \\ & p\end{pmatrix}, \begin{pmatrix}1 &  \\ & p\end{pmatrix} \right) \in G(\Q_p) = \GL_2(\Q_p) \times \GL_2(\Q_p).
\] Then, the double coset $\Iw_G \bfupsilon_p \Iw_G$ has the following coset decomposition \[
    \Iw_{G} \bfupsilon_p \Iw_G = \bigsqcup_{\substack{c = (c_1, c_2)\\ 0\leq c_1, c_2\leq p-1}} \bfupsilon_{p, c} \Iw_G
\]
with \[
    \bfupsilon_{p, c} =  \left( \begin{pmatrix}1 &   \\ pc_1 & p\end{pmatrix}, \begin{pmatrix}1 &  \\ pc_2 & p\end{pmatrix} \right).
\] 
Note that each $\begin{pmatrix} 1 & \\ pc_i & p\end{pmatrix}$ acts on a polynomial $\phi_i(X)\in P_{k_i}(\Z_p)$ (resp.,  $P_{k_i}(\Q_p)$) via \[
    \left( \phi_i |_{k_i} \begin{pmatrix} 1 & \\ pc_i & p\end{pmatrix} \right) (X) = \phi_i\left(p+pc_i X\right).
\] 
Consequently, one obtains an induced right action by $\bfupsilon_{p,c}$ on $V_k(\Z_p)$ (resp., $V_k(\Q_p)$) and thus an induced left action on $V_k^{\vee}(\Z_p)$ (resp., $V_k^{\vee}(\Q_p)$). As explained in \cite[\S 2.1]{Hansen-PhD}, the left action of $\bfupsilon_{p,b}$ on $V_k^{\vee}(\Z_p)$ (resp., $V_k^{\vee}(\Q_p)$) and its natural action on $Y_{\Iw_G}$ then induces a left action on $H^i(Y_{\Iw_G}, V_k^{\vee}(\Z_p))$ (resp., $H^i(Y_{\Iw_G}, V_k^{\vee}(\Q_p))$). We then define the $U_p$-operator on $H^i(Y_{\Iw_G}, V_k^{\vee}(\Z_p))$ (resp., $H^i(Y_{\Iw_G}, V_k^{\vee}(\Q_p))$) by the formula \[
    U_p([\mu]) = \sum_{c} \bfupsilon_{p,c} * [\mu].
\]
Finally, one applies a similar construction to $H_c^i(Y_{\Iw_G}, V_k^{\vee}(\Z_p))$ (resp., $H_c^i(Y_{\Iw_G}, V_k^{\vee}(\Q_p))$) and so obtains the $U_p$-operator on $H_{\Par}^i(Y_{\Iw_G}, V_k^{\vee}(\Z_p))$ (resp., $H_{\Par}^i(Y_{\Iw_G}, V_k^{\vee}(\Q_p))$).


\begin{Remark}\label{Remark: other Hecke operators at p}
    Recall $p\calO_K = \frakp \overline{\frakp}$. Define \[
        \bfupsilon_{\frakp} = \left(\begin{pmatrix}1 & \\ & p\end{pmatrix}, \begin{pmatrix}1 & \\ & 1\end{pmatrix}\right), \quad \bfupsilon_{\overline{\frakp}} = \left(\begin{pmatrix}1 & \\ & 1\end{pmatrix},  \begin{pmatrix}1 & \\ & p\end{pmatrix} \right) \in G(\Q_p).
    \]
    Similar discussion as above applies to these two matrices and allows us to define operators $U_{\frakp}$ and $U_{\overline{\frakp}}$ respectively. More precisely, $U_{\frakp}$ (resp., $U_{\overline{\frakp}}$) is defined by the coset decomposition \[
        \Iw_{G} \bfupsilon_{\frakp} \Iw_{G} = \bigsqcup_{0\leq c \leq p-1} \bfupsilon_{\frakp, c} \Iw_{G} \quad (\text{resp., } \Iw_{G} \bfupsilon_{\overline{\frakp}} \Iw_{G} = \bigsqcup_{0\leq c \leq p-1} \bfupsilon_{\overline{\frakp}, c} \Iw_{G}),
    \] where \[
        \bfupsilon_{\frakp, c} = \left( \begin{pmatrix} 1 & \\ pc & p\end{pmatrix}, \begin{pmatrix}1 & \\ & 1\end{pmatrix}\right) \quad \left( \text{resp., } \bfupsilon_{\overline{\frakp}, c} = \left( \begin{pmatrix}1 & \\ & 1\end{pmatrix}, \begin{pmatrix} 1 & \\ pc & p\end{pmatrix}\right) \right).
    \] In particular, one sees that \[
        U_p = U_{\frakp} U_{\overline{\frakp}} = U_{\overline{\frakp}} U_{\frakp}.
    \]
\end{Remark}

\begin{Remark}\label{Remark: coset decomposition for Tp}
    Let $\bfupsilon_{\frakp}$ (resp., $\bfupsilon_{\overline{\frakp}}$) be as in Remark \ref{Remark: other Hecke operators at p}. One can also use the coset decomposition to define $T_{\bfupsilon_{\frakp}}$ (resp., $T_{\bfupsilon_{\overline{\frakp}}}$) on cohomology groups of $Y$. However, the coset decomposition is given by \[
        G(\Z_p) \bfupsilon_{\frakp} G(\Z_p) = \left( \bigsqcup_{0\leq c \leq p-1} \bfupsilon_{\frakp, c} G(\Z_p)\right) \sqcup \bfupsilon_{\frakp}^{*} G(\Z_p) \quad \left(\text{resp., } \left( \bigsqcup_{0\leq c \leq p-1} \bfupsilon_{\overline{\frakp}, c} G(\Z_p)\right) \sqcup \bfupsilon_{\overline{\frakp}}^{*} G(\Z_p) \right),
    \] where \[
        \bfupsilon_{\frakp}^{*} = \left( \begin{pmatrix} p & \\ & 1\end{pmatrix}, \begin{pmatrix}1 & \\ & 1\end{pmatrix}\right) \quad \left(\text{resp., } \bfupsilon_{\overline{\frakp}}^{*} = \left(\begin{pmatrix}1 & \\ & 1\end{pmatrix}, \begin{pmatrix} p & \\ & 1\end{pmatrix}\right)\right).
    \] In particular, if $[\mu]\in H^i(Y, V_k^{\vee}(\Z_p))$ and use the same symbol to denote its image (via the pullback map) in $H^i(Y_{\Iw_G}, V_k^{\vee}(\Z_p))$, we see that \[
        T_{\bfupsilon_{\frakp}} [\mu] = U_{\bfupsilon_{\frakp}} [\mu] + \bfupsilon_{\frakp}^{*} * [\mu].
    \] Similar for $\bfupsilon_{\overline{\frakp}}$ and for other cohomology group appearing above. 
\end{Remark}

Following the discussion above, we shall consider the following two Hecke algebras \[
    \bbT := \bigotimes_{\ell\not\in \frakn}\bbT_{\ell} \quad \text{ and }\bbT_{\Iw_G} := \left(\bigotimes_{\ell\not\in \frakn p}\bbT_{\ell}\right) \otimes_{\Z_p} \Z_p[U_{\frakp}, U_{\overline{\frakp}}].
\] Our next task is to understand the interaction between these two Hecke algebras. 

\begin{Remark}
    In this paper, we ignore the Hecke operators at $\frakn$ for simplicity. Note that, since our tame level is given by $\Gamma_0(\frakn)$, the only interesting Hecke operators at $\frakn$ are the $U_{\bfupsilon_{\frakl}}$-operators for those primes $\frakl$ dividing $\frakn$, where $\bfupsilon_{\frakl}$ is similarly defined as $\bfupsilon_{\frakp}$ and $\bfupsilon_{\overline{\frakp}}$. 
\end{Remark}

To simplify the notations, let us now denote by $V_k^{\vee}$ either $V_k^{\vee}(\Z_p)$ or $V_k^{\vee}(\Q_p)$. Suppose $[\mu]\in H_{\Par}^i(Y, V_k^{\vee})$ is an eigenclass for $\bbT$. We denote by $[\mu]$ again its image in $H_{\Par}^i(Y_{\Iw_G}, V_k^{\vee})$. One easily sees that $[\mu]$ is an eigenclass for $\bbT_{\ell}$ for $\ell \not\in \frakn p$. However, it needs not be an eigenclass for the $U_p$-operator. Instead, $[\mu]$ decomposes into a linear combination of $\bbT_{\Iw_G}$-eigenclasses. These eigenclasses are the \textit{\textbf{$p$-stabilisations}} of $[\mu]$. We summarise some essential properties for the $p$-stabilisations in the following result (see, for example, \cite[\S 3.3]{Palacios-BianchiFunctionalEquation} or \cite[\S 1.2.5]{Palacios-PhD}).

\begin{Proposition}\label{Proposition: p-stabilisation}
    Let $[\mu]\in H_{\Par}^i(Y, V_k^{\vee})$ be an eigenclass for $\bbT$. In what follows, we extend the coefficient of $V_k^{\vee}$ if necessary.  
    \begin{enumerate}
        \item[(i)] There are four $p$-stabilisations for $[\mu]$, denoted by $[\mu]^{(p)}_{(0,0)}$, $[\mu]^{(p)}_{(0,1)}$, $[\mu]^{(p)}_{(1,0)}$, and $[\mu]^{(p)}_{(1,1)}$, indexed by the Weyl group of $G$.\footnote{ The Weyl group of $G$ is isomorphic to $(\Z/2\Z)^{[K:\Q]}\simeq (\Z/2\Z)^2$. This explains the notation.}  
        \item[(ii)] Let $\lambda_{\frakp}([\mu])$ (resp., $\lambda_{\overline{\frakp}}([\mu])$) be the $T_{\bfupsilon_{\frakp}}$- (resp., $T_{\bfupsilon_{\overline{\frakp}}}$-)eigenvalue of $[\mu]$. Consider the Hecke polynomial \[
            P_{\frakp}(X) := X^2 - \lambda_{\frakp}([\mu])X + p^{k_1+1} \quad \left(\text{resp., } P_{\overline{\frakp}}(X) := X^2 - \lambda_{\overline{\frakp}}([\mu]) X + p^{k_2+1} \right)
        \] 
        Denote by $\alpha_{\frakp}^{(0)}$, $\alpha_{\frakp}^{(1)}$ (resp., $\alpha_{\overline{\frakp}}^{(0)}$, $\alpha_{\overline{\frakp}}^{(1)}$) the two roots of $P_{\frakp}$ (resp., $P_{\overline{\frakp}}$). Then, the $U_p$-eigenvalue for $[\mu]^{(p)}_{(i,j)}$ is given by $\alpha_{\frakp}^{(i)}\alpha_{\overline{\frakp}}^{(j)}$.
    \end{enumerate}
\end{Proposition}
\begin{proof}[Sketch of proof]
    We sketch the construction of $[\mu]_{(i,j)}^{(p)}$'s, which are similar to the case for modular forms.
    
    Denote by $\Iw_{\frakp} = \Gamma_0(p) \times \GL_2(\Z_p)$ and $\Iw_{\overline{\frakp}} = \GL_2(\Z_p) \times \Gamma_0(p)$. Let $Y_{\Iw_{\frakp}}$ and $Y_{\Iw_{\overline{\frakp}}}$ be the Bianchi threefolds over $Y$ with an extra level at $p$ given by $\Iw_{\frakp}$ and $\Iw_{\overline{\frakp}}$ respectively. Hence, we have a commutative diagram of Bianchi threefolds \[
        \begin{tikzcd}
            & Y_{\Iw_G}\arrow[rd]\arrow[ld]\\
            Y_{\Iw_{\frakp}}\arrow[rd] && Y_{\Iw_{\overline{\frakp}}}\arrow[ld]\\
            & Y
        \end{tikzcd}.
    \]The strategy is to first construct two classes $[\mu]_{0}^{(\frakp)}$ and $[\mu]_{1}^{(\frakp)}$ (resp., $[\mu]_{0}^{(\overline{\frakp})}$ and $[\mu]_{1}^{(\overline{\frakp})}$) at level $Y_{\Iw_{\frakp}}$ (resp., $Y_{\Iw_{\overline{\frakp}}}$), by using information of $\bfupsilon_{\frakp}$ (resp., $\bfupsilon_{\overline{\frakp}}$), whose $U_{\frakp}$-eigenvalues (resp., $U_{\overline{\frakp}}$-eigenvalues) are $\alpha_{\frakp}^{(0)}$ and $\alpha_{\frakp}^{(1)}$ (resp., $\alpha_{\overline{\frakp}}^{(0)}$ and $\alpha_{\overline{\frakp}}^{(1)}$) respectively. Then, we apply the same method again but use the information of $\bfupsilon_{\overline{\frakp}}$ (resp., $\bfupsilon_{\frakp}$). Although we have two approaches, they commute with each other. Thus, we only demonstrate one in what follows. 

    For $i=0, 1$, define \[
        [\mu]_i^{(\frakp)} := [\mu] - \alpha_{\frakp}^{(i), -1}\bfupsilon_{\frakp}^* * [\mu] \in H_{\Par}^i(Y_{\Iw_{\frakp}}, V_k^{\vee}).
    \] We claim that \[
        U_{\frakp} ([\mu]_i^{(\frakp)}) = \alpha_{\frakp}^{(i)} [\mu]_i^{(\frakp)}
    \]
    As mentioned in Remark \ref{Remark: coset decomposition for Tp}, we have \[
        U_{\frakp} = T_{\bfupsilon_{\frakp}}[\mu] - \bfupsilon_{\frakp}^* * [\mu]. 
    \] Suppose we have the identity \begin{equation}\label{eq: key identity for p-stabilisation}
        U_{\frakp} \left( \bfupsilon_{\frakp}^* * [\mu]\right) = p^{k_1+1}[\mu],
    \end{equation} 
    then \begin{align*}
        U_{\frakp} ([\mu]_i^{(\frakp)} )& = T_{\bfupsilon_{\frakp}} [\mu] - \bfupsilon_{\frakp}^* * [\mu] - U_{\frakp} \left( \bfupsilon_{\frakp}^* * [\mu]\right)\\
         & = \lambda_{\frakp}([\mu]) [\mu] - \bfupsilon_{\frakp}^* * [\mu] - p^{k_1+1}\alpha_{\frakp}^{(i), -1}[\mu]\\
         & = (\alpha_{\frakp}^{(i)} + \alpha_{\frakp}^{(1-i)})[\mu] - \bfupsilon_{\frakp}^* * [\mu] - \alpha_{\frakp}^{(1-i)}[\mu]\\
         & = \alpha_{\frakp}^{(i)} [\mu]_0^{(\frakp)},
    \end{align*} 
    where the penultimate equation follows from the relation of roots and coefficients of $P_{\frakp}(X)$.

    It remains to show \eqref{eq: key identity for p-stabilisation}. Using the identification in Lemma \ref{Lemma: Vk and polynomials}, we have to understand the action of $\begin{pmatrix}1 &  \\ pc & p\end{pmatrix}\begin{pmatrix}p & \\ & 1\end{pmatrix}$ on $P_{k_1}$. Note that \[
        \begin{pmatrix}1 &  \\ pc & p\end{pmatrix}\begin{pmatrix}p & \\ & 1\end{pmatrix} = \begin{pmatrix}p & \\ p^2c & p\end{pmatrix} = \begin{pmatrix}p & \\ & p\end{pmatrix} \begin{pmatrix} 1 & \\ pc & 1\end{pmatrix}.
    \] By the formula of the action on $P_{k_1}$ in Lemma \ref{Lemma: Vk and polynomials}, one sees that the scalar matrix $\begin{pmatrix}p & \\ & p\end{pmatrix}$ acts on a polynomial by multiplying $p^{k_1}$. On the other hand, $\begin{pmatrix} 1 & \\ pc & 1\end{pmatrix} \in \Gamma_0(p)$. We then conclude that \[
        U_{\frakp}(\bfupsilon_{\frakp}^* * [\mu]) = p^{k_1} \sum_{c=0}^{p-1} \left(\begin{pmatrix} 1 & \\ pc & 1\end{pmatrix}, \begin{pmatrix} 1 & \\ & 1\end{pmatrix}\right) * [\mu] = p^{k_1+1}[\mu].
    \]

    Finally, \[
        [\mu]_{(i,j)}^{(p)} := [\mu]_i^{(\frakp)} - \alpha_{\frakp}^{(j), -1}\bfupsilon_{\overline{\frakp}}^* * [\mu]_i^{(\frakp)}.
    \]
    Similar computations as above shows the $U_p$-eigenvalue for $[\mu]_{(i,j)}^{(p)}$ is $\alpha_{\frakp}^{(i)}\alpha_{\overline{\frakp}}^{(j)}$.  In summary, \[
        [\mu]_{(i,j)}^{(p)} = [\mu] - \alpha_{\frakp}^{(i), -1}\bfupsilon_{\frakp}^* *[\mu] - \alpha_{\overline{\frakp}}^{(j), -1} \bfupsilon_{\overline{\frakp}}^* * [\mu] + \alpha_{\frakp}^{(i), -1}\alpha_{\overline{\frakp}}^{(j), -1} \bfupsilon_{\frakp}^* \bfupsilon_{\overline{\frakp}}^* * [\mu]
    \] for $(i,j)\in (\Z/2\Z)^2$.
\end{proof}

\subsection{Bianchi cuspforms}
We close this section with some recapitulation of Bianchi modular forms. To this end, recall the fixed isomorphism $\C_p \simeq \C$ and so we can consider \[
    H_{\Par}^i(Y_{?}, V_k^{\vee}(\C)) := H_{\Par}^i(Y_{?}, V_k^{\vee}(\Q_p))\otimes_{\Q_p}\C \quad \text{ and }\quad H_{\mathrm{Eis}}^i(Y_{?}, V_k^{\vee}(\C)) := H_{\mathrm{Eis}}^i(Y_{?}, V_k^{\vee}(\Q_p))\otimes_{\Q_p}\C 
\] where $?\in \{\emptyset, \Iw_G\}$.

On the other hand, let $\pi = \pi_{\fin} \otimes \pi_{\infty}$ be a cuspidal automorphic representation for $G$. Recall that $\pi$ is of \textbf{\textit{weight $k=(k_1,k_2)$}} if $\pi_{\infty}$ is the principal series representation for $G(\C)$ that corresponds to the character \[
    \left(\begin{pmatrix}a_1 & b_1\\ & d_1\end{pmatrix}, \begin{pmatrix}a_2 & b_2\\ & d_2\end{pmatrix} \right) \mapsto \left( \frac{a_1}{|a_1|}\right)^{k_1+1}\left(\frac{a_2}{|a_2|}\right)^{k_2+1}.
\] 
Following \cite[Chapter 4]{Taylor-PhD}, the space of \textbf{\textit{tame-level Bianchi cuspforms of weight $k$}} (resp., \textbf{\textit{Iwahori-level Bianchi cuspforms of weight $k$}}) is given by \[
    S_{k}(Y) := \bigoplus_{\text{weight }k}\pi_{\fin}^{G(\Z_p)\Gamma_0(\frakn)}\quad \left(\text{resp., } S_k(Y_{\Iw_G}) := \bigoplus_{\text{weight }k}\pi_{\fin}^{\Iw_G\Gamma_0(\frakn)} \right).
\]
The following result is due to G. Harder. 

\begin{Theorem}[Eichler--Shimura--Harder isomorphism, $\text{\cite{Harder},\cite[\S 4.2]{Taylor-PhD}}$]\label{Theorem: ESH isomorphism}
    Let $k = (k_1, k_2)\in \Z^2_{\geq 0}$ and let $?\in \{\emptyset, \Iw_G\}$.
    \begin{enumerate}
        \item[(i)] The following two vanishing results hold: \[
            \begin{array}{cl}
                H_{\Par}^i(Y_?, V_k^{\vee}(\C)) = 0 & \text{ when $k_1\neq k_2$ or $i\not\in\{1,2\}$}   \\
                H_{\mathrm{Eis}}^0(Y_?, V_k^{\vee}(\C)) = 0 & \text{ unless }k_1=k_2=0 
            \end{array}.
        \] 
        \item[(ii)] When $k_1=k_2$, there are Hecke-equivariant isomorphisms \[
            H_{\Par}^1(Y_?, V_k^{\vee}(\C)) \simeq H_{\Par}^2(Y_?, V_k^{\vee}(\C)) \simeq S_k(Y_?).
        \]
        Here, the Hecke operators act on the parabolic cohomology groups as before and on $S_k(Y_?)$ via their actions on the automorphic representations. 
    \end{enumerate}
\end{Theorem}


\section{Bianchi eigenvarieties}\label{section: eigenvariety}
    The purpose of this section is to construct the (cuspidal) Bianchi eigenvariety that is relevant in our study. We begin with some discussion about $p$-adic analysis in \S \ref{subsection: p-adic analysis}. Results therein are supposed to be well-known to experts. In \S \ref{subsection: weight sapce}, we study the $p$-adic weight space. We specify two types of weights in our consideration: \emph{small weights} and \emph{affinoid weights}. The virtue of small weights allows us to realise \emph{ordinary families} in terms of \emph{finite-slope families} in \S \ref{subsection: control theorems}. In \S \ref{subsection: overconvergent cohomology}, we discuss the notion of \emph{overconvergent cohomology} in the Bianchi case; in \S \ref{subsection: Hecke and eigenvariety}, we discuss the Hecke actions on the overconvergent cohomology groups and use this information to construct (cuspidal) Bianchi eigenvarieties; and in \S \ref{subsection: parallel eigenvarieties} we briefly discuss the (cuspidal) parallel Bianchi eigenvariety over the parallel weight space. 

\subsection{Some \texorpdfstring{$p$}{p}-adic analysis}\label{subsection: p-adic analysis}
In this subsection, we recall some study of $p$-adic analysis. The materials presenting here should be well-known to experts. 

\begin{Definition}
    Let $r\in \Q_{>0}$ and let $n\in \Z_{\geq 0}$. A continuous function $f: \Z_p^n \rightarrow \Z_p$ is \textbf{$r$-analytic} if for every $(a_1, ..., a_n)\in \Z_p^n$, there exists a power series $f_{(a_1, ..., a_n)}\in \Z_p\llbrack T_1, ..., T_n\rrbrack$, converges on the ball of radius $p^{-r}$, such that \[
        f(x_1+a_1, ..., x_n+a_n) = f_{(a_1, ..., a_n)}(x_1, ..., x_n)
    \]
    for all $x_i\in p^{\lceil r \rceil}\Z_p$.
\end{Definition}

\begin{Theorem}[Amice]\label{Theorem: Amice}
    Given $r\in \Q_{>0}$ and $n\in \Z_{\geq 0}$, let $C^r(\Z_p^n, \Z_p)$ be the $\Z_p$-module of $r$-analytic functions from $\Z_p^n$ to $\Z_p$. For any $i = (i_1, ..., i_n)\in \Z_{\geq 0}^n$, define the function \[
        e_i^{(r)}: \Z_p^n \rightarrow \Z_p, \quad (x_1, ..., x_n) \mapsto \prod_{j=1}^n\lfloor p^{-r}j \rfloor! \begin{pmatrix} x_j\\ i_j\end{pmatrix}.
    \]
    Then, $\{e_i^{(r)}\}_{i\in \Z_{\geq 0}^{n}}$ provides an orthonormal basis for $C^r(\Z_p^n, \Z_p)$. In particular, we have an isomorphism \[
        C^r(\Z_p^n, \Z_p) \simeq \widehat{\bigoplus}_{i\in \Z_{\geq 0}^n}\Z_p e_i^{(r)}.
    \]
\end{Theorem}
\begin{proof}
    This is a reformulation of \cite[Chapter III, 1.3.8]{Lazard}, which is based on \cite[\S 10]{Amice}.
\end{proof}

\begin{Remark}\label{Remark: continuous functions, changing r}
    Obviously, if $r' \geq r$, then there exists a natural inclusion \[
        C^r(\Z_p^n, \Z_p) \hookrightarrow C^{r'}(\Z_p^n, \Z_p).
    \]
\end{Remark}

\begin{Remark}\label{Remark: orthonormal basis for continuous functions}
    Let $C(\Z_p^n, \Z_p)$ denote the space of continuous functions from $\Z_p^n$ to $\Z_p$. One, in fact, has the following isomorphism \[
        C(\Z_p^n, \Z_p) \simeq \widehat{\bigoplus}_{i\in \Z_{\geq 0}^n} \Z_p e_i,
    \]
    where $e_i = \prod_{j=1}^n(\substack{x_j\\ i_j})$. This is a consequence of \cite[Proposition 3.1.5]{JN-extended}.
\end{Remark}

Given $r\in \Q_{\geq 0}$ and $n\in \Z_{\geq 0}$,  we also consider \[
    C^{r^+}(\Z_p^n, \Z_p) := \left\{ f = \sum_{i\in \Z_{\geq 0}^n}c_i e_i: c_i\in \Q_p \text{ and }\left|c_i (\prod_{j=1}^n\lfloor p^{-r} j\rfloor!)^{-1}\right| \leq 1 \right\}
\]
From the definition, one sees that we have a natural inclusion \[
    C^r(\Z_p^n, \Z_p) \hookrightarrow C^{r^+}(\Z_p^n, \Z_p). 
\]

\begin{Lemma}\label{Lemma: str. theorem for Cr+}
    There exists an isomorphism \[
        C^{r^+}(\Z_p^n, \Z_p) \simeq \prod_{i \in \Z_{\geq 0}^n}\Z_p e_{i}^{(r)}
    \]
    and the natural inclusion $C^r(\Z_p^n, \Z_p) \hookrightarrow C^{r^+}(\Z_p^n, \Z_p)$ is given by the natural inclusion \[
         \widehat{\bigoplus}_{i \in \Z_{\geq 0}^n}\Z_p e_{i}^{(r)}  \hookrightarrow  \prod_{i \in \Z_{\geq 0}^n}\Z_p e_{i}^{(r)}.
    \]
\end{Lemma}
\begin{proof}
    By definition, for any $f = \sum_{i\in \Z_{\geq 0}^n}c_i e_i$ with $c_i\in \Q_p$, \[
        f = \sum_{i\in \Z_{\geq 0}^n} c_i \left( \prod_{j=1}^n \lfloor p^{-r}j\rfloor!\right)^{-1} e_i^{(r)} \in C^{r^+}(\Z_p^n, \Z_p) \Leftrightarrow \left|   c_i \left( \prod_{j=1}^n \lfloor p^{-r}j\rfloor!\right)^{-1} \right| \leq 1 \text{ for all }i\in \Z_{\geq 0}^n. 
    \]
    Hence, we see a natural isomorphism \[
        C^{r^+}(\Z_p^n, \Z_p) \simeq \prod_{i\in \Z_{\geq 0}^n}\Z_p, \quad f = \sum_{i\in \Z_{\geq 0}^n}c_i e_i \mapsto \left( c_i \left( \prod_{j=1}^n \lfloor p^{-r}j\rfloor!\right)^{-1} \right)_{i\in \Z_{\geq 0}^n}.
    \]
    This proves the lemma. 
\end{proof}

The following corollary justifies the notation for $C^{r^+}(\Z_p^n, \Z_p)$.

\begin{Corollary}
    There is a canonical isomorphism \[
        C^{r^+}(\Z_p^n, \Z_p) \simeq \varprojlim_{r'>r}C^{r'}(\Z_p^n, \Z_p). 
    \]
\end{Corollary}
\begin{proof}
    We first show that there is a natural inclusion \[
    C^{r^+}(\Z_p^n, \Z_p) \hookrightarrow C^{r'}(\Z_p^n, \Z_p)
\] for any $r'>r$.
By Theorem \ref{Theorem: Amice}, we have \[
    C^{r'}(\Z_p^n, \Z_p) \simeq \widehat{\bigoplus}_{i\in \Z_{\geq 0}^n} \Z_p e_i^{(r')}. 
\] Hence, it is enough to show that the image of $e_i^{(r)}$ in $C^{r'}(\Z_p^n, \Z_p)$ tends to $0$ as $i\rightarrow \infty$. Indeed, we have \[
    e_i^{(r)} = \frac{\prod_{j=1}^n \lfloor p^{-r} i_j \rfloor!}{ \prod_{j=1}^n \lfloor p^{-r'} i_j \rfloor!} e_i^{(r')}
\]
and \begin{align*}
    v_p\left(\frac{\prod_{j=1}^n \lfloor p^{-r} i_j \rfloor!}{ \prod_{j=1}^n \lfloor p^{-r'} i_j \rfloor!}\right) & = v_p\left(\prod_{j=1}^n \lfloor p^{-r} i_j \rfloor!\right) - v_p\left( \prod_{j=1}^n \lfloor p^{-r'} i_j \rfloor! \right)\\
    & = \sum_{j=1}^n \sum_{t >0}  \left\lfloor \frac{i_j}{p^{r+t}} \right\rfloor - \left\lfloor \frac{i_j}{p^{r'+t}} \right\rfloor
\end{align*}
which tends to $\infty$ as $i\rightarrow \infty$. 

On the other hand, fix $r_0>r$ and take $f\in \varprojlim_{r<r'} C^{r'}(\Z_p^n, \Z_p)$. We can write $f = \sum_{i\in \Z_{\geq 0}^n} c_i e_i^{(r_0)}$ with $c_i\in \Z_p$ and $c_i \rightarrow 0$ as $i\rightarrow \infty$. However, we can also express $f$ as follows: \begin{align*}
    f & = \sum_{i\in \Z_{\geq 0}^n} c_i \left( \prod_{j=1}^n \lfloor p^{-r_0}i_j\rfloor !\right) e_i,\\
    f & = \sum_{i\in \Z_{\geq 0}^n} c_i \cdot \frac{\prod_{j=1}^n \lfloor p^{-r_0}i_j\rfloor !}{\prod_{j=1}^n \lfloor p^{-r'}i_j\rfloor !} e_i^{(r')} \text{ for any }r<r'\leq r_0.
\end{align*}
By definition, we have to show that \[
    \left| c_i \left( \prod_{j=1}^n \lfloor p^{-r_0}i_j\rfloor !\right)\left( \prod_{j=1}^n \lfloor p^{-r}i_j\rfloor !\right)^{-1} \right| \leq 1
\] for all $i\in \Z_{\geq 0}^n$.

For any $\epsilon>0$, we know from the second expression of $f$ that \begin{align*}
    \left| c_i \cdot \frac{\prod_{j=1}^n \lfloor p^{-r_0}i_j\rfloor !}{\prod_{j=1}^n \lfloor p^{-(r+\epsilon)}i_j\rfloor !}  \right| \leq 1 \quad \text{ and } \quad \left| c_i \cdot \frac{\prod_{j=1}^n \lfloor p^{-r_0}i_j\rfloor !}{\prod_{j=1}^n \lfloor p^{-(r+\epsilon)}i_j\rfloor !} \right| \rightarrow 0 \text{ as }i \rightarrow \infty.
\end{align*}
For any $i\in \Z_{\geq 0}^d$, consider \begin{align*}
    \scalemath{0.8}{ \left| c_i \left( \prod_{j=1}^n \lfloor p^{-r_0}i_j\rfloor !\right)\left( \prod_{j=1}^n \lfloor p^{-r}i_j\rfloor !\right)^{-1} \right| } 
    & \scalemath{0.8}{= \left| c_i \left( \prod_{j=1}^n \lfloor p^{-r_0}i_j\rfloor !\right)\left( \prod_{j=1}^n \lfloor p^{-r}i_j\rfloor !\right)^{-1} - c_i \cdot \frac{\prod_{j=1}^n \lfloor p^{-r_0}i_j\rfloor !}{\prod_{j=1}^n \lfloor p^{-(r+\epsilon)}i_j\rfloor !} + c_i \cdot \frac{\prod_{j=1}^n \lfloor p^{-r_0}i_j\rfloor !}{\prod_{j=1}^n \lfloor p^{-(r+\epsilon)}i_j\rfloor !} \right| }  \\
    & \scalemath{0.8}{ \leq \left| c_i \left( \prod_{j=1}^n \lfloor p^{-r_0}i_j\rfloor !\right)\left( \prod_{j=1}^n \lfloor p^{-r}i_j\rfloor !\right)^{-1} - c_i \cdot \frac{\prod_{j=1}^n \lfloor p^{-r_0}i_j\rfloor !}{\prod_{j=1}^n \lfloor p^{-(r+\epsilon)}i_j\rfloor !} \right| + \left| c_i \cdot \frac{\prod_{j=1}^n \lfloor p^{-r_0}i_j\rfloor !}{\prod_{j=1}^n \lfloor p^{-(r+\epsilon)}i_j\rfloor !} \right| }
\end{align*}
By choosing $\epsilon$ small enough so that $\lfloor p^{-r} i_j \rfloor = \lfloor p^{-r-\epsilon} i_j \rfloor$ for all $j = 1, ..., n$, we can conclude that \[
    \left| c_i \left( \prod_{j=1}^n \lfloor p^{-r_0}i_j\rfloor !\right)\left( \prod_{j=1}^n \lfloor p^{-r}i_j\rfloor !\right)^{-1} \right| \leq \left| c_i \cdot \frac{\prod_{j=1}^n \lfloor p^{-r_0}i_j\rfloor !}{\prod_{j=1}^n \lfloor p^{-(r+\epsilon)}i_j\rfloor !} \right| \leq 1,
\] which proves the desired result.
\end{proof}

\subsection{The weight space}\label{subsection: weight sapce}
Let's start with the following well-known lemma.

\begin{Lemma}
    Let $\Alg_{(\Z_p, \Z_p)}$ be the category of complete sheafy $(\Z_p, \Z_p)$-algebras. 
    \begin{enumerate}
        \item[(i)] The functor \[
            \Alg_{(\Z_p, \Z_p)} \rightarrow \Sets, \quad (R, R^+) \mapsto \Hom_{\Groups}^{\cts}(\Z_p^\times \times \Z_p^\times, R^\times)
        \]
        is representable by the $(\Z_p, \Z_p)$-algebra $(\Z_p\llbrack \Z_p^\times \times \Z_p^\times \rrbrack, \Z_p\llbrack \Z_p^\times \times \Z_p^\times \rrbrack)$ 
        \item[(ii)] Let $(R, R^+)\in \Alg_{(\Z_p, \Z_p)}$ and let $\kappa\in \Hom_{\Groups}^{\cts}(\Z_p^\times \times \Z_p^\times, R^\times)$. Then, the image of $\kappa$ lies in $R^{\circ, \times}$. 
    \end{enumerate}
\end{Lemma}
\begin{proof}
    To show (i), one sees that, given $(R, R^+) \in \Alg_{(\Z_p, \Z_p)}$, there is a bijection \[
        \Hom_{\Groups}^{\cts}(\Z_p^\times \times \Z_p^\times, R) \simeq \Hom_{\Z_p}^{\cts}(\Z_p\llbrack \Z_p^\times \times \Z_p^\times \rrbrack, R).
    \] 
    More precisely, one extends the continuous characters on the left-hand side $\Z_p$-linearly to continuous $\Z_p$-morphisms on the right-hand side. 

    To show (ii), note that for $\kappa\in \Hom_{\Groups}^{\cts}(\Z_p^\times \times \Z_p^\times, R^\times)$, $\kappa = (\kappa_1, \kappa_2)$ where each $\kappa_i: \Z_p^\times \rightarrow R^\times$ is a continuous group homomorphism. It is then enough to show that $\kappa_i(1+ p\Z_p)\subset R^\circ$. That is, we have to show if $1+pa\in 1+p\Z_p$, then $\{\kappa_i(1+pa)^n\}_{n\in \Z_{\geq 0}}$ is bounded, \emph{i.e.}, for any open neighbourhood $U$ of $0$ in $R$, there exists an open neighbourhood $V$ of $0$ such that \[
        \{\kappa_i(1+pa)^n\}_{n\in \Z_{\geq 0}} V \subset U.
    \] 
    Without loss of generality, we may assume $U = I^m$ for some $m\in \Z_{>0}$, where $I$ is an ideal of generation of a ring of definition $R_0 \subset R$. However, one sees easily that $V = I^m$ does the job. 
\end{proof}

\begin{Definition}
    The \textbf{weight space} $\calW$ is defined to be \[
        \calW := \Spa(\Z_p\llbrack \Z_p^\times \times \Z_p^\times \rrbrack, \Z_p\llbrack \Z_p^\times \times \Z_p^\times \rrbrack)^{\rig},
    \]
    where the superscript `$\bullet^{\rig}$' stands for taking the generic fibre over $\Spa(\Q_p, \Z_p)$.
\end{Definition}

In what follows, we consider two types of weights following the convention in \cite{CHJ}.

\begin{Definition}\label{Definition: weights}
    \begin{enumerate}
        \item[(i)] A $\Z_p$-algebra $R$ is \textbf{small} if it is $p$-torsion free, reduced, and is finite over $\Z_p\llbrack T_1, ..., T_d\rrbrack$ for some $d\in \Z_{\geq 0}$. In particular, $R$ is equipped with a canonical adic profinite topology and it complete with respect to the $p$-adic topology. 
        \item[(ii)] A \textbf{small weight} is a pair $(R_{\calU}, \kappa_{\calU})$, where $R_{\calU}$ is a small $\Z_p$-algebra and $\kappa_{\calU}: \Z_p^\times \times \Z_p^\times \rightarrow R_{\calU}^\times$ is a continuous group homomorphism such that $\kappa_{\calU}(1+p, 1+p)-1$ is a topological nilpotent in $R_{\calU}$ with respect to the $p$-adic topology. 
        \item[(iii)] An \textbf{affinoid weight} is a pair $(R_{\calU}, \kappa_{\calU})$, where $R_{\calU}$ is a reduced affinoid algebra, topologically of finite type over $\Q_p$, and $\kappa_{\calU}: \Z_p^\times \times \Z_p^\times \rightarrow R_{\calU}^\times$ is a continuous group homomorphism. 
        \item[(iv)] By a \textbf{weight}, we mean either a small weight or an affinoid weight. 
    \end{enumerate}
\end{Definition}

\begin{Remark}\label{Remark: maps of weights to the weight space}
    Given a small weight (resp., an affinoid weight) $(R_{\calU}, \kappa_{\calU})$, we see that there is natural morphism \[
        \calU = \Spa(R_{\calU}, R_{\calU})^{\rig} \rightarrow \calW \quad (\text{resp., }\calU = \Spa(R_{\calU}, R_{\calU}^{\circ}) \rightarrow \calW)
    \]
    by the universal property of the weight space. Occasionally, we abuse the terminology and call $\calU$ a weight. 
\end{Remark}

\begin{Remark}\label{Remark: defining rU}
    \normalfont Given a weight $(R_{\calU}, \kappa_{\calU})$, $R_{\calU}[1/p]$ admits a structure of a uniform $\Q_p$-Banach algebra by letting $R_{\calU}^{\circ}$ be its unit ball and equipping it with the corresponding spectral norm, denoted by $|\cdot|_{\calU}$. Then, we define \[
        r_{\calU} := \min\left\{ r\in \Z_{\geq 0}: |\kappa_{\calU}(1+p, 1+p)|_{\calU} < p^{-\frac{1}{p^{r}(p-1)}}\right\}.
    \]
    See \cite[pp. 202]{CHJ}.
\end{Remark}

\subsection{Overconvergent cohomology groups}\label{subsection: overconvergent cohomology}
In this subsection, we recall and study the \emph{overconvergent cohomology groups} in the Bianchi case. Readers are also encouraged to consult \cite{Williams-PhD, BW}.

Let $R$ be either a small $\Z_p$-algebra or a reduced affinoid algebra over $\Q_p$. We follow the idea in \cite{Hansen-Iwasawa} and write \[
    \begin{array}{cc}
        A_R^{r, \circ} := C^r(\Z_p^2, \Z_p)\widehat{\otimes}_{\Z_p}R^{\circ}, & A_{R}^{r} := A_R^{r, \circ}[1/p]\\
        A_R^{r^+, \circ} := C^{r^+}(\Z_p^2, \Z_p)\widehat{\otimes}_{\Z_p} R^{\circ}, & A_{R}^{r^+} := A_R^{r^+, \circ}[1/p]
    \end{array}.\footnote{ If $R$ is a small $\Z_p$-algebra, $R = R^{\circ}$. On the other hand, if $R$ is a reduced affinoid algebra over $\Q_p$, then $R = R^{\circ}[1/p]$.}
\]
Obviously, one may view $A_{R}^{r, \circ}$ (resp., $A_R^r$) as a $R^{\circ}$-submodule (resp., $R^{\circ}[1/p]$-submodule) of the space of continuous functions from $\Z_p^2$ to $R^{\circ}$ (resp., $R^{\circ}[1/p]$).
Thanks to Theorem \ref{Theorem: Amice}, we have the following immediate consequence. 

\begin{Corollary}\label{Corollary: R-valued analytic functions}
    \begin{enumerate}
        \item[(i)] Let $R$ be a small $\Z_p$-algebra and let $\fraka$ be a ideal of definition. We may assume $p\in \fraka$. Then, \[
            A_R^{r, \circ} = \left\{ \sum_{i\in \Z_{\geq 0}^2} a_i e_i^{(r)}: a_i \in R\text{ and } a_i \rightarrow 0 \text{ $\fraka$-adically }\right\} \simeq \widehat{\bigoplus}_{i\in \Z_{\geq 0}^2}Re_i^{(r)}\quad \text{ and }\quad A_R^{r^+, \circ} \simeq \prod_{i\in \Z_{\geq 0}^2} Re_i^{(r)}.
        \] 
        \item[(ii)] Let $R$ be a reduced affinoid algebra, then \[
            A_R^r = \left\{ \sum_{i\in \Z_{\geq 0}^2} a_i e_i^{(r)}: a_i \in R\text{ and } a_i \rightarrow 0 \text{ $p$-adically }\right\} \simeq \widehat{\bigoplus}_{i\in \Z_{\geq 0}^2}R e_i^{(r)} \quad \text{ and }\quad A_R^{r^+} \simeq \left( \prod_{i\in \Z_{\geq 0}^2}R^{\circ}e_i^{(r)}\right)[1/p].
        \]
    \end{enumerate}
\end{Corollary}

Now, define \[
    \T_0 := (\Z_p^\times)^2 \times (p\Z_p)^2. 
\]
Elements in $\T_0$ are denoted by $\left( (\substack{a_1\\ c_1}), (\substack{a_2\\c_2})\right)$ with $a_i\in \Z_p^\times$ and $c_i\in p\Z_p$. Then, $\T_0$ is equipped with the following two actions: \begin{enumerate}
    \item[(I)] The action of $(\Z_p^\times)^2$ by \[
        (a_1', a_2')\left(\begin{pmatrix}a_1\\c_1\end{pmatrix}, \begin{pmatrix}a_2\\c_2\end{pmatrix}\right) = \left(\begin{pmatrix}a_1'a_1\\a_1'c_1\end{pmatrix}, \begin{pmatrix}a_2'a_2\\a_2'c_2\end{pmatrix}\right). 
    \] 
    \item[(II)] The left-action of $\Xi = \left(\begin{pmatrix} \Z_p^\times & \Z_p\\ p\Z_p & \Z_p\end{pmatrix} \times \begin{pmatrix} \Z_p^\times & \Z_p\\ p\Z_p & \Z_p\end{pmatrix}\right) \cap \left( \GL_2(\Q_p) \times \GL_2(\Q_p)\right)$ via the formula \[
        \left(\begin{pmatrix} a_1' & b_1'\\ c_1' & d_1'\end{pmatrix}, \begin{pmatrix} a_2' & b_2'\\ c_2' & d_2'\end{pmatrix}\right)\left( \begin{pmatrix}a_1\\c_1\end{pmatrix}, \begin{pmatrix}a_2\\c_2\end{pmatrix}\right) = \left( \begin{pmatrix}a_1'a_1+b_1'c_1\\c_1'a_1+d_1'c_1\end{pmatrix}, \begin{pmatrix}a_2'a_2+b_2'c_2\\c_2'a_2+d_2'c_2\end{pmatrix} \right).
    \]
    In particular, since $\Iw_G \subset \Xi$, $\T_0$ admits a left-action by $\Iw_G$. 
\end{enumerate}

We specify a subset \[
    \T_{00} := \left\{\left((\substack{1\\ c_1}), (\substack{1\\c_2})\right) \in \T_0\right\},
\]
which can be identified with $N_{G, 1}$ via \[
    \T_{00} \rightarrow N_{G,1}^{\opp}, \quad \left((\substack{1\\ c_1}), (\substack{1\\c_2})\right) \mapsto \left(\begin{pmatrix}1 & \\ c_1 & 1\end{pmatrix}, \begin{pmatrix}1 & \\ c_2 & 1\end{pmatrix}\right).
\]

Given a weight $(R_{\calU}, \kappa_{\calU})$ and $r>r_{\calU}$, we define \begin{align*}
    A_{\kappa_{\calU}}^{r, \circ} & := \left\{ f: \T_0 \rightarrow R_{\calU}: \begin{array}{l}
        f((\substack{a_1'a_1\\a_1'c_1}, (\substack{a_2'a_2\\a_2'c_2}))) = \kappa_{\calU, 1}(a_1')\kappa_{\calU, 2}(a_2')f((\substack{a_1\\c_1}), (\substack{a_2\\c_2})) \,\,\forall (a_1', a_2')\in (\Z_p^\times)^2 \text{ and }((\substack{a_1\\c_1}), (\substack{a_2\\c_2}))\in \T_0\\
        f|_{\T_{00}} \in A_{R_{\calU}}^{r, \circ} 
    \end{array} \right\}, \\
    A_{\kappa_{\calU}}^r & := A_{\kappa_{\calU}}^{r, \circ}[1/p], \\
    A_{\kappa_{\calU}}^{r^+, \circ} & := \left\{ f: \T_0 \rightarrow R_{\calU}: \begin{array}{l}
        f((\substack{a_1'a_1\\a_1'c_1}, (\substack{a_2'a_2\\a_2'c_2}))) = \kappa_{\calU, 1}(a_1')\kappa_{\calU, 2}(a_2')f((\substack{a_1\\c_1}), (\substack{a_2\\c_2})) \,\,\forall (a_1', a_2')\in (\Z_p^\times)^2 \text{ and }((\substack{a_1\\c_1}), (\substack{a_2\\c_2}))\in \T_0\\
        f|_{\T_{00}} \in A_{R_{\calU}}^{r^+, \circ} 
    \end{array} \right\}, \\
    A_{\kappa_{\calU}}^{r^+} & := A_{\kappa_{\calU}}^{r^+, \circ}[1/p].
\end{align*}
Here, we make use of the identifications $\T_{00} \simeq N_{G, 1}^{\opp}\simeq \Z_p^2$. That is, $f|_{\T_{00}}$ is viewed as a function on $\Z_p^2$ via \[
    f|_{\T_{00}}: \Z_p^2 \rightarrow R_{\calU}, \quad (X_1, X_2)\mapsto f\left(\begin{pmatrix}1 & \\ pX_1 & 1\end{pmatrix}, \begin{pmatrix}1 & \\ pX_2 & 1\end{pmatrix}\right).
\] Note that the left-action of $\Xi$ induces right-actions of $\Xi$ on $A_{\kappa_{\calU}}^{r, \circ}$, $A_{\kappa_{\calU}}^r$, $A_{\kappa_{\calU}}^{r^+, \circ}$, and $A_{\kappa_{\calU}}^{r^+}$.

\begin{Example}\label{Example: highest weight element in analytic functions}
    \normalfont Given a weight $(R_{\calU}, \kappa_{\calU})$ and $r> r_{\calU}$, one sees that the function \[
        \left(\begin{pmatrix}a_1\\c_1\end{pmatrix},\begin{pmatrix}a_2\\c_2\end{pmatrix}\right) \mapsto \kappa_{\calU, 1}(a_1)\kappa_{\calU, 2}(a_2)
    \]
    is an element in $A_{\kappa_{\calU}}^{r, \circ}$. We still denote this element by $\kappa_{\calU} = (\kappa_{\calU, 1}, \kappa_{\calU, 2})$. 

    Moreover, for any given $u=(u_1, u_2)\in (\Z_p^\times)^2$ and any given $z=(z_1, z_2)\in \Z_p^2$, the function \[
        f_{\kappa_{\calU}}^{(u, z)}:\left(\begin{pmatrix}a_1\\c_1\end{pmatrix},\begin{pmatrix}a_2\\c_2\end{pmatrix}\right) \mapsto \kappa_{\calU, 1}(u_1a_1+z_1c_1)\kappa_{\calU, 2}(u_2a_2+z_2c_2)
    \]
    is an element in $A_{\kappa_{\calU}}^{r, \circ}$. Indeed, one easily checks that this function satisfies the first condition of $A_{\kappa_{\calU}}^{r, \circ}$ while the second condition (\emph{i.e.}, its restriction to $\T_{00}$ lies in $A_{R_{\calU}}^{r, \circ}$) follows from \cite[Proposition 2.6]{CHJ}.
\end{Example}

We also consider the continuous duals of these spaces: \[
    \begin{array}{ll}
        D_{\kappa_{\calU}}^{r, \circ} := \Hom_{R_{\calU}^{\circ}}^{\cts}(A_{\kappa_{\calU}}^{r, \circ}, R_{\calU}^{\circ}) & D_{\kappa_{\calU}}^r := D_{\kappa_{\calU}}^{r, \circ}[1/p] \\
        D_{\kappa_{\calU}}^{r^+, \circ} := \Hom_{R_{\calU}^{\circ}}^{\cts}(A_{\kappa_{\calU}}^{r^+, \circ}, R_{\calU}^{\circ}) & D_{\kappa_{\calU}}^{r^+} := D_{\kappa_{\calU}}^{r^+, \circ}[1/p] 
    \end{array}.
\]
Consequently, the right-actions of $\Xi$ on $A_{\kappa_{\calU}}^{r, \circ}$, $A_{\kappa_{\calU}}^r$, $A_{\kappa_{\calU}}^{r^+, \circ}$, and $A_{\kappa_{\calU}}^{r^+}$ induce left-actions of $\Xi$ on $D_{\kappa_{\calU}}^{r, \circ}$, $D_{\kappa_{\calU}}^r$, $D_{\kappa_{\calU}}^{r^+, \circ}$, and $D_{\kappa_{\calU}}^{r^+}$ respectively.

\begin{Proposition}\label{Proposition: str. of Dkappa}
    Let $(R_{\calU}, \kappa_{\calU})$ be a weight and let $r>r_{\calU}$. For any $i\in \Z_{\geq 0}^2$, let $e_{i}^{(r), \vee}$ be the dual vector of $e_i^{(r)}$ introduced in Theorem \ref{Theorem: Amice}.
    \begin{enumerate}
        \item[(i)] If $(R_{\calU}, \kappa_{\calU})$ is a small weight, then $D_{\kappa_{\calU}}^{r, \circ} \simeq \prod_{i\in \Z_{\geq 0}^2}R_{\calU}e_i^{(r), \vee}$.
        \item[(ii)] If $(R_{\calU}, \kappa_{\calU})$ is an affinoid weight, then $D_{\kappa_{\calU}}^{r^+} \simeq \widehat{\bigoplus}_{i\in \Z_{\geq 0}^2}R_{\calU}e_{i}^{(r), \vee}$.
    \end{enumerate}
\end{Proposition}
\begin{proof}
    This is an immediate consequence of the definition and Corollary \ref{Corollary: R-valued analytic functions}.
\end{proof}

\begin{Remark}\label{Remark: filtration on distributions}
    If $(R_{\calU}, \kappa_{\calU})$ is a small weight and $r>1+r_{\calU}$, we can equip $D_{\kappa_{\calU}}^{r, \circ}$ with the following filtration inspired by \cite[\S 2.1]{Hansen-Iwasawa}: 

    Let $\fraka_{\calU}\subset R_{\calU}$ be an ideal of definition. We may assume $p\in \fraka_{\calU}$. The natural map $A_{\kappa_{\calU}}^{r-1, \circ} \rightarrow A_{\kappa_{\calU}}^{r, \circ}$ induces the map $D_{\kappa_{\calU}}^{r, \circ} \rightarrow D_{\kappa_{\calU}}^{r-1, \circ}$. Then, we define \[
        \Fil^j D_{\kappa_{\calU}}^{r, \circ} := \ker\left( D_{\kappa_{\calU}}^{r, \circ} \rightarrow D_{\kappa_{\calU}}^{r-1, \circ}/\fraka_{\calU}^jD_{\kappa_{\calU}}^{r-1, \circ}\right).
    \]
    Similarly as in \emph{loc. cit.}, one can show that the graded pieces $D_{\kappa_{\calU}, j}^{r, \circ} = D_{\kappa_{\calU}}^{r, \circ}/\Fil^jD_{\kappa_{\calU}}^{r, \circ}$ is a finite abelian group
    and \[
        D_{\kappa_{\calU}}^{r, \circ} = \varprojlim_{j} D_{\kappa_{\calU}, j}^{r, \circ}
    \]
    as a topological $R_{\calU}$-module. Moreover, the filtration $\Fil^{\bullet}D_{\kappa_{\calU}}^{r, \circ}$ is $\Xi$-stable. 
\end{Remark}

Since the spaces of distributions $D_{\kappa_{\calU}}^{r, \circ}$, $D_{\kappa_{\calU}}^r$,  $D_{\kappa}^{r^+, \circ}$, and $D_{\kappa_{\calU}}^{r^+}$ admits a left action by $\Iw_G$, they define local systems on $Y_{\Iw_G}$ (see, for example, \cite[\S 2.2]{Ash-Stevens}), which are still denoted by the same notation. In particular, we can consider the cohomology groups \[
    H^i(Y_{\Iw_G}, D_{\kappa_{\calU}}^{r, \circ}), \quad H^i(Y_{\Iw_G}, D_{\kappa_{\calU}}^{r}), \quad H^i(Y_{\Iw_G}, D_{\kappa_{\calU}}^{r^+, \circ}), \quad \text{ and }\quad H^i(Y_{\Iw_G}, D_{\kappa_{\calU}}^{r^+})
\] as well as the compactly supported cohomology groups \[
    H_c^i(Y_{\Iw_G}, D_{\kappa_{\calU}}^{r, \circ}), \quad H_c^i(Y_{\Iw_G}, D_{\kappa_{\calU}}^{r}), \quad H_c^i(Y_{\Iw_G}, D_{\kappa_{\calU}}^{r^+, \circ}), \quad \text{ and }\quad H_c^i(Y_{\Iw_G}, D_{\kappa_{\calU}}^{r^+}).
\]

We remark that these cohomology groups can be computed via the \emph{augmented Borel--Serre cochain complex} associated with the Borel--Serre compactification $X_{\Iw_G}^{\mathrm{BS}}$ of $Y_{\Iw_G}$ (\cite{Borel--Serre}). More precisely, after fixing a finite simplicial decomposition on $X_{\Iw_G}^{\mathrm{BS}}$, we define the cochain complex $C^{\bullet}(\Iw_G, D)$ associated with this simplicial decomposition with coefficients in $D \in \{D_{\kappa_{\calU}}^{r, \circ}, D_{\kappa_{\calU}}^{r}, D_{\kappa_{\calU}}^{r^+, \circ}, D_{\kappa_{\calU}}^{r^+}\}$. Then, by the discussion in \cite[\S 2.1]{Hansen-PhD}, one sees that there is a homotopy between $C^{\bullet}(\Iw_G, D)$ and the singular cochain complex of $Y_{\Iw_G}$ with coefficients in $D_{\kappa_{\calU}}^{r, \circ}$. Hence, we have $H^i(C^{\bullet}(\Iw_G, D)) \simeq H^i(Y_{\Iw_G}, D)$. 

Moreover, the finite simplicial decomposition also provides a finite simplicial decomposition on the boundary $\partial X_{\Iw_G}^{\mathrm{BS}} : = X_{\Iw_G}^{\mathrm{BS}} \smallsetminus Y_{\Iw_G}$. Hence, we can consider the \emph{boundary cochain complex} $C_{\partial}^{\bullet}(\Iw_G, D)$. The natural inclusion $\partial X_{\Iw_G}^{\mathrm{BS}} \hookrightarrow X_{\Iw}^{\mathrm{BS}}$ implies a natural morphism of cochain complexes \[
    \pi_{\partial}: C^{\bullet}(\Iw_G, D) \rightarrow C_{\partial}^{\bullet}(\Iw_G, D). 
\] By defining \[
    C_c^{\bullet}(\Iw_G, D) := \mathrm{Cone}(\pi_{\partial})[-1]
\]
and applying the strategy of proof in \cite[Proposition 3.5]{Barrera-Hilbert}, one concludes that $C_c^{\bullet}(\Iw_G, D)$ computes $H_c^i(Y_{\Iw_G}, D)$.

\begin{Remark}\label{Remark: total cochain complex is potentially ON-able}
    If $(R_{\calU}, \kappa_{\calU})$ is an affinoid weight. Since $C^{\bullet}(\Iw_G, D_{\kappa_{\calU}}^{r^+})$ is a finite cochain complex and $D_{\kappa_{\calU}}^{r^+} \simeq \widehat{\bigoplus}_{i\in \Z_{\geq 0}^2}R_{\calU}e_i^{(r), \vee}$ (Proposition \ref{Proposition: str. of Dkappa}), we see that the total space \[
    C_{\kappa_{\calU}, r^+}^{\mathrm{tot}} := \bigoplus_{j} C^j(\Iw_G, D_{\kappa_{\calU}}^{r^+})
\]
is a potentially ON-able Banach module over $R_{\calU}$ (\cite[pp. 70]{Buzzard-eigen}). By construction, one also sees that \[
    C_{c, \kappa_{\calU}, r^+}^{\mathrm{tot}} := \bigoplus_j C_c^j(\Iw_G, D_{\kappa_{\calU}}^{r^+})
\]
is also potentially ON-able over $R_{\calU}$.
\end{Remark}

\subsection{Hecke operators and the
Bianchi eigenvarieties}\label{subsection: Hecke and eigenvariety}
The purpose of this subsection is to construct one-dimensional Bianchi eigencurves. To this end, we start with a brief discussion about the Hecke algebra in our consideration, which is similar to the one discussed in \S \ref{subsection: Hecke operators for classical cohomology}. 

For a prime number $\ell$ such that $\ell \not\in (p)\frakn$, we consider the spherical Hecke algebra \[
    \bbT_{\ell} := \Z_p[G(\Z_{\ell})\backslash G(\Q_{\ell})/G(\Z_{\ell})].
\] 
For $D\in \{D_{\kappa_{\calU}}^{r, \circ}, D_{\kappa_{\calU}}^{r}, D_{\kappa_{\calU}}^{r^+, \circ}, D_{\kappa_{\calU}}^{r^+}\}$, the action of $\bbT_{\ell}$ on $H^i(Y_{\Iw_G}, D)$ and $H^i_c(Y_{\Iw_G}, D)$ is given as follows. For any $\bfdelta\in G(\Q_{\ell})$, we apply a similar strategy as before and obtain two morphisms \[
    \begin{tikzcd}
        T_{\bfdelta}:  H^i(Y_{\Iw_G}, D) \arrow[r, "\pr_2^{-1}"] & H^i(Y_{\bfdelta^{-1}\Gamma\bfdelta \cap \Gamma}, D) \arrow[r, "\bfdelta^{-1}", "\cong"'] & H^i(Y_{\Gamma \cap \bfdelta \Gamma \bfdelta^{-1}}, D) \arrow[ld, "\cong"', out=355,in=175] \\
        & H^i(Y_{\Iw_G}, \pr_{1, *}\pr_1^{-1}D) \arrow[r, "\tr"] & H^i(Y_{\Iw_G}, D)
    \end{tikzcd}
\] 
and \[
    \begin{tikzcd}
        T_{\bfdelta}:  H_c^i(Y_{\Iw_G}, D) \arrow[r, "\pr_2^{-1}"] & H_c^i(Y_{\bfdelta^{-1}\Gamma\bfdelta \cap \Gamma_0(\frakn)\Iw_G}, D) \arrow[r, "\bfdelta^{-1}", "\cong"'] & H_c^i(Y_{\Gamma \cap \bfdelta \Gamma \bfdelta^{-1}}, D) \arrow[ld, "\cong"', out=355,in=175] \\
        & H_c^i(Y_{\Iw_G}, \pr_{1, *}\pr_1^{-1}D) \arrow[r, "\tr"] & H_c^i(Y_{\Iw_G}, D)
    \end{tikzcd}.
\]
Combining everything together, one obtains the $T_{\bfdelta}$-operator \[
    T_{\bfdelta}: H_{\Par}^i(Y_{\Iw_G}, D) \rightarrow H^i_{\Par}(Y_{\Iw_G}, D).
\]

For the Hecke algebra at $p$, recall the matrix \[
    \bfupsilon_{p} := \left( \begin{pmatrix}1 &  \\ & p\end{pmatrix}, \begin{pmatrix}1 &  \\ & p\end{pmatrix} \right) \in G(\Q_p) = \GL_2(\Q_p) \times \GL_2(\Q_p)
\] and the double-coset decomposition \[
    \Iw_{G} \bfupsilon_p \Iw_G = \bigsqcup_{\substack{b = (b_1, b_2)\\ 0\leq c_1, c_2\leq p-1}} \bfupsilon_{p, c} \Iw_G
\]
with \[
    \bfupsilon_{p, c} =  \left( \begin{pmatrix}1 &   \\ pc_1 & p\end{pmatrix}, \begin{pmatrix}1 &  \\ pc_2 & p\end{pmatrix} \right).
\] 
Observe that each $\bfupsilon_{p, c}\in \Xi$ and so each $\bfupsilon_{p,c}$ acts on $\T_0$. 
Consequently, we have the action of $\bfupsilon_{p}$ on $D\in \{D_{\kappa_{\calU}}^{r, \circ}, D_{\kappa_{\calU}}^{r}, D_{\kappa_{\calU}}^{r^+, \circ}, D_{\kappa_{\calU}}^{r^+}\}$ given by \begin{equation}\label{eq: Up on distributions}
    \bfupsilon_p * \mu = \sum_{b} \bfupsilon_{p,b} \mu.
\end{equation}
On the other hand, the matrices $\bfupsilon_{p,c}$'s act on $Y_{\Iw_G}$ naturally. Combining these two actions together, one obtains the operator $U_p$ acting on $C^i(\Iw_G, D)$ and $C_c^i(\Iw_G, D)$ by the formula \[U_p (\sigma) = \sum_{\substack{c=(c_1, c_2)\\ 0\leq c_1,c_2\leq p-1}} \bfupsilon_{p,c}\sigma. 
\]

\begin{Lemma}\label{Lemma: Up is compact}
    Let $(R_{\calU}, \kappa_{\calU})$ be an affinoid weight and let $r\in \Q_{>0}$ such that $r>r_{\calU}$. Then, $U_p$ is a compact operator on $C_{\kappa_{\calU}, r^+}^{\mathrm{tot}}$ and $C_{c, \kappa_{\calU}, r^+}^{\mathrm{tot}}$.
\end{Lemma}
\begin{proof}
    It is enough to show that the action of $\bfupsilon_p$ on $D_{\kappa_{\calU}}^{r^+}$ defined in \eqref{eq: Up on distributions} is a compact operator. From the discussion in \cite[\S 2.2]{Hansen-PhD}, we know that the action of $\bfupsilon_p$ on $D_{\kappa_{\calU}}^r$ factors as \[
        D_{\kappa_{\calU}}^r \rightarrow D_{\kappa_{\calU}}^{r+1} \hookrightarrow D_{\kappa_{\calU}}^r,
    \]
    where the last arrow is the natural inclusion. Consequently, the action of $\bfupsilon_p$ on $D_{\kappa_{\calU}}^{r^+}$ factors as \[
        D_{\kappa_{\calU}}^{r^+} \rightarrow D_{\kappa_{\calU}}^{(r+1)^+} \hookrightarrow D_{\kappa_{\calU}}^{r^+}.
    \]
    However, since $D_{\kappa_{\calU}}^{r^+}$ is potentially ON-able and the natural inclusion $D_{\kappa_{\calU}}^{(r+1)^+} \hookrightarrow D_{\kappa_{\calU}}^{r^+}$ is compact, we conclude the desired result. (See also \cite[Corollary 3.3.10]{JN-extended}.) 
\end{proof}

\begin{Remark}\label{Remark: Hansen's framework works for Dr+}
    Suppose $(R_{\calU}, \kappa_{\calU})$ is an affinoid weight and $r>r_{\calU}$. Since $U_p$ acts on $C_{\kappa_{\calU}, r^+}^{\mathrm{tot}}$ compactly, we have a well-defined Fredholm determinant \[
        F_{\kappa_{\calU}}^r := \det\left( 1 - TU_p| C_{\kappa_{\calU}, r^+}^{\mathrm{tot}} \right) \in R_{\calU}\llbrack T \rrbrack.
    \]
    Then, by adapting the same arguments in \cite[\S 3.1]{Hansen-PhD}, we see that \begin{enumerate}
        \item[(i)] The Fredhol determinant $F_{\kappa_{\calU}}^r$ is independent to the choice of $r$. 
        \item[(ii)] If $F_{\kappa_{\calU}}^r$ admits a slope-$\leq h$ decomposition and $(R_{\calU'}, \kappa_{\calU'})$ is an affinoid weight such that we have an open immersion $\calU' \hookrightarrow \calU$ on the associated rigid analytic spaces, then \[
            C^{\bullet}(\Iw_G, D_{\kappa_{\calU}}^{r^+})^{\leq h} \otimes_{R_{\calU}} R_{\calU'} \simeq C^{\bullet}(\Iw_G, R_{\calU'})^{\leq h} \quad \text{ and }\quad H^i(Y_{\Iw_G}, D_{\kappa_{\calU}}^{r^+})^{\leq h} \otimes_{R_{\calU}}R_{\calU'} \simeq H^i(Y_{\Iw_G}, D_{\kappa_{\calU'}}^{r^+})^{\leq h}.
        \]
    \end{enumerate}
    Similar results also apply when considering the compactly supported cohomology groups. We denote by $F_{c, \kappa_{\calU}}^r$ the Fredholm determinant associated with $C_{c, \kappa_{\calU}, r^+}^{\mathrm{tot}}$. 
\end{Remark}

\begin{Remark}\label{Remark: interior cohomology groups are Hecke-stable}
    For $D\in \{D_{\kappa_{\calU}}^{r, \circ}, D_{\kappa_{\calU}}^r, D_{\kappa_{\calU}}^{r^+, \circ}, D_{\kappa_{\calU}}^{r^+}\}$, it is easy to see that the natural map \[
        C_c^{\bullet}(\Iw_G, D) \rightarrow C^{\bullet}(\Iw_G, D)
    \] is Hecke-equivariant. 
    Consequently, the natural map \[
        H^i_c(Y_{\Iw_G}, D) \rightarrow H^i(Y_{\Iw_G}, D)
    \] is Hecke-equivariant for all degree $i$. In particular, if $(R_{\calU}, \kappa_{\calU})$ is an affinoid weight, then we have a Hecke-equivariant map \[
        H_c^i(Y_{\Iw_G}, D_{\kappa_{\calU}}^{r^+})^{\leq h} \rightarrow H^i(Y_{\Iw_G}, D_{\kappa_{\calU}}^{r^+})^{\leq h}
    \]
    whenever there is a slope-$\leq h$ decomposition on $F_{c, \kappa_{\calU}}^r$. In such a case, we define \[
        H^i_{\mathrm{par}}(Y_{\Iw_G}, D_{\kappa_{\calU}}^{r^+})^{\leq h} := \image \left( H_c^i(Y_{\Iw_G}, D_{\kappa_{\calU}}^{r^+})^{\leq h} \rightarrow H^i(Y_{\Iw_G}, D_{\kappa_{\calU}}^{r^+})^{\leq h} \right)
    \]
    and call it the \emph{slope-$\leq h$ part} of the overconvergent parabolic cohomology group.
\end{Remark}

Thanks to Remark \ref{Remark: Hansen's framework works for Dr+}, one sees that $F_{c, \kappa_{\calU}}^r$'s glue to a power series $F_c^{\dagger}\in \scrO_{\calW}(\calW)\{\{ T \}\}$ such that $F_c^{\dagger}|_{\calU} = F_{c, \kappa_{\calU}}^r$ for any affinoid weight $(R_{\calU}, \kappa_{\calU})$ whose induced morphism $\calU \rightarrow \calW$ is an open immersion. By setting $\calZ$ to be the spectral variety associated with $F_c^{\dagger}$, we have the eigenvariety datum (see \cite[Definition 4.2.1]{Hansen-PhD}; see also the discussion before \cite[Lemma 5.1]{Buzzard-eigen}) \[
    (\calW, \calZ, \scrH_{\mathrm{par}}^{\mathrm{tot}}, \bbT_{\Iw_G}, \psi),
\]
where \begin{itemize}
    \item $\scrH_{\mathrm{par}}^{\mathrm{tot}}$ is the coherent sheaf on $\calZ$, assigning to each slope-adapted affinoid $\calZ_{\calU, h} \subset \calZ$ the module \[
        \scrH_{\mathrm{par}}^{\mathrm{tot}}(\calZ_{\calU, h}) = \bigoplus_{i}H_{\mathrm{par}}^i(Y_{\Iw_G}, D_{\kappa_{\calU}}^{r})^{\leq h};\footnote{ Here, we apply \cite[Theorem 3.3.1]{Hansen-PhD} to obtain the skope-$\leq h$ decomposition of $H^i(Y_{\Iw_G}, D_{\kappa_{\calU}}^r)$ and $H_c^i(Y_{\Iw_G}, D_{\kappa_{\calU}}^r)$.}
    \]
    \item $\psi: \bbT_{\Iw_G} \rightarrow \End_{\calZ}\left(\scrH_{\mathrm{par}}^{\mathrm{tot}}\right)$ is the morphism induced by the action of Hecke-operators on each $H_{\mathrm{par}}^i(Y_{\Iw_G}, D_{\kappa_{\calU}}^{r})^{\leq h}$.
\end{itemize}
 
Define the sheaf of $\scrO_{\calZ}$-algebras $\scrT$ by \[
    \scrT(\calZ_{\calU, h}) := \text{the reduced $\scrO_{\calZ_{\calU, h}}(\calZ_{\calU, h})$-algebra generated by $\image(\psi: \bbT_{\Iw_G}\otimes_{\Z_p}R_{\calU} \rightarrow \End_{\calZ_{\calU, h}}(\scrH_{\Par}^{\tot}))$}
\] 
for any slope-adapted affinoid $\calZ_{\calU, h} \subset \calZ$. Note that $\scrT$ is a coherent $\scrO_{\calZ}$-module. Then, the \emph{cuspidal Bianchi eigenvarity} is defined to be the relative adic spectrum \[
    \calE := \Spa_{\calZ}(\scrT, \scrT^{\circ}),
\]
where the existence of $\scrT^{\circ}$ is provided by \cite[Lemma A. 3]{JN-extended}. The natural map given by the composition \[
    \wt: \calE \rightarrow \calZ \rightarrow \calW
\]
is called the \emph{weight map}.

Although the construction of $\calE$ uses the information of the parabolic cohomology in each degree, there is an unpleasant drawback due to the following result. 

\begin{Lemma}[$\text{\cite[Lemma 4.2]{BW}}$]\label{Lemma: cannot vary 2-dim. family in H1}
    Let $(R_{\calU}, \kappa_{\calU})$ be an affinoid weight such that the associated rigid analytic space $\calU \hookrightarrow \calW$ is an affinoid open subset. Suppose $r>r_{\calU}$ and $(\calU, h)$ is slope-adapted. Then, $H_c^1(Y_{\Iw_G}, D_{\kappa_{\calU}}^r)^{\leq h}$ vanishes. 
\end{Lemma}

In what follows, we would like to establish a relation between a pairing (between $H^1$ and $H^2$) and the local geometry of the eigenvariety. The lemma above implies that such a strategy cannot work on the whole $\calE$ and so we have to restrict ourselves to certain $1$-dimensional families. In particular, we will often make the following assumption (see also \cite[\S 4.2]{BW}).

\begin{Assumption}\label{Assumption: vary in 1-dimensional family}
    Let $f$ be a classical cuspidal Bianchi eigenform of weight $k$. Let $\calV \subset \calW$ be an affinoid open subset containing $k$. Suppose $f$ defines a point $x_f$ in $\calE_{\calV, h}$ --the preimage of $\calZ_{\calV, h}$ in $\calE$. We assume $x_f$ varies in a family over a curve $\calU \subset \calV$ such that $\calU$ is smooth at $k$. 
\end{Assumption}

\subsection{Parallel eigenvarieties}\label{subsection: parallel eigenvarieties}
The purpose of this subsection is to show that we have a handful of sources of classical points living in the cuspidal parallel eigenvariety that satisfy Assumption \ref{Assumption: vary in 1-dimensional family}. To this end, let \[
    \calW_{\circ} \hookrightarrow \calW.
\]
be the one-dimensional closed subspace defined by the condition \[
    \Hom_{\Groups}^{\cts}(\Z_p^\times, R^\times) \rightarrow \Hom_{\Groups}^{\cts}(\Z_p^\times \times \Z_p^\times, R^\times), \quad \kappa \mapsto \left( (a_1, a_2)\mapsto \kappa(a_1)\kappa(a_2)\right)
\] for any complete sheafy $(\Z_p, \Z_p)$-algebra $(R, R^+)$.

We can now again apply Remark \ref{Remark: Hansen's framework works for Dr+} and get a Fredholm power series $F_{c, \circ}^{\dagger}\in \scrO_{\calW_{\circ}}(\calW_{\circ})\{\{ T \}\}$ and let $\calZ_{\circ}$ be the associated spectral variety. We follow the strategy in \cite{BW} and consider the sheaf \[
    \scrH_{\Par}^1 : \calZ_{\calU, h} \mapsto H_{\Par}^1(Y_{\Iw_G}, D_{\kappa_{\calU}}^r)^{\leq h}
\] for any slope-adapted affinoid open subset $\calZ_{\calU, h}\subset \calZ_{\circ}$. By similarly redefine $\psi$, one obtains the eigenvariety datum \[
    (\calW_{\circ}, \calZ_{\circ}, \scrH_{\mathrm{par}}^1, \psi).
\]
Then, we can then consider the associated eigenvariety $\calE_{\circ}$. More precisely, define the sheaf of $\scrO_{\calZ_{\circ}}$-algebras $\scrT_{\circ}$ by \[
    \scrT_{\circ}(\calZ_{\calU, h}) := \text{ the reduced $\scrO_{\calZ_{\calU, h}}(\calZ_{\calU, h})$-algebra generated by }\image \left( \psi: \bbT_{\Iw_G}\otimes_{\Z_p}R_{\calU} \rightarrow \End_{\calZ_{\calU, h}}(\scrH_{\mathrm{par}}^1)\right)
\]
for any slope-adapted affinoid $\calZ_{\calU, h}\subset \calZ_{\circ}$. Note that $\scrT_{\circ}$ is a coheret $\scrO_{\calZ_{\circ}}$-module. Then, the \emph{cuspidal parallel Bianchi eigenvariety} $\calE_{\circ}$ is defined (similarly) to be the relative adic spectrum \[
    \calE_{\circ} := \Spa_{\calZ_{\circ}}(\scrT_{\circ}, \scrT^{\circ}_{\circ}).
\]
We again have the weight map given by the composition \[
    \wt: \calE_{\circ} \rightarrow \calZ_{\circ} \rightarrow \calW_{\circ}.
\]

\begin{Proposition}\label{Proposition: parallel eigenvariety mapping into the whole eigenvariety}
    \begin{enumerate}
        \item[(i)] The cuspidal parallel Bianchi eigenvariety $\calE_{\circ}$ is $1$-dimensional and contains a very Zariski-dense set of classical points.
        \item[(ii)] There is a closed immersion $\calE_{\circ} \hookrightarrow \calE$. 
    \end{enumerate}
\end{Proposition}
\begin{proof}
    Denote by $\calE^{\mathrm{BSW}}$ (resp., $\calE^{\mathrm{BSW}}_{\circ}$) the Bianchi eigenvariety (resp., parallel Bianchi eigenvariety) discussed in \cite[\S 5.1]{BW}, constructed by considering Hecke eigensystems in $\bigoplus_{i=0}^3 H_c^i(Y_{\Iw_G}, D_{\kappa_{\calU}}^r)^{\leq h}$ (resp., $H^1_c(Y_{\Iw_G}, D_{\kappa_{\calU}}^r)^{\leq h}$). We denote by $\scrT^{\mathrm{BSW}}$ (resp., $\scrT^{\mathrm{BSW}}_{\circ}$) the corresponding sheaf of $\scrO_{\scrZ}$-algebras (resp., $\scrO_{\calZ_{\circ}}$-algebras) and so $\calE^{\mathrm{BSW}} = \Spa_{\calZ}(\scrT^{\mathrm{BSW}}, \scrT^{\mathrm{BSW}, \circ})$ (resp., $\calE^{\mathrm{BSW}}_{\circ} = \Spa_{\calZ_{\circ}}(\scrT^{\mathrm{BSW}}_{\circ}, \scrT^{\mathrm{BSW}, \circ}_{\circ})$). 
    
    By construction, $H_{\Par}^i$'s receive Hecke-equivariant surjections from $H_c^i$'s. Hence, we have surjections of sheaves \[
        \scrT^{\mathrm{BSW}} \twoheadrightarrow \scrT \quad \text{ and }\quad \scrT^{\mathrm{BSW}}_{\circ} \twoheadrightarrow \scrT_{\circ},
    \] 
    which yield (Zariski) closed embeddings \[
        \calE \hookrightarrow \calE^{\mathrm{BSW}} \quad \text{ and }\calE_{\circ} \hookrightarrow \calE^{\mathrm{BSW}}_{\circ}.
    \]

    Since classical weights are very Zariski dense in $\calW_{\circ}$, together with Stevens's control theorem (Theorem \ref{Theorem: control theorem for overconvergent cohomology}), one deduces that (cuspidal) classical points are very Zariski dense in $\calE_{\circ}$ (see also \cite[Proposition 5.1]{BW}). We denote by $\calE_{\circ}^{\mathrm{cl}}$ the set of (cuspidal) classical points in $\calE_{\circ}$. The Zariski closure $\overline{\calE_{\circ}^{\mathrm{cl}}}^{\mathrm{Zar}}$ of $\calE_{\circ}^{\mathrm{cl}}$ in $\calE^{\mathrm{BSW}}_{\circ}$ is then the union of irreducible components on which cuspidal classical points are Zariski dense. However, since $\calE_{\circ} \hookrightarrow \calE_{\circ}^{\mathrm{BSW}}$ is a Zariski closed embedding and $\calE_{\circ}^{\mathrm{cl}}$ is Zariski dense in $\calE_{\circ}$, $\overline{\calE_{\circ}^{\mathrm{cl}}}^{\mathrm{Zar}}$ coincides with $\calE_{\circ}$. We then conclude (i).

    For (ii), the proof goes verbatim as in \cite[Corollary 5.3]{BW} by applying \cite[Theorem 3.2.1]{JN-irreducible}. In conclusion, we have a commutative diagram of eigenvarieties \[
        \begin{tikzcd}
            \calE \arrow[r, hook] & \calE^{\mathrm{BSW}} \arrow[r, "\wt"] & \calW \\
            \calE_{\circ}\arrow[r, hook]\arrow[u, hook] & \calE_{\circ}^{\mathrm{BSW}}\arrow[u, hook]\arrow[r, "\wt"] & \calW_{\circ}\arrow[u, hook]
        \end{tikzcd},
    \]
    where all arrows are closed embeddings and the rigid analytic spaces that appear in the bottom row are all 1-dimensional. 
\end{proof}

Recall the ideal $\frakn \subset \calO_K$. Suppose it is of the form $\frakn = N\calO_K$, where $N$ is a prime-to-$p$ integer. Let $\calC$ be the Coleman--Mazur eigencurve (of tame level $\Gamma_0(N)$). Then, \cite[Theorem 3.4]{BW} (see also \cite[\S 4.3]{JN-irreducible}), we know that there exists a base change map \[
    \mathrm{BC}: \calC \rightarrow \calE_{\circ},
\]
which is a finite morphism of rigid analytic spaces.

\begin{Corollary}\label{Corollary: parallel family behaves well}
    Any classical base-change point satisfies Assumption \ref{Assumption: vary in 1-dimensional family}.
\end{Corollary}
\begin{proof}
    Given a classical point $x\in \calE_{\circ}(\C_p)$ inside the image of $\mathrm{BC}$, \emph{i.e.}, $x = \mathrm{BC}(y)$ for some $y\in \calC(\C_p)$. Let $\calY$ be a connected affinoid open in $\calC$ passing through $y$. Since $\calY$ is connected, $\mathrm{BC}(\calY)$ is also connected; we may find a connected open neighbourhood $\calX$ of $x$ containing $\mathrm{BC}(\calY)$. Since the Coleman--Mazur eigencurve is eiquidimensional of dimension $1$ and $\mathrm{BC}$ is a finite map, we see that \[
        \dim \calX \geq \dim \calY = 1.
    \] On the other hand, by Proposition \ref{Proposition: parallel eigenvariety mapping into the whole eigenvariety}, $\dim \calX \leq 1$, we thus conclude that $\dim \calX=1$. This shows that $x$ varies in a one-dimensional family. The assertion then follows from that $\calW_{\circ}$ is smooth.
\end{proof}

The phenomenon discussed in Corollary \ref{Corollary: parallel family behaves well} is related to the following conjecture, known as the Calegari--Mazur expectation. We refer the readers to \cite{Calegari-Mazur} and \cite[\S 5.3]{BW} for more discussion. 

\begin{Conjecture}[Calegari--Mazur, Barrera Salazar--Williams]\label{Conjecture: Calegari--Mazur conjecture}
    Given a one-dimensional non-ordinary\footnote{ See Definition \ref{Definition: ordinary part} below. } irreducible component $\calX$ in $\calE_\circ$. Then, there exists an irreducible component $\calY$ of the Coleman--Mazur eigencurve and a finite-order Hecke character $\chi$ such that $\calX = \mathrm{BC}(\calY) \otimes \chi$ in the following sense: for any classical point $x$ in $\calX$, there exists a classical point $y\in \calY$ such that $x = \mathrm{BC}(y)\otimes \chi$. 
\end{Conjecture}

\begin{Remark}\label{Remark: other cases satisfying Assumption 1}
    The purpose of this subsection is to show that we have a handful of sources of classical points which satisfy Assumption \ref{Assumption: vary in 1-dimensional family}. In particular, Corollary \ref{Corollary: parallel family behaves well} shows that base-change points provide such examples. We remark that there are other sources of such classical points. For example, \cite[Theorem 3.8]{BW} showed that any non-critical classical points\footnote{See Definition \ref{Definition: non-critical} below.} may satisfy Assumption \ref{Assumption: vary in 1-dimensional family}.
\end{Remark}

\subsection{Control theorems}\label{subsection: control theorems}
In this subsection, we discuss the control theorems. We first recall \emph{Stevens's control theorem} in our situation. Then, we deduce a control theorem for \emph{ordinary families} à la Hida by applying Stevens's control theorem.

\vspace{2mm}

First of all, let $(R_{\calU}, \kappa_{\calU})$ be an affinoid weight with the associated weight space $\calU \hookrightarrow \calW$ (not necessarily an open immersion), and let $r>r_{\calU}$. Let $F$ be a finite extension of $\Q_p$ and suppose $\kappa : \Z_p^\times  \rightarrow F$ defines a point in $\calU(F)$.\footnote{ Here, we identify $\kappa$ with $(\kappa, \kappa): \Z_p^\times \times\Z_p^\times \rightarrow F$. Similarly, if $k\in \Z_{\geq 0}$, we identify it with $(k,k)\in \Z_{\geq 0}^2$.} In other words, $\kappa$ yields a surjection $R_{\calU} \twoheadrightarrow F$, which is still denoted by $\kappa$. Then, we have a specialisation map \begin{equation}\label{eq: specialisation map; overconvergent}
    D_{\kappa_{\calU}}^r \rightarrow D_{\kappa}^r = D_{\kappa_{\calU}}^r \widehat{\otimes}_{R_{\calU}, \kappa} F.
\end{equation}
If $\kappa = k \in \Z_{\geq 0}$, then there is an obvious injective $\Q_p$-morphism \[
    V_k(\Q_p) \hookrightarrow A_k^{r},
\]
which induces a surjective $\Q_p$-morphism \begin{equation}\label{eq: from D to dual of V}
    D_k^r \twoheadrightarrow V_k^{\vee}(\Q_p).
\end{equation}
Notice that these morphisms are all $\Iw_G$-equivariant. We recall the following control theorem \`a la Stevens:

\begin{Theorem}[Stevens's control theorem]\label{Theorem: control theorem for overconvergent cohomology}
    Let $k \in \Z_{\geq 0}$. Then, for any $h\in \Q_{\geq 0}$ such that $h<k+1$ the morphism in \eqref{eq: from D to dual of V} induces a Hecke-equivariant isomorphism \[
        H_{\mathrm{par}}^i(Y_{\Iw_G}, D_k^r)^{\leq h} \simeq H_{\mathrm{par}}^i(Y_{\Iw_G}, V_k^{\vee}(\Q_p))^{\leq h}.
    \]
\end{Theorem}
\begin{proof}
    Note that we have a commutative diagram \[
        \begin{tikzcd}
            H_c^i(Y_{\Iw_G}, D_{k}^r) \arrow[r]\arrow[d] & H_c^i(Y_{\Iw_G}, V_k^{\vee}(\Q_p)) \arrow[d]\\
            H^i(Y_{\Iw_G}, D_k^r) \arrow[r] & H^i(Y_{\Iw_G}, V_k^{\vee}(\Q_p))
        \end{tikzcd},
    \] where the arrows are all Hecke-equivariant. Applying Stevens's control theorem (see, \cite[Theorem 3.2.5]{Hansen-PhD}), for all $h< h_k$, the horizontal arrows of the commutative diagram \[
        \begin{tikzcd}
            H_c^i(Y_{\Iw_G}, D_{k}^r)^{\leq h} \arrow[r]\arrow[d] & H_c^i(Y_{\Iw_G}, V_k^{\vee}(\Q_p))^{\leq h} \arrow[d]\\
            H^i(Y_{\Iw_G}, D_k^r)^{\leq h} \arrow[r] & H^i(Y_{\Iw_G}, V_k^{\vee}(\Q_p))^{\leq h}
        \end{tikzcd}
    \] are isomorphisms. 
\end{proof}

Now we turn our attention to \emph{ordinary families}. We fix a small weight $(R_{\calU}, \kappa_{\calU})$ with $\calU \hookrightarrow \calW$ (not necessarily an open immersion). Suppose $k \in \Z_{\geq 0}$ defines a point in $\calU(\Q_p)$. Similar as in \eqref{eq: specialisation map; overconvergent} and \eqref{eq: from D to dual of V}, there is a specialisation map \begin{equation}\label{eq: specialisation map; Hida}
     D_{\kappa_{\calU}}^{r, \circ} \rightarrow D_k^{r, \circ} = D_{\kappa_{\calU}}^{r, \circ}\widehat{\otimes}_{R_{\calU}, k}\Z_p \twoheadrightarrow V_k^{\vee}(\Z_p).
\end{equation}
The first goal of this subsection is to show the following control theorem \`a la Hida:

\begin{Theorem}[Hida's control theorem]\label{Theorem: Hida's control theorem}
    Define the ordinary projector \[
        e_{\ord} := \lim_{n \rightarrow \infty} U_p^{n!}.
    \]
    Let $k\in \Z_{\geq 0}$. Then, the morphism in \eqref{eq: specialisation map; Hida} induces a Hecke-equivariant isomorphism \[
        \left(e_{\ord}H_{\mathrm{par}}^i(Y_{\Iw_G}, D_{k}^{r, \circ}) \right)[1/p] \simeq \left(e_{\ord}H_{\mathrm{par}}^i(Y_{\Iw_G}, V_k^{\vee}(\Z_p))\right)[1/p].
    \]
\end{Theorem}
\begin{proof}
    Let $M$ be the kernel of the map $D_{k}^{r, \circ} \rightarrow V_k^{\vee}(\Z_p)$ and so we a Hecke-equivariant  short exact sequence of cochain complexes \begin{equation}\label{eq: short exact sequence for integral specialisation}
        0 \rightarrow C^{\bullet}(\Iw_G, M) \rightarrow C^{\bullet}(\Iw_G, D_k^{r, \circ}) \rightarrow C^{\bullet}(\Iw_G, V_k^{\vee}(\Z_p)) \rightarrow 0.
    \end{equation}
    By taking cohomology and applying the ordinary projector, we have a commutative diagram \[
        \begin{tikzcd}
            \cdots \arrow[r] &  H^i(Y_{\Iw_G}, D_{k}^{r, \circ}) \arrow[r]\arrow[d, two heads] & H^i(Y_{\Iw_G}, V_k^{\vee}(\Z_p)) \arrow[r]\arrow[d, two heads] & H^{i+1}(Y_{\Iw_G}, M) \arrow[r]\arrow[d, two heads] & \cdots\\
            \cdots \arrow[r] & e_{\ord}H^i(Y_{\Iw_G}, D_k^{r, \circ}) \arrow[r] & e_{\ord} H^i(Y_{\Iw_G}, V_k^{\vee}(\Z_p)) \arrow[r] & e_{\ord} H^{i+1}(Y_{\Iw_G}, M) \arrow[r] & \cdots
        \end{tikzcd},
    \] where the top row is a long exact sequence and the vertical arrows are surjections. Note that the vertical arrows in the diagram admit splittings given by the natural inclusions, making the whole diagram commutative again. We claim that the bottom row of the diagram is also exact. Indeed, suppose $[\mu]\in e_{\ord}H^i(Y_{\Iw_G}, V_k^{\vee}(\Z_p))$ maps to $0$ in $e_{\ord} H^{i+1}(Y_{\Iw_G}, M)$, then there exists $[\mu']\in H^i(Y_{\Iw_G}, D_k^{r, \circ})$ whose image in $H^i(Y_{\Iw_G}, V_k^{\vee}(\Z_p))$ is $[\mu]$. However, $e_{\ord}[\mu']\in e_{\ord} H^i(Y_{\Iw_G}, D_k^{r, \circ})$ maps to $[\mu]$. This shows the exactness at $e_{\ord}H^i(Y_{\Iw_G}, V_k^{\vee}(\Z_p))$. One then uses the same argument to show the exactness at the other terms. 

    On the other hand, for any $h < h_k$, we have a surjection \[
        C^{\bullet}(\Iw_G, D_k^{r})^{\leq h} \rightarrow C^{\bullet}(\Iw_G, V_k^{\vee}(\Q_p))^{\leq h}.
    \] We then define \[
        C^{\bullet}(\Iw_G, M[1/p])^{\leq h} : = \ker \left(  C^{\bullet}(\Iw_G, D_k^{r})^{\leq h} \rightarrow C^{\bullet}(\Iw_G, V_k^{\vee}(\Q_p))^{\leq h} \right).
    \] The cohomology of $C^{\bullet}(\Iw_G, M[1/p])^{\leq h}$ is then denoted by $H^i(Y_{\Iw_G}, M[1/p])^{\leq h}$. Hence, we have a long exact sequence \[
        \cdots \rightarrow H^i(Y_{\Iw_G}, D_k^r)^{\leq h} \rightarrow H^i(Y_{\Iw_G}, V_k^{\vee}(\Q_p))^{\leq h} \rightarrow H^{i+1}(Y_{\Iw_G}, M[1/p])^{\leq h} \rightarrow \cdots .
    \] However, by Stevens's control theorem (see, for example, \cite[Theorem 3.2.5]{Hansen-PhD} or the proof of Theorem \ref{Theorem: control theorem for overconvergent cohomology}), we know that $H^i(Y_{\Iw_G}, M)^{\leq h} = 0$.\footnote{ In fact, one knows first that $H^i(Y_{\Iw_G}, M)^{\leq h} = 0$ to conclude the isomorphism in Stevens's control theorem.}

    Now, for any $N \in \{ D_k^{r, \circ}, V_k^{\vee}(\Z_p), M \}$, we have a natural map \[
        e_{\ord}H^i(Y_{\Iw_G}, N) \hookrightarrow H^i(Y_{\Iw_G}, N) \rightarrow H^i(Y_{\Iw_G}, N)[1/p],
    \] which factors through an inclusion \[
        \left(e_{\ord} H^i(Y_{\Iw_G}, N) \right)[1/p] \hookrightarrow H^i(Y_{\Iw_G}, N)[1/p].
    \] Moreover, one also observes that the $U_p$-eigenclasses in $\left(e_{\ord} H^i(Y_{\Iw_G}, N) \right)[1/p]$ are of slope $0$. Thus, the inclusion factors as \[
        \left(e_{\ord} H^i(Y_{\Iw_G}, N) \right)[1/p] \hookrightarrow H^i(Y_{\Iw_G}, N)[1/p]^{\leq h} \hookrightarrow H^i(Y_{\Iw_G}, N)[1/p]
    \] for any $h < h_k$.

    Putting everything together, one obtains a commutative diagram \[
        \begin{tikzcd}
            \cdots \arrow[r] & \left(e_{\ord} H^i(Y_{\Iw_G}, D_k^{r, \circ})\right)[1/p] \arrow[r]\arrow[d, hook'] & \left(e_{\ord} H^i(Y_{\Iw_G}, V_k^{\vee}(\Z_p))\right)[1/p] \arrow[r]\arrow[d, hook'] & \left(e_{\ord} H^{i+1}(Y_{\Iw_G}, M)\right)[1/p] \arrow[r]\arrow[d, hook'] & \cdots\\
            \cdots \arrow[r] & H^i(Y_{\Iw_G}, D_k^{r})^{\leq h} \arrow[r] & H^i(Y_{\Iw_G}, V_k^{\vee}(\Q_p))^{\leq h} \arrow[r] & H^{i+1}(Y_{\Iw_G}, M[1/p])^{\leq h} \arrow[r] & \cdots
        \end{tikzcd},
    \] where the horizontal arrows are exact. Since $H^{i}(Y_{\Iw_G}, M[1/p])^{\leq h} = 0$, we then conclude that \[
        \left(e_{\ord} H^i(Y_{\Iw_G}, D_k^{r, \circ})\right)[1/p] \simeq \left(e_{\ord} H^i(Y_{\Iw_G}, V_k^{\vee}(\Z_p))\right)[1/p].
    \]

    Finally, one uses the same argument to deduce that \[
        \left(e_{\ord} H_c^i(Y_{\Iw_G}, D_k^{r, \circ})\right)[1/p] \simeq \left(e_{\ord} H_c^i(Y_{\Iw_G}, V_k^{\vee}(\Z_p))\right)[1/p].
    \] The desired assertion then follows from the Hecke-equivariant commutative diagram \[
        \begin{tikzcd}
            \left(e_{\ord} H_c^i(Y_{\Iw_G}, D_k^{r, \circ})\right)[1/p] \arrow[r, "\simeq"] \arrow[d] & \left(e_{\ord} H_c^i(Y_{\Iw_G}, V_k^{\vee}(\Z_p))\right)[1/p]\arrow[d]\\
            \left(e_{\ord} H^i(Y_{\Iw_G}, D_k^{r, \circ})\right)[1/p] \arrow[r, "\simeq"] & \left(e_{\ord} H^i(Y_{\Iw_G}, V_k^{\vee}(\Z_p))\right)[1/p]
        \end{tikzcd}.
    \]
\end{proof}


\begin{Remark}
    Theorem \ref{Theorem: Hida's control theorem} is supposed to be well-known to experts. However, we lack the knowledge of such a statement and its proof in the literature. Whence, we provide them here. In fact, one sees that the statement and the argument can easily be generalised to other cases. 
\end{Remark}

\begin{Definition}\label{Definition: ordinary part}
    For any small weight $(R_{\calU}, \kappa_{\calU})$ such that  $\calU = \Spa(R_{\calU}, R_{\calU})^{\rig}$ is a subspace in $\calW$, we shall denote \[
        H^i_{\mathrm{par}}(Y_{\Iw_G}, D_{\kappa_{\calU}}^{r})^{\ord} := \left( e_{\ord} H_{\mathrm{par}}^i(Y_{\Iw_G}, D_{\kappa_{\calU}}^{r, \circ})  \right)[1/p] \hookrightarrow H_{\mathrm{par}}^i(Y_{\Iw_G}, D_{\kappa_{\calU}}^{r})
    \] and call it the \textbf{ordinary part} of $H_{\mathrm{par}}^i(Y_{\Iw_G}, D_{\kappa_{\calU}}^{r})$. If moreover $\calU$ is open in $\calW_{\circ}$, any Hecke-eigenclass in $H^i_{\mathrm{par}}(Y_{\Iw_G}, D_{\kappa_{\calU}}^{r})^{\ord}$ shall be referred to as an \textbf{ordinary family} (or a \textbf{Hida family}) of cuspidal Bianchi eigenclasses.
\end{Definition}

Inspired by the control theorems above, we make the following definition (see also \cite[Definition 2.8]{BW}).

\begin{Definition}\label{Definition: non-critical}
    Let $f$ be a cuspidal Bianchi eigenform of parallel weight $k>0$ and let $\frakm_f$ be the maximal ideal in $\bbT_{\Iw_G}$ defined by $f$. We say $f$ is \textbf{non-critical} if, for each $i$, the natural map \[
        H_{\Par}^i(Y_{\Iw_G}, D_k^r) \rightarrow H_{\Par}^i(Y_{\Iw_G}, V_k^{\vee}(\Q_p))
    \]
    becomes an isomorphism after localising at $\frakm_f$. 
\end{Definition}

\begin{Remark}\label{Remark: non-critical classes}
    Let $f$ be as in the definition above. Let $[\mu_f^{(i)}]$ be the image of $f$ in $H_{\Par}^i(Y_{\Iw_G}, V_k^{\vee}(\Q_p))$ via the Eichler--Shimura--Harder isomorphism. From the control theorems above, one easily sees that, in the following two cases, $f$ is non-critical:\begin{enumerate}
        \item[$\bullet$] if $h<k+1$ and $[\mu_f^{(i)}] \in H_{\Par}^i(Y_{\Iw_G}, V_k^{\vee}(\Q_p))^{\leq h}$; 
        \item[$\bullet$] if $[\mu_f^{(i)}]\in H_{\Par}^i(Y_{\Iw_G}, V_k^{\vee}(\Q_p))^{\ord} := \left( e_{\ord}H_{\Par}^i(Y_{\Iw_G}, V_k^{\vee}(\Z_p)) \right)[1/p]$.
    \end{enumerate}
    Moreover, we have more examples for $f$ being non-critical. Indeed, suppose $\alpha_{\frakp}$ (resp., $\alpha_{\overline{\frakp}}$) is the $U_{\frakp}$-eigenvalue (resp., $U_{\overline{\frakp}}$-eigenvalue) of $f$, by \cite[Theorem 2.10]{BW}, $f$ is also non-critical when \[
        \max\{ v_p(\alpha_{\frakp}), v_p(\alpha_{\overline{\frakp}})\} <k+1.
    \]
\end{Remark}
\section{The adjoint pairing}\label{section: pairing}
In this section, we discuss a pairing of the overconvergent cohomology groups. This pairing is inspired by the works of Kim and Bellaïche (\cite{Kim, Eigen}), which is later generalised by the second-named author to the Siegel case (\cite{Wu-pairing}) and by Balasubramanyam--Longo to the Hilbert case \cite{BL}. By applying the formalism in \cite[Chapter VIII]{Eigen}, we deduce the relation between this pairing and the geometry of the Bianchi eigencurves. The relation between this pairing and the adjoint $L$-values is discussed in \S \ref{section: adjoint L-values}. 

\subsection{The pairing on the overconvergent cohomology groups}\label{subsection: pairing; overconvergent}

Let $(R_{\calU}, \kappa_{\calU})$ be a weight and let $r>r_{\calU}$. We define a pairing on $D_{\kappa_{\calU}}^{r, \circ}$ by the following formula \begin{align*}
    [\cdot, \cdot]_{\kappa_{\calU}}^{\circ}: D_{\kappa_{\calU}}^{r, \circ} \times D_{\kappa_{\calU}}^{r, \circ} & \rightarrow R_{\calU}^{\circ}\\
    (\mu, \mu') & \mapsto \int_{\left(\left(\substack{1\\c_1}\right), \left(\substack{1\\c_2}\right)\right)\in \T_{00}}\int_{\left(\left(\substack{1\\c_1'}\right), \left(\substack{1\\c_2'}\right)\right)\in \T_{00}} \kappa_{\calU, 1}(1+c_1c_1'/p)\kappa_{\calU, 2}(1+c_2c_2'/p) d\mu' d\mu.
\end{align*}
Let us explain the formula. Given a fixed $\left(\left(\substack{1\\c_1}\right), \left(\substack{1\\c_2}\right)\right)\in \T_{00}$, the function \[
    \left(\begin{pmatrix}a_1'\\c_1'\end{pmatrix}, \begin{pmatrix}a_2'\\c_2'\end{pmatrix}\right) \mapsto \kappa_{\calU,1}(a_1'+c_1c_1'/p)\kappa_{\calU,2}(a_2'+c_2c_2'/p)
\]
defines a function in $A_{\kappa_{\calU}}^{r, \circ}$ (see Example \ref{Example: highest weight element in analytic functions} and note that $p \mid c_i$). We shall view $1 + c_ic_i'/p$ as the matrix multiplication \[
    1 + c_ic_i'/p = \begin{pmatrix}1 & c_i'\end{pmatrix}\begin{pmatrix}1 & \\& p^{-1}\end{pmatrix} \begin{pmatrix}1\\ c_i\end{pmatrix} = \trans\begin{pmatrix}1\\ c_i'\end{pmatrix} \begin{pmatrix}1 & \\ & p^{-1}\end{pmatrix}\begin{pmatrix}1\\ c_i\end{pmatrix}.
\] Hence, we have a morphism \[
    D_{\kappa_{\calU}}^{r, \circ} \rightarrow A_{\kappa_{\calU}}^{r, \circ}, \quad \mu \mapsto \left( \Phi_{\mu}: \left(\begin{pmatrix}a_1'\\c_1'\end{pmatrix}, \begin{pmatrix}a_2'\\c_2'\end{pmatrix}\right) \mapsto \int_{\left(\left(\substack{1\\c_1}\right), \left(\substack{1\\c_2}\right)\right)\in \T_{00}} \kappa_{\calU,1}(a_1'+c_1c_1'/p)\kappa_{\calU,2}(a_2'+c_2c_2'/p) d\mu\right)
\]
Hence, $[\cdot, \cdot]_{\kappa_{\calU}}^{\circ}$ is nothing but the composition of morphisms \[
    D_{\kappa_{\calU}}^{r, \circ} \times D_{\kappa_{\calU}}^{r, \circ} \rightarrow D_{\kappa_{\calU}}^{r, \circ} \times A_{\kappa_{\calU}}^{r, \circ} \rightarrow R_{\calU}^{\circ}, \quad (\mu, \mu')\mapsto (\mu, \Phi_{\mu'}) \mapsto \mu(\Phi_{\mu'}).
\]
We denote by $[\cdot, \cdot]_{\kappa_{\calU}}$ the induced pairing on $D_{\kappa_{\calU}}^{r}$.

\begin{Remark}\label{Remark: relation to the algebraic pairing}
    For any $(k_1, k_2)\in \Z_{\geq 0}^2$ and recall the algebraic representation $V_k(\Q_p)$ and $V_k^{\vee}(\Q_p)$. Thanks to Remark \ref{Remark: alg. rep. is self-dual}, we see that there is a pairing \[
        \langle \cdot, \cdot \rangle_{k}: V_k^{\vee}(\Q_p) \times V_k^{\vee}(\Q_p) \rightarrow \Q_p, \quad (\mu, \mu') \mapsto \mu\left(\phi_{\mu'}\right).
    \]
    Unwinding the formulae in \emph{loc. cit.}, this pairing can be written as \[
        (\mu, \mu')\mapsto \int_{\left(\begin{pmatrix}1 & \\ X_1 & 1\end{pmatrix}, \begin{pmatrix}1 & \\ X_2 & 1\end{pmatrix}\right)\in N_G^{\opp}(\Q_p)}\int_{\left(\begin{pmatrix}1 & \\ X_1' & 1\end{pmatrix}, \begin{pmatrix}1 & \\ X_2' & 1\end{pmatrix}\right)\in N_G^{\opp}(\Q_p)} (X_1' - X_1)^{k_1}(X_2' - X_2)^{k_2} d\mu' d\mu.
    \]
    We can view $[\cdot, \cdot]_{k}$ as a twist of $\langle \cdot, \cdot \rangle_k$ by the \emph{Atkin--Lehner operator}\[
        \bfomega_p = \left(\begin{pmatrix} & -p\\ 1 & \end{pmatrix}, \begin{pmatrix} & -p\\ 1 & \end{pmatrix}\right)\in G(\Q_p)
    \] on the second variable. To see this, note that the embedding $V_k(\Q_p) \hookrightarrow A_k^r$ makes functions $f$ in $V_k(\Q_p)$ as functions on $\Z_p^2$ via \[
        f: \Z_p^2 \rightarrow \Q_p, \quad (X_1, X_2) \mapsto f\left( \begin{pmatrix}1 & \\ pX_1 & 1\end{pmatrix}, \begin{pmatrix}1 & \\ pX_2 & 1\end{pmatrix}\right).
    \] Applying $[\cdot, \cdot]_{k}$ to $V_k^{\vee}(\Q_p)$, we have \[
        [\mu, \mu']_k = \int_{(X_1, X_2)\in \Z_p^2}\int_{(X_1', X_2')\in \Z_p^2} (1+pX_1pX_1'/p)^{k_1} (1+pX_xpX_2')^{k_2} d\mu d\mu' = \int\int (1+pX_1X_1')^{k_1}(1+pX_2X_2')^{k_2} d\mu d\mu'.
    \]
    On the other hand, we have the following computation \[
        \trans\left(\begin{pmatrix}& -p\\ 1\end{pmatrix}\begin{pmatrix}1\\ X_i'\end{pmatrix}\right) \begin{pmatrix}& -1\\ 1\end{pmatrix}\begin{pmatrix}1\\ X_i\end{pmatrix} = \begin{pmatrix}1 & X_i'\end{pmatrix} \begin{pmatrix}& 1\\ -p\end{pmatrix}\begin{pmatrix}& -1\\ 1\end{pmatrix}\begin{pmatrix}1\\ X_i\end{pmatrix} = 1+pX_iX_i'
    \]
    for $i=1, 2$, which gives back the formula for $[\cdot, \cdot]_k$. 
\end{Remark}

\begin{Lemma}\label{Lemma: adjoint operators for the pairing}
    Recall that $\T_0$ admits a left-action by $\Xi$. We define an involution on $\Xi$ by \[
        \bfalpha = \left(\begin{pmatrix} a_1 & b_1\\ c_1 & d_1\end{pmatrix}, \begin{pmatrix} a_2 & b_2\\ c_2 & d_2\end{pmatrix}\right) \mapsto \bfalpha^{\Shi} = \left(\begin{pmatrix} a_1 & p^{-1}c_1\\ pb_1 & d_1\end{pmatrix}, \begin{pmatrix} a_2 & p^{-1}c_2\\ pb_2 & d_2\end{pmatrix}\right).
    \]
    Let $(R_{\calU}, \kappa_{\calU})$ be a weight and let $r>r_{\calU}$. Then, for any $\mu, \mu'\in D_{\kappa_{\calU}}^{r, \circ}$, we have \[
        [\bfalpha\mu, \mu']_{\kappa_{\calU}}^{\circ} = [\mu, \bfalpha^{\Shi}\mu']_{\kappa_{\calU}}^{\circ}.
    \]
\end{Lemma}
\begin{proof}
    The lemma follows from the following computation of matrix multiplication \[
        \begin{pmatrix}1 & \\ & p^{-1}\end{pmatrix}\begin{pmatrix} a & b \\ c & d\end{pmatrix} = \begin{pmatrix} a & b \\ p^{-1}c & p^{-1}d\end{pmatrix} = \begin{pmatrix}a & pb\\ p^{-1}c & d\end{pmatrix}\begin{pmatrix}1 & \\ & p^{-1}\end{pmatrix}.
    \]
\end{proof}

\begin{Lemma}\label{Lemma: the pairing respect to the filtration}
    Let $(R_{\calU}, \kappa_{\calU})$ be a small weight and let $r>1+ r_{\calU}$. Let $\fraka_{\calU}\subset R_{\calU}$ be an ideal of definition as in Remark \ref{Remark: filtration on distributions}. Then, the pairing $[\cdot, \cdot]_{\kappa_{\calU}}^{\circ}$ induces a pairing on the graded pieces \[
        [\cdot, \cdot]_{\kappa_{\calU}, j}^{\circ} : D_{\kappa_{\calU}, j}^{r, \circ} \times D_{\kappa_{\calU}, j}^{r, \circ} \rightarrow R_{\calU}/\fraka_{\calU}^j.
    \]
\end{Lemma}
\begin{proof}
    Observe first that we have a commutative diagram \[
        \begin{tikzcd}
            D_{\kappa_{\calU}}^{r, \circ} \times D_{\kappa_{\calU}}^{r, \circ} \arrow[d, shift left = 5mm]\arrow[d, shift right = 5mm]\arrow[r, "\text{$[\cdot, \cdot]_{\kappa_{\calU}}^{\circ}$}"] & R_{\calU}\arrow[d, equal]\\
            D_{\kappa_{\calU}}^{r-1, \circ} \times D_{\kappa_{\calU}}^{r-1, \circ} \arrow[r, "\text{$[\cdot, \cdot]_{\kappa_{\calU}}^{\circ}$}"] & R_{\calU}
        \end{tikzcd}
    \]
    since the pairings on both spaces are defined by using the same formula. 
    Moreover, we have \[
        [\cdot, \cdot]_{\kappa_{\calU}}^{\circ}: \fraka_{\calU}^j D_{\kappa_{\calU}}^{r, \circ} \times \fraka_{\calU}^j D_{\kappa_{\calU}}^{r, \circ} \rightarrow \fraka_{\calU}^jR_{\calU}
    \]
    by the linearity of the pairing. Putting everything together, we arrive at the desired pairing. 
\end{proof}

Given a weight $(R_{\calU}, \kappa_{\calU})$ and $r>r_{\calU}$. The pairing $[\cdot, \cdot]_{\kappa_{\calU}}^{\circ}$ and the cup pruduct then induces a pairing on the cohomology groups \[
    H^i(Y_{\Iw_G}, D_{\kappa_{\calU}}^{r, \circ}) \times H_c^{3-i}(Y_{\Iw_G}, D_{\kappa_{\calU}}^{r, \circ}) \xrightarrow{\smile} H_c^{3}(Y_{\Iw_G}, D_{\kappa_{\calU}}^{r, \circ}\otimes_{R_{\calU}^{\circ}}D_{\kappa_{\calU}}^{r, \circ})\xrightarrow{[\cdot, \cdot]_{\kappa_{\calU}}^{\circ}} H^3_c(Y_{\Iw_G}, R_{\calU}^{\circ}) \simeq R_{\calU}^{\circ}. 
\]
The compatibility of cup products yields a commutative diagram \[
    \begin{tikzcd}
        H^i(Y_{\Iw_G}, D_{\kappa_{\calU}}^{r, \circ}) \times H_{c}^{3-i}(Y_{\Iw_G}, D_{\kappa_{\calU}}^{r, \circ}) \arrow[r, "\smile"] & H_c^{3}(Y_{\Iw_G}, D_{\kappa_{\calU}}^{r, \circ} \otimes_{R_{\calU}^{\circ}} D_{\kappa_{\calU}}^{r, \circ}) \\
        H_c^i(Y_{\Iw_G}, D_{\kappa_{\calU}}^{r, \circ}) \times H_c^{3-i}(Y_{\Iw_G}, D_{\kappa_{\calU}}^{r, \circ}) \arrow[r, "\smile"] \arrow[u, shift left = 15mm]\arrow[u, equal, shift right = 15mm] \arrow[d, equal, shift right = 15mm]\arrow[d, shift left = 15mm] & H_c^{3}(Y_{\Iw_G}, D_{\kappa_{\calU}}^{r, \circ} \otimes_{R_{\calU}^{\circ}} D_{\kappa_{\calU}}^{r, \circ}) \arrow[u, equal]\arrow[d, equal]\\
        H_c^i(Y_{\Iw_G}, D_{\kappa_{\calU}}^{r, \circ}) \times H^{3-i}(Y_{\Iw_G}, D_{\kappa_{\calU}}^{r, \circ}) \arrow[r, "\smile"] & H_c^{3}(Y_{\Iw_G}, D_{\kappa_{\calU}}^{r, \circ} \otimes_{R_{\calU}^{\circ}} D_{\kappa_{\calU}}^{r, \circ})
    \end{tikzcd}
\]
and hence induces a well-defined pairing \[
    [\cdot, \cdot]_{\kappa_{\calU}}^{\circ}: H_{\mathrm{par}}^i(Y_{\Iw_G}, D_{\kappa_{\calU}}^{r, \circ}) \times H_{\mathrm{par}}^{3-i}(Y_{\Iw_G}, D_{\kappa_{\calU}}^{r, \circ}) \rightarrow R_{\calU}^{\circ},
\]
for which we abuse the notation and still denote it by $[\cdot, \cdot]_{\kappa_{\calU}}^{\circ}$.\footnote{  Here, for $D\in \{D_{\kappa_{\calU}}^{r, \circ}, D_{\kappa_{\calU}}^r, D_{\kappa_{\calU}}^{r^+, \circ}, D_{\kappa_{\calU}}^{r^+}\}$, $H_{\mathrm{par}}^i(Y_{\Iw_G}, D)$ is (obviously) defined to be the image of $H_c^i(Y_{\Iw_G}, D)$ inside $H^i(Y_{\Iw_G}, D)$. Note that since the natural map $H_c^i(Y_{\Iw_G}, D) \rightarrow H^i(Y_{\Iw_G}, D)$ is Hecke-equivariant, $H_{\mathrm{par}}^i(Y_{\Iw_G}, D)$ is $\bbT$-stable. } Similarly, we use the notation $[\cdot, \cdot]_{\kappa_{\calU}}$ to denote that induced pairing on $H_{\mathrm{par}}^i(Y_{\Iw_G}, D_{\kappa_{\calU}}^r)$.

\begin{Proposition}\label{Proposition: pairing on cohomology groups is Hecke-equivariant}
    Given a weight $(R_{\calU}, \kappa_{\calU})$ and $r>r_{\calU}$, the pairings $[\cdot, \cdot]_{\kappa_{\calU}}^{\circ}$ and  $[\cdot, \cdot]_{\kappa_{\calU}}$ on $H_{\mathrm{par}}^*$ are Hecke-equivariant.
\end{Proposition}
\begin{proof}
    For Hecke operators away from $p$ and $\frakn$, since those Hecke operators acts on $D_{\kappa_{\calU}}^{r, \circ}$ and $D_{\kappa_{\calU}}^{r}$ trivially, the equivariance of those operators is formal. For the operators $U_{\frakp}$ and $U_{\overline{\frakp}}$ operator, note that the double cosets $\Iw_G \bfupsilon_{\frakp} \Iw_G$ and $\Iw_G \bfupsilon_{\overline{\frakp}} \Iw_G$ are stable under the involution $\bfalpha \mapsto \bfalpha^{\Shi}$ (Lemma \ref{Lemma: adjoint operators for the pairing}), hence the pairing is also $U_{\frakp}$- and $U_{\overline{\frakp}}$-equivariant. 
\end{proof}

\begin{Proposition}\label{Proposition: integral pairing when we have a small weight}
    Let $(R_{\calU}, \kappa_{\calU})$ be a small weight and let $r>1+r_{\calU}$. For any $j\in \Z_{\geq 0}$, let $D_{\kappa_{\calU}, j}^{r, \circ}$ be the $j$-th graded piece of the filtration defined in Remark \ref{Remark: filtration on distributions}.
    \begin{enumerate}
        \item[(i)]  We have a canonical isomorphism \[
            H_{\mathrm{par}}^i(Y_{\Iw_G}, D_{\kappa_{\calU}}^{r, \circ}) \simeq \varprojlim_{j} H_{\mathrm{par}}^i(Y_{\Iw_G}, D_{\kappa_{\calU}, j}^{r, \circ}).
        \] 
        \item[(ii)] For any $j$, we have a commutative diagram \[
            \begin{tikzcd}
                H_{\mathrm{par}}^i(Y_{\Iw_G}, D_{\kappa_{\calU}}^{r, \circ}) \times H_{\mathrm{par}}^{3-i}(Y_{\Iw_G}, D_{\kappa_{\calU}}^{r, \circ}) \arrow[r, "\text{$[\cdot, \cdot]_{\kappa_{\calU}}^{\circ}$}"]\arrow[d, shift left = 15mm]\arrow[d, shift right = 15mm] & R_{\calU}^{\circ}\arrow[d]\\
                H_{\mathrm{par}}^i(Y_{\Iw_G}, D_{\kappa_{\calU},j}^{r, \circ}) \times H_{\mathrm{par}}^{3-i}(Y_{\Iw_G}, D_{\kappa_{\calU},j}^{r, \circ}) \arrow[r, "\text{$[\cdot, \cdot]_{\kappa_{\calU},j}^{\circ}$}"] & R_{\calU}^{\circ}/\fraka_{\calU}^j
            \end{tikzcd}.
        \]
    \end{enumerate}
\end{Proposition}
\begin{proof}
    The second assertion follows immediately from (i) and Lemma \ref{Lemma: the pairing respect to the filtration}. 
    
    To show (i), it is enough to show that \[
        H^i(Y_{\Iw_G}, D_{\kappa_{\calU}}^{r, \circ}) \simeq \varprojlim_{j} H^i(Y_{\Iw_G}, D_{\kappa_{\calU}, j}^{r, \circ})\quad \text{ and }\quad H^i_c(Y_{\Iw_G}, D_{\kappa_{\calU}}^{r, \circ}) \simeq \varprojlim_j H_c^i(Y_{\Iw_G}, D_{\kappa_{\calU}, j}^{r, \circ}). 
    \]
    This amount to show that $R^1\varprojlim_j D_{\kappa_{\calU}, j}^{r, \circ} = 0$. However, since the transition maps in the projective system $\{D_{\kappa_{\calU}, j}^{r, \circ}\}_{j\in \Z_{\geq 0}}$ are surjective, it is a Mittag-Leffler system. Consequently,  $R^1\varprojlim_j D_{\kappa_{\calU}, j}^{r, \circ} = 0$.
\end{proof}

\begin{Proposition}\label{Proposition: non-degeneracy of the pairing}
    Let $k\in \Z_{\geq 0}$. Then, for any $h\in \Q_{\geq 0}$ with $h<k+1$, the pairing \[
        H_{\Par}^i(Y_{\Iw_G}, D_{k}^r)^{\leq h} \times H_{\Par}^{3-i}(Y_{\Iw_G}, D_k^r)^{\leq h} \rightarrow \Q_p
    \]
    is a perfect pairing. 
\end{Proposition}
\begin{proof}
    In this situation, we first apply Theorem \ref{Theorem: control theorem for overconvergent cohomology} and so we only need to show the perfectness of the pairing \[
        H_{\Par}^i(Y_{\Iw_G}, V_k^{\vee}(\Q_p))^{\leq h} \times H_{\Par}^{3-i}(Y_{\Iw_G}, V_k^{\vee}(\Q_p))^{\leq h} \rightarrow \Q_p.
    \]

    To this end, we first claim that the pairing \[
        [\cdot, \cdot]_k : V_k^{\vee}(\Q_p) \times V_k^{\vee}(\Q_p) \rightarrow \Q_p
    \] is a perfect pairing. By Remark \ref{Remark: relation to the algebraic pairing}, we may view $[\cdot, \cdot]_{k}$ as a twisted pairing of the perfect pairing $\langle \cdot, \cdot \rangle_k$ (see Remark \ref{Remark: alg. rep. is self-dual}) by the Atkin--Lehner operator $\bfomega_p$. However, the action of $\bfomega_p$ on $V_k^{\vee}(\Q_p)$ is invertible with inverse given by $\bfomega_p^{-1} = \left( (\substack{\,\,\,\,\,\,\,\,\,\,\,\,\,\,\,\,\,\, 1\\ -p^{-1}\,\,\,\,\,\,\,\,\,\,}), (\substack{\,\,\,\,\,\,\,\,\,\,\,\,\,\,\,\,\,\, 1\\ -p^{-1}\,\,\,\,\,\,\,\,\,\,})\right) \in G(\Q_p)$.

    Now, since $[\cdot, \cdot]_k$ is a perfect pairing, by Poincaré duality (see, for example, \cite[Corollary III.3.12]{Eigen}), we have two perfect pairings \begin{align}
        H^i(Y_{\Iw_G}, V_k^{\vee}(\Q_p))^{\leq h} \times H_{c}^{3-i}(Y_{\Iw_G}, V_k^{\vee}(\Q_p))^{\leq h} & \rightarrow \Q_p, \\
        H_{c}^i(Y_{\Iw_G}, V_k^{\vee}(\Q_p))^{\leq h} \times H^{3-i}(Y_{\Iw_G}, V_k^{\vee}(\Q_p))^{\leq h} & \rightarrow \Q_p .
    \end{align}
    Hence, we have a commutative diagram \[
        \begin{tikzcd}
            \Hom_{\Q_p}(H^i(Y_{\Iw_G}, V_k^{\vee}(\Q_p)), \Q_p) \arrow[r, two heads] & \Hom_{\Q_p}(H_{\Par}^i(Y_{\Iw_G}, V_k^{\vee}, \Q_p)) \arrow[r, hook] & \Hom_{\Q_p}(H_c^i(Y_{\Iw_G}, V_k^{\vee}(\Q_p)), \Q_p)\\
            H_c^{3-i}(Y_{\Iw_G}, V_k^{\vee}(\Q_p)) \arrow[r, two heads]\arrow[u, equal] & H_{\Par}^{3-i}(Y_{\Iw_G}, V_k^{\vee}(\Q_p)) \arrow[u] \arrow[r, hook] & H^{3-i}(Y_{\Iw_G}, V_k^{\vee}(\Q_p))\arrow[u, equal]
        \end{tikzcd}.
    \]
    One then deduces that the middle vertical map is surjective and injective, which shows the desired result. 
\end{proof}

As an immediate consequence of Proposition \ref{Proposition: integral pairing when we have a small weight} and Proposition \ref{Proposition: non-degeneracy of the pairing}, we have the following corollaries. 

\begin{Corollary}\label{Corollary: non-degenercay of the pairing; ordinary case}
    Let $k\in \Z_{\geq 0}$. Then, we have a perfect pairing \[
        H_{\Par}^i(Y_{\Iw_G}, D_k^r)^{\ord} \times H_{\Par}^{3-i}(Y_{\Iw_G}, D_k^r)^{\ord} \rightarrow \Q_p.
    \]
\end{Corollary}

\begin{Corollary}\label{Corollary: non-degenercay of the pairing; non-critical case}
    Let $k\in \Z_{> 0}$. Let $f$ be a classical cuspidal Bianchi eigenform of cohomological parallel weight $k$. Let $\frakm_f$ be the maximal ideal in the Hecke algebra defined by $f$. Let $F_f$ be the residue field of $\frakm_f$. Suppose $f$ is non-critical. Then, \[
        H_{\Par}^1(Y_{\Iw_G}, D_k^r)_{\frakm_f} \times H_{\Par}^2(Y_{\Iw_G}, D_k^r)_{\frakm_f} \rightarrow F_f
    \]  is a perfect pairing. 
\end{Corollary}

\subsection{\texorpdfstring{$p$}{p}-adic 
 adjoint \texorpdfstring{$L$}{L}-functions varying in finite-slope families}\label{subsection: p-adic adjoint L-functions varying in finie-slope families}

Recall the cuspidal eigenvariety $\calE$. Suppose $x\in \calE(F_x)$\footnote{ Here, $F_x$ denotes the residue field associated with $x$.} is a point varying in a curve $\calU \subset \calW$. Let $(R_{\calU}, \kappa_{\calU})$ be the universal weight of $\calU$. We denote by $\kappa = \wt(x)$ and $\frakm_{\kappa}$ the corresponding maximal ideal in $R_{\calU}$ and we assume $\calU$ is smooth at $\kappa$ (see Assumption \ref{Assumption: vary in 1-dimensional family}). 

For $i\in \{0, 1, 2, 3\}$, we denote by $\bbT_{\calU, h}^{(i)}$ the reduced $R_{\calU}$-subalgebra of $\End_{R_{\calU}}(H_{\Par}^i(Y_{\Iw_G}, D_{\kappa_{\calU}}^r)^{\leq h})$ generated by the image of $\bbT_{\Iw_G}$. We shall assume $x$ defines a maximal ideal $\frakm_x$ in both $\bbT_{\calU, h}^{(1)}$ and $\bbT_{\calU, h}^{(2)}$. We consider the weight map \[
    \wt: \calE_{\calU, h} := \Spa(\bbT_{\calU, h}^{(1)}, \bbT_{\calU, h}^{(1), \circ}) \rightarrow \calU
\]
induced by the natural map $R_{\calU} \rightarrow \bbT_{\calU}^{(1)}$. Note that, if $\calU$ is a family of parallel weights, then $\calE_{\calU, h}$ is nothing but an affinoid open in the cuspidal parallel Bianchi eigenvariety $\calE_\circ$. 

We assume furthermore that we are in the following situation: \begin{enumerate}
    \item[(Adj1)] The localisation $\bbT_{\calU, h, \frakm_x}^{(1)}$ is flat over $R_{\calU, \frakm_{\kappa}}$. 
    \item[(Adj2)] The dual algebra $\bbT_{\calU, h, \frakm_x}^{(1), \vee} := \Hom_{R_{\calU, \frakm_{\kappa}}}(\bbT_{\calU, h, \frakm_x}^{(1)}, R_{\calU, \frakm_{\kappa}})$ is free of rank 1 over $\bbT_{\calU, h, \frakm_{x}}^{(1)}$, \emph{i.e.}, $\bbT_{\calU, h, \frakm_x}^{(1)}$ is Gorenstein over $R_{\calU, \frakm_{\kappa}}$. 
    \item[(Adj3)] The induced pairing \[
        [\cdot, \cdot]_{\kappa_{\calU}, \frakm_x}:  H_{\Par}^1(Y_{\Iw_G}, D_{\kappa_{\calU}}^r)^{\leq h}_{\frakm_x} \times   H_{\Par}^2(Y_{\Iw_G}, D_{\kappa_{\calU}}^r)^{\leq h}_{\frakm_x} \rightarrow R_{\calU, \frakm_{\kappa}}
    \] is non-degenerate. In particular, there is an injection $H_{\Par}^2(Y_{\Iw_G}, D_{\kappa_{\calU}}^r)_{\frakm_x}^{\leq h} \hookrightarrow \Hom_{R_{\calU, \frakm_{\kappa}}}(H_{\Par}^1(Y_{\Iw_G}, D_{\kappa_{\calU}}^r)_{\frakm_x}^{\leq h}, R_{\calU, \frakm_{\kappa}})$.
    \item[(Adj4)] We have isomorphisms of $\bbT_{\calU, h, \frakm_x}^{(1)}$-modules \[
         H_{\Par}^i(Y_{\Iw_G}, D_{\kappa_{\calU}}^r)^{\leq h}_{\frakm_x} \simeq \bbT_{\calU, h, \frakm_x}^{(1), \vee}
    \] for $i=1, 2$. Here, the $\bbT_{\calU, h, \frakm_x}^{(1)}$-module structure on $H_{\Par}^2(Y_{\Iw_G}, D_{\kappa_{\calU}}^r)^{\leq h}_{\frakm_x}$ is given by the injection provided in (Adj3).\footnote{ More precisely, the $\bbT_{\calU, h, \frakm_x}^{(1)}$-module structure on $H_{\Par}^2(Y_{\Iw_G}, D_{\kappa_{\calU}}^r)^{\leq h}_{\frakm_x}$ can be described explicitly as follows. Firstly, since $\bbT_{\calU, h, \frakm_x}^{(1)}$ is generated by the image of the Hecke-operators, we only have to describe how the images of Hecke-operators act on $H_{\Par}^2(Y_{\Iw_G}, D_{\kappa_{\calU}}^r)^{\leq h}_{\frakm_x}$. Let $T^{(i)}$ be the image of a Hecke operator (\emph{i.e.}, $T_{\bfdelta}$ or $U_{\frakp}$ or $U_{\overline{\frakp}}$) in $\bbT_{\calU, h, \frakm_x}^{(i)}$ (for $i=1, 2$) and let $[\mu]$ be a class in  $H_{\Par}^2(Y_{\Iw_G}, D_{\kappa_{\calU}}^r)^{\leq h}_{\frakm_x}$. Then, via the injection $H_{\Par}^2(Y_{\Iw_G}, D_{\kappa_{\calU}}^r)_{\frakm_x}^{\leq h} \hookrightarrow \Hom_{R_{\calU, \frakm_{\kappa}}}(H_{\Par}^1(Y_{\Iw_G}, D_{\kappa_{\calU}}^r)_{\frakm_x}^{\leq h}, R_{\calU, \frakm_{\kappa}})$, the action of $T^{(1)}$ on $[\mu]$ can be described by \[
        (T^{(1)} * [\mu])([\mu']) := \left[T^{(1)}([\mu']), [\mu]\right]_{\kappa_{\calU}, \frakm_x}
    \] for any $[\mu']\in H_{\Par}^1(Y_{\Iw_G}, D_{\kappa_{\calU}}^r)^{\leq h}_{\frakm_x}$. However, since $[\cdot, \cdot]_{\kappa_{\calU}, \frakm_x}$ is Hecke-equivariant, \[
        \left[T^{(1)}([\mu']), [\mu]\right]_{\kappa_{\calU}, \frakm_x} = \left[[\mu'], T^{(2)}([\mu])\right]_{\kappa_{\calU}, \frakm_x}
    \] and thus $T^{(1)} * [\mu] = T^{(2)}([\mu])$.
    
    In fact, we shall see in \S \ref{subsection: non-vanishing; p-adic adjoint L-functions} below that, under certain assumptions, one can show $\bbT_{\calU, h, \frakm_x}^{(1)} \simeq R_{\calU, \frakm_{\kappa}} \simeq \bbT_{\calU, h, \frakm_x}^{(2)}$ canonically. }
\end{enumerate}

Under these assumptions, we may then apply the formalism developed in \cite[\S 9.1]{Eigen} to conclude the following. Let $\calX$ be the connected component of $\wt^{-1}(\calU)$ that contains $x$ and let $e_{\calX}$ be the idempotent in $\bbT_{\calU, h}^{(1)}$ such that $\calX$ is defined by $e_{\calX}=1$, \emph{i.e.,} $e_{\calX}\bbT_{\calU, h}^{(1)}$ is the ring of global sections of $\calX$. Then $[\cdot, \cdot]_{\kappa_{\calU}}$ defines an ideal $\frakL^{\adj}_{\calX,\calU, h} =\frakL^{\adj}_{\calU, h} \subset e_{\calX}\bbT_{\calU, h}^{(1)}$, which becomes a principal ideal in $\bbT_{\calU, h, \frakm_x}^{(1)}$ after localising at $\frakm_x$  (\cite[Corollary 9.1.13]{Eigen}). In particular, after possibly further shrinking $\calU$, $\frakL^{\adj}_{\calU, h} \subset e_{\calX}\bbT_{\calU, h}^{(1)}$ is a principal ideal. We pick a generator $L_{\calX, \calU, h}^{\adj} =L_{\calU, h}^{\adj}\in \frakL_{\calU, h}^{\adj}$ and term it a \textbf{\textit{$p$-adic adjoint $L$-function of $x$}}, varying in a finite-slope family of weight $\kappa_{\calU}$. 

The following theorem is an immediate consequence of \cite[Theorem 9.1.3]{Eigen} (see also \cite[Theorem 2.7]{Auslander-Buchsbaum}).

\begin{Theorem}\label{Theorem: etaleness of the weight map and the p-adic adjoint L-function}
    Keep the notation as above and assume (Adj1) -- (Adj4) are satisfied. Then, the weight map $\wt$ is étale at $x$ if and only if \[
        L_{\calU, h}^{\adj}(x) \neq 0.
    \]
    Here, $L_{\calU, h}^{\adj}(x)$ denotes the image of $L_{\calU, h}^{\adj}$ in the residue field $F_x$.
\end{Theorem}

\subsection{A non-vanishing result for \texorpdfstring{$p$}{p}-adic adjoint \texorpdfstring{$L$}{L}-functions}\label{subsection: non-vanishing; p-adic adjoint L-functions}

The goal of this subsection is to provide a more precise choice of the point $x\in \calE(F_x)$ in \S \ref{subsection: p-adic adjoint L-functions varying in finie-slope families}. To this end, let $f$ be a non-critical classical cuspidal Bianchi eigenform of cohomological parallel weight $k$ such that $k\neq (0,0)$. Suppose that $f$ defines a point $x_{f} \in \calE(F_f)$. Suppose $f$ varies in a family over an affinoid curve $\calU \subset \calW$. Let $(R_{\calU}, \kappa_{\calU})$ be the universal weight of $\calU$. Suppose $(\calU, h)$ is slope-adapted and Assumption \ref{Assumption: vary in 1-dimensional family} is satisfied (in particular, if $\calU \subset \calW_{\circ}$ and $f$ is a base-change form, Assumption \ref{Assumption: vary in 1-dimensional family} is satisfied by Corollary \ref{Corollary: parallel family behaves well}).  We denote by $\frakm_f$ the maximal ideal of $\bbT_{\Iw_G}\otimes_{\Z_p}R_{\calU}$ defined by $f$.

The following result is presumably well-known to experts. However, we lack the knowledge of its reference and so we record its proof here.

\begin{Lemma}\label{Lemma: concentration of the cohomology degree}
    Keep the notation as above. \begin{enumerate}
        \item[(i)] The cohomology groups $H^i(Y_{\Iw_G}, V_k^{\vee}(\Q_p))$ and $H_c^i(Y_{\Iw_G}, V_k^{\vee}(\Q_p))$ are concentrated in degree $1$ and $2$.  
        \item[(ii)]  Suppose the generalised $\bbT_{\Iw_G}$-eigenspace $H_c^i(Y_{\Iw_G}, V_k^{\vee}(F_f))_{\frakm_f}$ is one-dimensional for $i=1,2$. Then, the natural map $H_c^i(Y_{\Iw_G}, V_k^{\vee}(F_f))_{\frakm_f} \rightarrow H^i(Y_{\Iw_G}, V_k^{\vee}(F_f))_{\frakm_f}$ is an isomorphism.
    \end{enumerate}
\end{Lemma}
\begin{proof}
    To show (i), it is enough to show the assertions after the base change to $\C$. Recall that we have a long exact sequence \[
        \cdots \rightarrow H_c^i(Y_{\Iw_G}, V_k^{\vee}(\C)) \rightarrow H^i(Y_{\Iw_G}, V_k^{\vee}(\C)) \rightarrow H_{\partial}^i(Y_{\Iw_G}, V_k^{\vee}(\C)) \rightarrow \cdots,
    \] the $H_{\partial}^i$'s are computed using the boundary cochain complex introduced in \S \ref{subsection: overconvergent cohomology}. This long exact sequence gives rise to short exact sequences \[
        0 \rightarrow H_{\Par}^i(Y_{\Iw_G}, V_k^{\vee}(\C)) \rightarrow H^i(Y_{\Iw_G}, V_k^{\vee}(\C)) \rightarrow H_{\mathrm{Eis}}^i(Y_{\Iw_G}, V_k^{\vee}(\C)) \rightarrow 0.
    \]

    When $i=0, 3$, Theorem \ref{Theorem: ESH isomorphism} (i) yields $H_{\Par}^i(Y_{\Iw_G}, V_k^{\vee}(\C)) = 0$ and so we have an isomorphism \[
        H^i(Y_{\Iw_G}, V_k^{\vee}(\C)) \simeq H_{\mathrm{Eis}}^i(Y_{\Iw_G}, V_k^{\vee}(\C)).
    \]
    If $i=0$, since $k\neq (0,0)$, the second vanishing result in Theorem \ref{Theorem: ESH isomorphism} (i) implies that $H^0(Y_{\Iw_G}, V_k^{\vee}(\C)) = 0$. If $i=3$, since $\dim_{\R} \partial X_{\Iw_G}^{\mathrm{BS}} <3$, we see that $H^3_{\partial}(Y_{\Iw_G}, V_k^{\vee}(\C)) = 0$ and so $H^3_{\mathrm{Eis}}(Y_{\Iw_G}, V_k^{\vee}(\C))  = 0$, which implies that $H^3(Y_{\Iw_G}, V_k^{\vee}(\C)) = 0$. Then, by Poincaré duality (see also the proof of Proposition \ref{Proposition: non-degeneracy of the pairing}), one concludes that $H_c^i(Y_{\Iw_G}, V_k^{\vee}(\C)) =0$. This shows (i).
    
    We turn our attention to (ii). By assumption and the Hecke-equivariance of the Poincaré duality (Proposition \ref{Proposition: pairing on cohomology groups is Hecke-equivariant}), we know that $H^i(Y_{\Iw_G}, V_k^{\vee}(F_f))_{\frakm_f}$ is also one-dimensional for $i=1, 2$. Since localisation is exact, we have an exact sequence \[
        H_{\partial}^{i-1}(Y_{\Iw_G}, V_k^{\vee}(F_f))_{\frakm_f} \rightarrow H_{c}^i(Y_{\Iw_G}, V_k^{\vee}(F_f))_{\frakm_f} \rightarrow H^i(Y_{\Iw_G}, V_k^{\vee}(F_f))_{\frakm_f} \rightarrow H_{\partial}^i(Y_{\Iw_G}, V_k^{\vee}(F_f))_{\frakm_f}.
    \] 
    Since $f$ appears in $H^i_{\Par}(Y_{\Iw_G}, V_k^{\vee}(F_f))$, we know that the middle morphism is nonzero. Then, by dimensional reason and the exactness of the sequence, we conclude the desired isomorphism. 
\end{proof}

In what follows, we work under the assumption in Lemma \ref{Lemma: concentration of the cohomology degree} (ii).
Remark that this assumption holds when $f$ is a regular $p$-stabilised newform (see also \cite[(4.1)]{BW}).

\begin{Assumption}\label{Assumption: multiplicity one on degree 1 and 2}
    The generalised $\bbT_{\Iw_G}$-eigenspace $H_c^i(Y_{\Iw_G}, V_k^{\vee}(F_f))_{\frakm_f}$ is one-dimensional for $i=1, 2$.
\end{Assumption}

A key statement of the non-vanishing result for $p$-adic adjoint $L$-functions is the following proposition.

\begin{Proposition}\label{Proposition: structure theorem}
    Keep the notations and assumptions as above; in particular, Assumptions \ref{Assumption: vary in 1-dimensional family} and \ref{Assumption: multiplicity one on degree 1 and 2} are imposed. For $i=1, 2$, the $\bbT_{\calU, h, \frakm_f}^{(i)}$-module $H_{\Par}^i(Y_{\Iw_G}, D_{\kappa_{\calU}}^r)^{\leq h}_{\frakm_f}$ is free of rank $1$. 
\end{Proposition}

Before proving the proposition, recall the specialisation map $\mathrm{sp}_k: D_{\kappa_{\calU}}^r \twoheadrightarrow D_k^r$. We need the following lemma. 

\begin{Lemma}\label{Lemma: H3 with coefficients in ker(spk) vanishes}
    The finite-slope cohomology group $H^3(Y_{\Iw_G}, \ker \mathrm{sp}_k)^{\leq h}_{\frakm_f}$ vanishes. 
\end{Lemma}
\begin{proof}
    In this proof, we consider $R_{\calU}\widehat{\otimes}_{\Q_p}F_f$ as well as $D_{\kappa_{\calU}}^r\widehat{\otimes}_{R_{\calU}}R_{\calU}\widehat{\otimes}_{\Q_p}F_f = D_{\kappa_{\calU}}^r \widehat{\otimes}_{\Q_p}F_f$; however, to simplify the notation, we still denote them by $R_{\calU}$ and $D_{\kappa_{\calU}}^r$ respectively. Let $\frakm_k$ be the maximal ideal of $R_{\calU}$ defined by the weight $k$ and so $R_{\calU}/\frakm_k = F_f$. Note that the localisation $R_{\calU, \frakm_k}$ is a regular local ring, hence a unique factorisation domain. This means that all height-one prime ideals of $R_{\calU, \frakm_k}$ are principal (\cite[\href{https://stacks.math.columbia.edu/tag/0AFT}{Tag 0AFT}]{stacks-project}) and so we may pick a generator $a_k\in \frakm_k R_{\calU, \frakm_k}$. Up to shrinking $\calU$, we may assume $a_k\in R_{\calU}$. Consequently, we have an exact sequence \[
        D_{\kappa_{\calU}}^r \xrightarrow{\cdot a_k} D_{\kappa_{\calU}}^r \xrightarrow{\mathrm{sp}_k} D_{k}^r \rightarrow 0.
    \] We claim that multiplication by $a_k$ is an injection. Indeed, by dualising the one-variable version of Corollary \ref{Corollary: R-valued analytic functions}, we see that $D_{\kappa_{\calU}}^r \simeq \left(\prod_{i\in \Z_{\geq 0}} R_{\calU}^{\circ}\right)[1/p]$. In particular, $D_{\kappa_{\calU}}^r$ is flat over $R_{\calU}$ and so the claim. 

    Therefore, by taking cohomology and localisation, one obtains an exact sequence \begin{equation}\label{eq: exact seq. of specialisation map at degree 3}
        \begin{tikzcd}[row sep = tiny]
            H^3(Y_{\Iw_G}, D_{\kappa_{\calU}}^r)_{\frakm_f}^{\leq h} \arrow[r, "\cdot a_k"]\arrow[d, equal] & H^3(Y_{\Iw_G}, D_{\kappa_{\calU}}^r)_{\frakm_f}^{\leq h} \arrow[r] & H^3(Y_{\Iw_G}, D_k^r)_{\frakm_f}^{\leq h}\\
            H^3(Y_{\Iw_G}, \ker \mathrm{sp}_k)_{\frakm_f}^{\leq h}
        \end{tikzcd}.
    \end{equation}
    By control theorem and Lemma \ref{Lemma: concentration of the cohomology degree}, we know that \[
        H^3(Y_{\Iw_G}, D_k^r)_{\frakm_f}^{\leq h} \simeq H^3(Y_{\Iw_G}, V_k^{\vee}(F_f))_{\frakm_f}^{\leq h} = 0.
    \]
    On the other hand, we have an injection \[
        H^3(Y_{\Iw_G}, D_{\kappa_{\calU}}^r)_{\frakm_f}^{\leq h} \otimes_{R_{\calU, \frakm_k}} R_{\calU, \frakm_k}/\frakm_k \hookrightarrow H^3(Y_{\Iw_G}, D_k^r)_{\frakm_f}^{\leq h},
    \]
    which shows that $H^3(Y_{\Iw_G}, D_{\kappa_{\calU}}^r)_{\frakm_f}^{\leq h} \otimes_{R_{\calU, \frakm_k}} R_{\calU, \frakm_k}/\frakm_k  = 0$. Then, by Nakayama's lemma, one sees that $H^3(Y_{\Iw_G}, D_{\kappa_{\calU}}^r)_{\frakm_f}^{\leq h} = 0$. The desired result then follows from \eqref{eq: exact seq. of specialisation map at degree 3}.
\end{proof}

\begin{proof}[Proof of Proposition \ref{Proposition: structure theorem}]
    First of all, we claim that the natural morphisms \[
        H_c^i(Y_{\Iw_G}, D_{\kappa_{\calU}}^r)^{\leq h}_{\frakm_f} \twoheadrightarrow H_{\Par}^i(Y_{\Iw_G}, D_{\kappa_{\calU}}^r)^{\leq h}_{\frakm_f} \hookrightarrow H^i(Y_{\Iw_G}, D_{\kappa_{\calU}}^r)^{\leq h}_{\frakm_f}
    \] are isomorphisms. To show this, after applying Theorem \ref{Theorem: control theorem for overconvergent cohomology}, one observes that we have a commutative diagram \[
        \begin{tikzcd}[column sep = small]
            \scalemath{0.8}{H_{\partial}^{i-1}(Y_{\Iw_G}, D_{\kappa_{\calU}}^r)_{\frakm_f}^{\leq h}} \arrow[r]\arrow[d, two heads] & \scalemath{0.8}{H_{c}^{i}(Y_{\Iw_G}, D_{\kappa_{\calU}}^r)_{\frakm_f}^{\leq h} }\arrow[r]\arrow[d, two heads] & \scalemath{0.8}{H^{i}(Y_{\Iw_G}, D_{\kappa_{\calU}}^r)_{\frakm_f}^{\leq h} }\arrow[r]\arrow[d, two heads] & \scalemath{0.8}{H_{\partial}^{i}(Y_{\Iw_G}, D_{\kappa_{\calU}}^r)_{\frakm_f}^{\leq h} } \arrow[d, two heads]\\
            \scalemath{0.8}{H_{\partial}^{i-1}(Y_{\Iw_G}, D_{\kappa_{\calU}}^r)_{\frakm_f}^{\leq h} \otimes_{R_{\calU, \frakm_k}} F_f }\arrow[r]\arrow[d, hook] & \scalemath{0.8}{ H_{c}^{i}(Y_{\Iw_G}, D_{\kappa_{\calU}}^r)_{\frakm_f}^{\leq h} \otimes_{R_{\calU, \frakm_k}} F_f } \arrow[r]\arrow[d, hook] & \scalemath{0.8}{H^i(Y_{\Iw_G}, D_{\kappa_{\calU}}^r)_{\frakm_f}^{\leq h} \otimes_{R_{\calU, \frakm_k}}F_f }\arrow[r]\arrow[d, hook] & \scalemath{0.8}{H_{\partial}^{i}(Y_{\Iw_G}, D_{\kappa_{\calU}}^r)_{\frakm_f}^{\leq h}\otimes_{R_{\calU, \frakm_k}} F_f} \arrow[d, hook]\\
            \scalemath{0.8}{H_{\partial}^{i-1}(Y_{\Iw_G}, V_k^{\vee}(F_f))_{\frakm_f}^{\leq h} } \arrow[r] & \scalemath{0.8}{H_{c}^{i}(Y_{\Iw_G}, V_k^{\vee}(F_f))_{\frakm_f}^{\leq h}} \arrow[r] & \scalemath{0.8}{H^{i}(Y_{\Iw_G}, V_k^{\vee}(F_f))_{\frakm_f}^{\leq h} }\arrow[r] & \scalemath{0.8}{H_{\partial}^{i}(Y_{\Iw_G}, V_k^{\vee}(F_f))_{\frakm_f}^{\leq h}}
        \end{tikzcd},
    \]
    where \begin{enumerate} 
        \item[$\bullet$] the tensor products in the second raw is given via $R_{\calU, \frakm_k}/\frakm_k=F_f$; 
        \item[$\bullet$] the top and the bottom rows are exact.
    \end{enumerate} 
    Since $H_{c}^{i}(Y_{\Iw_G}, V_k^{\vee}(F_f))_{\frakm_f}^{\leq h} \rightarrow H^{i}(Y_{\Iw_G}, V_k^{\vee}(F_f))_{\frakm_f}^{\leq h}$ is an isomorphism by Lemma \ref{Lemma: concentration of the cohomology degree}, one sees that \[
        H_{\partial}^{i}(Y_{\Iw_G}, V_k^{\vee}(F_f))_{\frakm_f}^{\leq h} = 0
    \] for all $i$. This implies that \[
        H_{\partial}^{i}(Y_{\Iw_G}, D_{\kappa_{\calU}}^r)_{\frakm_f}^{\leq h}\otimes_{R_{\calU, \frakm_k}} F_f = 0
    \] for all $i$. By Nakayama's lemma, we conclude that \[
        H_{\partial}^{i}(Y_{\Iw_G}, D_{\kappa_{\calU}}^r)_{\frakm_f}^{\leq h} =0 .
    \] The desired isomorphisms then follow from the exactness of the top row.

    Since we have isomorphisms \[
        H_c^i(Y_{\Iw_G},  D_{\kappa_{\calU}}^r)^{\leq h}_{\frakm_f} \simeq H_{\Par}^i(Y_{\Iw_G}, D_{\kappa_{\calU}}^r)^{\leq h}_{\frakm_f} \simeq H^i(Y_{\Iw_G}, D_{\kappa_{\calU}}^r)^{\leq h}_{\frakm_f},
    \] we can now apply \cite[Theorem 4.5]{BW} and know that $H^1_{\Par}(Y_{\Iw_G}, D_{\kappa_{\calU}}^r)_{\frakm_f}^{\leq h}$ is free of rank one over $\bbT_{\calU, h, \frakm_f}^{(1)}$ (as well as over $R_{\calU, \frakm_k}$). Note that, from the proof of \cite[Proposition 4.6]{BW}, the natural map \[
         H^1_{\Par}(Y_{\Iw_G}, D_{\kappa_{\calU}}^r)_{\frakm_f}^{\leq h} \otimes_{R_{\calU, \frakm_f}} R_{\calU, \frakm_f}/\frakm_f \rightarrow H_{\Par}^1(Y_{\Iw_G}, D_k^r)_{\frakm_f}^{\leq h}
    \] is an isomorphism. 

    To show the desired result for $H^2_{\Par}(Y_{\Iw_G}, D_{\kappa_{\calU}}^r)_{\frakm_f}^{\leq h}$, we claim that the induced pairing \begin{equation}\label{eq: pairing after localisation}
        H^1_{\Par}(Y_{\Iw_G}, D_{\kappa_{\calU}}^r)_{\frakm_f}^{\leq h} \times H^2_{\Par}(Y_{\Iw_G}, D_{\kappa_{\calU}}^r)_{\frakm_f}^{\leq h} \rightarrow R_{\calU, \frakm_k}
    \end{equation} is non-degenerate. Note that we have a commutative diagram of exact sequences \[
        \begin{tikzcd}
            H^2(Y_{\Iw_G}, D_{\kappa_{\calU}}^r)_{\frakm_f}^{\leq h} \arrow[r] & H^2(Y_{\Iw_G}, D_{k}^r)_{\frakm_f}^{\leq h} \arrow[r] & H^3(Y_{\Iw_G}, \ker \mathrm{sp}_k)_{\frakm_f}^{\leq h} \\
            H_c^2(Y_{\Iw_G}, D_{\kappa_{\calU}}^r)_{\frakm_f}^{\leq h} \arrow[r]\arrow[u, "\simeq"] & H_c^2(Y_{\Iw_G}, D_{k}^r)_{\frakm_f}^{\leq h} \arrow[r]\arrow[u, "\simeq "] & H_c^3(Y_{\Iw_G}, \ker \mathrm{sp}_k)_{\frakm_f}^{\leq h}\arrow[u]
        \end{tikzcd} .
    \] Then, by Lemma \ref{Lemma: H3 with coefficients in ker(spk) vanishes}, one then concludes that \[
        H_{\Par}^2(Y_{\Iw_G}, D_{\kappa_{\calU}}^r)_{\frakm_f}^{\leq h} \rightarrow H_{\Par}^2(Y_{\Iw_G}, D_k^r)_{\frakm_f}^{\leq h}
    \] is a surjection. Since $H_{\Par}^1(Y_{\Iw_G}, D_k^r)_{\frakm_f}^{\leq h}$ is one-dimensional, one deduces from Corollary \ref{Corollary: non-degenercay of the pairing; non-critical case} and Nakayama's lemma that $H_{\Par}^2(Y_{\Iw_G}, D_{\kappa_{\calU}}^r)_{\frakm_f}^{\leq h}$ is generated by one element over $R_{\calU, \frakm_k}$. Now, let $e_f^{(i)}$ be a basis of $H_{\Par}^i(Y_{\Iw_G}, D_k^r)_{\frakm_f}^{\leq h}$ (for $i=1,2$) such that $[e_f^{(1)}, e_f^{(2)}]_k = 1$. We lift them to generators of $H^i_{\Par}(Y_{\Iw_G}, D_{\kappa_{\calU}}^r)_{\frakm_f}^{\leq h}$ (by Nakayama's lemma again) and still denote them by $e_f^{(i)}$. Then, we have \[
        [e_f^{(1)}, e_f^{(2)}]_{\kappa_{\calU}, \frakm_f} \equiv 1 \mod \frakm_k.
    \] Since both $e_f^{(1)}$ and $e_f^{(2)}$ are generators, this shows the non-degeneracy of \eqref{eq: pairing after localisation}. 
    
    Consequently, the natural morphism \begin{equation}\label{eq: H^2-->Hom(H^1, R)}
        H^2_{\Par}(Y_{\Iw_G}, D_{\kappa_{\calU}}^r)_{\frakm_f}^{\leq h} \xrightarrow{\eta} \Hom_{R_{\calU, \frakm_k}}(H^1_{\Par}(Y_{\Iw_G}, D_{\kappa_{\calU}}^r)^{\leq h}_{\frakm_f}, R_{\calU, \frakm_k}) \simeq R_{\calU, \frakm_k}
    \end{equation} is an injection, where $\eta([\mu]) = [-, [\mu]]_{\kappa_{\calU, \frakm_f}}$ and the last isomorphism is given by $\varphi \mapsto \varphi(e_f^{(1)})$. Since $[e_f^{(1)}, e_f^{(2)}]_{\kappa_{\calU}, \frakm_f} \equiv 1 \mod \frakm_k$, $u_f:=[e_f^{(1)}, e_f^{(2)}]_{\kappa_{\calU}, \frakm_f}$ is a unit in $R_{\calU, \frakm_k}$. Thus, for any $a\in R_{\calU, \frakm_k}$, \[
        a = au_f^{-1}[e_f^{(1)}, e_f^{(2)}]_{\kappa_{\calU}, \frakm_f}= \eta(u_f^{-1}ae_f^{(2)})(e_f^{(1)})
    \] and so \eqref{eq: H^2-->Hom(H^1, R)} is also a surjection and hence an isomorphism. This then yields an injection \[
        \bbT_{\calU, h, \frakm_f}^{(2)} \hookrightarrow \End_{R_{\calU, \frakm_k}}(H^2_{\Par}(Y_{\Iw_G}, D_{\kappa_{\calU}}^r)^{\leq h}_{\frakm_f}) \simeq R_{\calU, \frakm_k},
    \] which is then an isomorphism. 
\end{proof}

\begin{Corollary}\label{Corollary: weight map is étale at the chosen point}
    Keep the notation as above. The weight map $\wt$ is étale at $x_f$.
\end{Corollary}
\begin{proof}
    In the proof of Proposition \ref{Proposition: structure theorem} (see also \cite[Theorem 4.5]{BW}), we know that \[
        \bbT_{\calU, h, \frakm_f}^{(1)}  \simeq  R_{\calU, \frakm_k}, 
    \] which is obviously  étale over $R_{\calU, \frakm_k}$.
\end{proof}


\begin{Corollary}\label{Corollary: non-vanishing result for p-adic adjoint L-function}
    Keep the notation as above. Then, $L_{\calU, h}^{\adj}(x_f) \neq 0$.
\end{Corollary}
\begin{proof}
    By the proof of Proposition \ref{Proposition: structure theorem}, we know that (Adj1) -- (Adj4) are satisfied. Hence, by applying Theorem \ref{Theorem: etaleness of the weight map and the p-adic adjoint L-function} and Corollary \ref{Corollary: weight map is étale at the chosen point}, we conclude the desired result. 
\end{proof}

Recall the setting in \S \ref{subsection: p-adic adjoint L-functions varying in finie-slope families}. Let $\calX$ be the connected component of $\wt^{-1}(\calU)$ that contains $x_f$. Thanks to Proposition \ref{Proposition: structure theorem}, we may assume (after shrinking $\calU$ if necessary) that $R_{\calU} \simeq e_{\calX}\bbT_{\calU, h}^{(1)} = \scrO_{\calE_{\calU, h}}(\calX)$ and $e_{\calX}H_{\Par}^{i}(Y_{\Iw_G}, D_{\kappa_{\calU}}^r)^{\leq h}$ is free of rank $1$ over $R_{\calU}$ for $i=1, 2$. In this situation, we see from the construction of $\frakL_{\calU, h}^{\adj}$ (\cite[\S 9.1]{Eigen}) that it is generated by the image of the induced pairing \[
    [\cdot, \cdot]_{\kappa_{\calU}}: e_{\calX}H_{\Par}^{1}(Y_{\Iw_G}, D_{\kappa_{\calU}}^r)^{\leq h} \times e_{\calX}H_{\Par}^{2}(Y_{\Iw_G}, D_{\kappa_{\calU}}^r)^{\leq h} \rightarrow R_{\calU}.
\] In particular, by letting $[\mu_{\calX, i}]$ be generators $e_{\calX}H_{\Par}^{i}(Y_{\Iw_G}, D_{\kappa_{\calU}}^r)^{\leq h}$, the resulting pairing $[[\mu_{\calX, 1}], [\mu_{\calX, 2}]]_{\kappa_{\calU}}$ generates $\frakL_{\calU, h}^{\adj}$. Consequently, we may choose our $p$-adic adjoint $L$-function to be \begin{equation}\label{eq: natural choice of p-adic L-function}
    L_{\calU, h}^{\adj} := \left[[\mu_{\calX, 1}], [\mu_{\calX, 2}] \right]_{\kappa_{\calU}}. 
\end{equation}

Suppose $g$ is another Bianchi cuspidal eigenform of cohomological weight $k'$ that corresponds to a classical point $x_g$ in $\calX$, let $[\mu_{g, i}]\in H_{\Par}^i(Y_{\Iw_G}, V_{k'}^{\vee}(\C))$ be the image of $f$ under the Eichler--Shimura--Harder isomorphism. Write $[\mu_{\calX, i}](x_g)$ for the image of $[\mu_{\calX, i}]$ in $H_{\Par}^i(Y_{\Iw_G}, D_{k'}^{r})^{\leq h}$ after specialisation. Then, there are $p$-adic error terms $\Omega_{g, i}\in \C^\times\simeq \C_p^\times$ such that \[
    \Omega_{g, i}[\mu_{\calX, i}](x_g) = [\mu_{g, i}].
\]
Immediately from the functoriality of the pairing, we have the following corollary. 

\begin{Corollary}\label{Corollary: specific choice of p-adic adjoint L-function and its special value}
    Keep the notation as above and let $L_{\calU, h}^{\adj}$ be chosen as in \eqref{eq: natural choice of p-adic L-function}. Then, \[
        L_{\calU, h}^{\adj}(x_g) = \frac{\left[[\mu_{g, 1}], [\mu_{g, 2}]\right]_{k'}}{\Omega_{g, 1}\Omega_{g, 2}} .
    \]
\end{Corollary}

\subsection{\texorpdfstring{$p$}{p}-adic 
 adjoint \texorpdfstring{$L$}{L}-functions varying in ordinary families}\label{subsection: p-adic adjoint L-functions varying in ordinary families}

Let's now turn our attention to ordinary families. Let $(R_{\calU}, \kappa_{\calU})$ be a small weight such that the induced weight space $\calU = \Spa(R_{\calU}, R_{\calU})^{\rig} \hookrightarrow \calW$ is a curve. For $i\in \{0, 1, 2, 3\}$, we denote by $\bbT_{\calU, \ord}^{(i)}$ (resp., $\bbT_{\calU, \ord}^{(i), \circ}$) the reduced $R_{\calU}[1/p]$-subalgebra (resp., $R_{\calU}$-subalgebra) of $\End_{R_{\calU}[1/p]}(H_{\Par}^i(Y_{\Iw_G}, D_{\kappa_{\calU}}^r)^{\ord})$ (resp., $\End_{R_{\calU}}(H_{\Par}^i(Y_{\Iw_G}, D_{\kappa_{\calU}}^{r, \circ})^{\ord})$ ) generated by the image of $\bbT_{\Iw_G}$. 

Suppose $\frakm_x$ is a maximal ideal of $\bbT_{\calU, \ord}^{(1)}$ and we let $\frakm_{\kappa}$ be its inverse image in $R_{\calU}[1/p]$. We assume $\calU$ is smooth at $\kappa$ (Assumption \ref{Assumption: vary in 1-dimensional family}). In this situation, we look at the weight map \[
    \wt: \calE_{\calU, \ord} := \Spa(\bbT_{\calU, \ord}^{(1), \circ}, \bbT_{\calU, \ord}^{(1), \circ})^{\rig} \rightarrow \calU.
\] 
We assume furthermore that (Adj1) -- (Adj4) in this situation hold: \begin{enumerate}
    \item[(Adj1)] The localisation $\bbT_{\calU, \ord, \frakm_x}^{(1)}$ is flat over $R_{\calU}[1/p]_{\frakm_{\kappa}}$. 
    \item[(Adj2)] The dual algebra $\bbT_{\calU, \ord, \frakm_x}^{(1), \vee} := \Hom_{R_{\calU}[1/p]_{\frakm_{\kappa}}}(\bbT_{\calU, \ord, \frakm_x}^{(1)}, R_{\calU}[1/p]_{\frakm_{\kappa}})$ is free of rank 1 over $\bbT_{\calU, \ord, \frakm_x}^{(1)}$. 
    \item[(Adj3)] The induced pairing \[
        [\cdot, \cdot]_{\kappa_{\calU}, \frakm_x} : H_{\Par}^1(Y_{\Iw_G}, D_{\kappa_{\calU}}^r)_{\frakm_x}^{\ord} \times H_{\Par}^2(Y_{\Iw_G}, D_{\kappa_{\calU}}^r)_{\frakm_x}^{\ord} \rightarrow R_{\calU}[1/p]_{\frakm_{\kappa}}
    \] is non-degenerate.
    \item[(Adj4)] We have an isomorphism of $\bbT_{\calU, \ord, \frakm_x}^{(1)}$-modules \[
        H_{\Par}^i(Y_{\Iw_G}, D_{\kappa_{\calU}}^r)_{\frakm_x}^{\ord} \simeq \bbT_{\calU, \ord, \frakm_x}^{(1), \vee}
    \] for $i=1, 2$.
\end{enumerate}

Under these assumptions, since the formalism in \cite[\S 9.1]{Eigen} is purely algebraic, let $\calX$ be the connected component of $\wt^{-1}(\calU)$ that contains $x$ and let $e_{\calX}\bbT_{\calU, \ord}^{(1)} = \scrO_{\calE_{\calU, \ord}}^+(\calX)[1/p]$, one then again obtains an ideal $\frakL_{\calX, \calU, \ord}^{\adj} = \frakL_{\calU, \ord}^{\adj}\subset e_{\calX}\bbT_{\calU, \ord}^{(1)}$, which again becomes a principal ideal in $\bbT_{\calU, \ord, \frakm_x}^{(1)}$ after localising at $\frakm_x$. Hence, after possibly shrinking $\calU$, $\frakL_{\calU, \ord}^{\adj}\subset e_{\calX}\bbT_{\calU, \ord}^{(1)}$ is a principal ideal. We again pick a generator $L_{\calU, \ord}^{\adj} \in \frakL_{\calU, \ord}^{\adj}$ and call it a $p$-adic adjoint $L$-function of $x$, varying in a ordinary family of weight $\kappa_{\calU}$.

\begin{Theorem}\label{Theorem: etaleness of the weight map and the p-adic adjoint L-function; ordinary case}
    Keep the notation as above and assume (Adj1) -- (Adj4) are satisfied. Then, the natural morphism $R_{\calU}[1/p]_{\frakm_{\kappa}} \rightarrow \bbT_{\calU, \ord, \frakm_x}^{(1)}$ is étale if and only if $L_{\calU, \ord}^{\adj} \neq 0$ in $\bbT_{\calU, \ord, \frakm_x}^{(1)}/\frakm_x$.
\end{Theorem}

As before, we would like to exhibit a more precise choice of $\frakm_x$ such that (Adj1) -- (Adj4) hold. To this end, let $f\in S_k(Y_{\Iw_G})$ such that it is a $\bbT_{\Iw_G}$-eigenform and its image in $H_{\Par}^1(Y_{\Iw_G}, V_k^{\vee}(\C))$ lies in $H_{\Par}^1(Y_{\Iw_G}, V_k^{\vee}(F_f))^{\ord}$. We shall again assume that $k \neq (0,0)$. We denote by $\frakm_f$ the maximal ideal of $\bbT_{\Iw_G} \otimes_{\Z_p}R_{\calU}[1/p]$ defined by $f$. We also assume Assumption \ref{Assumption: multiplicity one on degree 1 and 2} holds for $f$.

\begin{Proposition}\label{Proposition: structure theorem; ordinary case}
    Keep the notation as above. For $i=1, 2$, the $\bbT_{\calU, \ord, \frakm_f}^{(i)}$-module $H_{\Par}^i(Y_{\Iw_G}, D_{\kappa_{\calU}}^r)^{\ord}_{\frakm_f}$ is free of rank $1$. 
\end{Proposition}
\begin{proof}
    Instead of verifying that our previous method goes verbatim, we make use of our knowledge of Proposition \ref{Proposition: structure theorem} to deduce the desired result. 
    
    Let $(R_{\calV}, \kappa_{\calV})$ be an affinoid weight such that the associated weight space $\calV$ is an affinoid open subset in $\calU$ and $f$ defines a point $x_{f}\in \calV$. We still denote by $\frakm_f$ the corresponding maximal ideal in $R_{\calV}$. By construction, we have a natural morphism \[
        \Psi_{\calV, \frakm_f}^{\ord}:H_{\Par}^i(Y_{\Iw_G}, D_{\kappa_{\calU}}^r)_{\frakm_f}^{\ord} \rightarrow H_{\Par}^i(Y_{\Iw_G}, D_{\kappa_{\calV}}^r)_{\frakm_f}^{\leq h},
    \] for any $h<k+1$. By control theorem and Nakayama's lemma, we know that the left-hand side is a finitely generated $R_{\calU}[1/p]_{\frakm_k}$-module while the right-hand side is a free $R_{\calV, \frakm_k}$-module of rank $1$ (by Proposition \ref{Proposition: structure theorem}). Observe that both $R_{\calU}[1/p]_{\frakm_k}$ and $R_{\calV, \frakm_k}$ are the local rings at $\frakm_k$, thus they agree to each other. We claim that the natural map above is an isomorphism of $R_{\calU}[1/p]_{\frakm_k}$-modules. 

    The injectivity of $\Psi_{\calV ,\frakm_f}^{\ord}$ is easy. By construction, we have an injection \[
        e_{\ord}H_{\Par}^i(Y_{\Iw_G}, D_{\kappa_{\calU}}^{r, \circ})\otimes_{R_{\calU}} R_{\calV} \hookrightarrow H_{\Par}^i(Y_{\Iw_G}, D_{\kappa_{\calV}}^r)^{\leq h}.
    \] Since localisation is exact, we see the desired injectivity. 

    Let's now show the surjectivity of $\Psi_{\calV, \frakm_f}^{\ord}$. We have a commutative diagram \[
        \begin{tikzcd}
            H_{\Par}^i(Y_{\Iw_G}, D_{\kappa_{\calU}}^r)_{\frakm_f}^{\ord} \arrow[r, "\Psi_{\calV, \frakm_f}^{\ord}"]\arrow[d, two heads] & H_{\Par}^i(Y_{\Iw_G}, D_{\kappa_{\calV}}^r)_{\frakm_f}^{\leq h} \arrow[r] \arrow[d, two heads] & \coker\Psi_{\calV, \frakm_f}^{\ord}\arrow[d, two heads]\\
            H_{\Par}^i(Y_{\Iw_G}, D_{\kappa_{\calU}}^r)_{\frakm_f}^{\ord} \otimes_{R_{\calU}[1/p]_{\frakm_f}, k} F_f \arrow[r]\arrow[d, hook] & H_{\Par}^i(Y_{\Iw_G}, D_{\kappa_{\calV}}^r)_{\frakm_f}^{\leq h} \otimes_{R_{\calV, \frakm_f}, k}F_f\arrow[r]\arrow[d, equal] & (\coker\Psi_{\calV, \frakm_f}^{\ord})\otimes_{R_{\calV, \frakm_f}, k}F_f\arrow[d, hook]\\
            H_{\Par}^i(Y_{\Iw_G}, V_k^{\vee}(F_f))_{\frakm_f}^{\ord} \arrow[r, "\Psi_{k, \frakm_f}^{\ord}"] & H_{\Par}^{i}(Y_{\Iw_G}, V_k^{\vee}(F_f))_{\frakm_f}^{\leq h} \arrow[r] & \coker\Psi_{k, \frakm_f}^{\ord}
        \end{tikzcd},
    \]
    where $H^i_{\Par}(Y_{\Iw_G}, V_k^{\vee}(F_f))^{\ord} := \left( e_{\ord}H_{\Par}^i(Y_{\Iw_G}, V_k^{\vee}(\calO_{F_f})) \right)[1/p]$ and the map $\Psi_{k, \frakm_f}^{\ord}$ is defined similarly as above. Note that the equation in the diagram follows from Assumption \ref{Assumption: multiplicity one on degree 1 and 2}. It also implies that $H_{\Par}^i(Y_{\Iw_G}, V_k^{\vee}(F_f))_{\frakm_f}^{\ord}$ is of dimension $1$ and so $\Psi_{k, \frakm_f}^{\ord}$ is an isomorphism. Consequently, $(\coker\Psi_{\calV, \frakm_f}^{\ord})\otimes_{R_{\calV, \frakm_f}, k}F_f=0$. Since the top row gives rise to a short exact sequence, the sequence \[
         H_{\Par}^i(Y_{\Iw_G}, D_{\kappa_{\calU}}^r)_{\frakm_f}^{\ord} \otimes_{R_{\calU}[1/p]_{\frakm_f}, k} F_f \xrightarrow{\Psi_{\calV, \frakm_f}^{\ord} \otimes_k F_f} H_{\Par}^i(Y_{\Iw_G}, D_{\kappa_{\calV}}^r)_{\frakm_f}^{\leq h} \otimes_{R_{\calV, \frakm_f}, k}F_f \rightarrow (\coker\Psi_{\calV, \frakm_f}^{\ord})\otimes_{R_{\calV, \frakm_f}, k}F_f \rightarrow 0
    \] is exact. The vanishing of $(\coker\Psi_{\calV, \frakm_f}^{\ord})\otimes_{R_{\calV, \frakm_f}, k}F_f$ implies that $\Psi_{\calV, \frakm_f}^{\ord} \otimes_k F_f$ is surjective. Then, by Nakayama's lemma, one sees that $\Psi_{\calV, \frakm_f}^{\ord}$ is also surjective. 

    Finally, we see from the proof of Proposition \ref{Proposition: structure theorem}, $\bbT_{\calV, h, \frakm_f}^{(i)} \simeq R_{\calV, \frakm_k}$ and the morphism above is Hecke-equivariant, we conclude that $\bbT_{\calU, \ord, \frakm_f}^{(i)} \simeq \bbT_{\calV, h, \frakm_f}^{(i)} \simeq R_{\calV, \frakm_k}$, which conclude the assertion. 
\end{proof}

Immediately, we have the following corollary. 

\begin{Corollary}\label{Corollary: non-vanishing result for p-adic adjoint L-function; ordinary case}
    Keep the notation as above. Then, the natural morphism $R_{\calU}[1/p]_{\frakm_{k}} \rightarrow \bbT_{\calU, \ord, \frakm_f}^{(1)}$ is étale and $L_{\calU, \ord}^{\adj} \neq 0$ in $\bbT_{\calU, \ord, \frakm_f}^{(1)}/\frakm_f$.
\end{Corollary}

\begin{Remark}
    We point out that in the construction of the $p$-adic adjoint $L$-functions varying in ordinary families, we work with \emph{small weights} instead of \emph{affinoid weights}. If $(R_{\calV}, \kappa_{\calV})$ is an affinoid weight and $(R_{\calU}, \kappa_{\calU})$ is a small weight such that $\calV$ is an affinoid open in $\calU$, then we have injections \[
        H_{\Par}^{i}(Y_{\Iw_G}, D_{\kappa_{\calU}}^r)^{\ord}\otimes_{R_{\calU}[1/p]}R_{\calV} \hookrightarrow H_{\Par}^i(Y_{\Iw_G}, D_{\kappa_{\calV}}^r)^{\leq 0}.
    \] Assuming $f$ is as above, these injections become isomorphisms after localising at $\frakm_f$ (by the proof of Proposition \ref{Proposition: structure theorem; ordinary case}). One can then easily deduce that, after shrinking $\calU$ if necessary, \[
        \frakL_{\calU, \ord}^{\adj} \otimes_{R_{\calU}[1/p]}R_{\calV} = \frakL_{\calV, 0}^{\adj}.
    \]
    In this subsection, we spell out the construction of the $p$-adic adjoint $L$-functions for ordinary families over small weights to resonate with the work of Hida. The discussions on the ordinary families in this paper shall also clarify how one can expect a comparison between the $p$-adic adjoint $L$-functions constructed in the present work and the $p$-adic $L$-functions constructed in \cite{Lee-PhD}.
\end{Remark}

\section{Adjoint \texorpdfstring{$L$}{L}-values}\label{section: adjoint L-values}

In the previous section, we have seen how the adjoint pairing $[\cdot, \cdot]_{\kappa_{\calU}}$ defines a \emph{$p$-adic adjoint $L$-function}. The aim of this section is to justify this name. More precisely, we compute the relation between $[\cdot, \cdot]_{\kappa_{\calU}}$ and the adjoint $L$-value. We remark in the beginning that our strategy follows from the strategy in \cite[\S 9.5]{Eigen}.

\subsection{Pairing and the \texorpdfstring{$p$}{p}-stabilisations}\label{subsection: pairing and p-stabilisations}
To simplify the notation, we assume in this subsection that $f$ is a Bianchi cuspidal eigenform of level $1$ and of cohomological parallel weight $k= (k, k) \neq (0,0)$.\footnote{ If $f$ has level $\Gamma_0(\frakn)$ with $\frakn$ nontrivial, the computation below is similar but the notation in the computation will be more complicated. We leave it to the interested readers. } We denote by $[\mu_{f, i}]$ its image in the degree-$i$ parabolic cohomology via the Eichler--Shimura--Harder isomorphism (Theorem \ref{Theorem: ESH isomorphism}). The goal of this subsection is to understand the relation between $\langle [\mu_{f, 1}], [\mu_{f, 2}]\rangle_k$ and $\left[[\mu_{f, 1}]^{(p)}_{(i,j)}, [\mu_{f, 2}]^{(p)}_{(i,j)}\right]_k$, where $[\mu_{f, i}]^{(p)}_{(i,j)}$ is the $p$-stabilisation of $[\mu_{f, i}]$ whose $U_p$-eigenvalue is $\alpha_{\frakp}^{(i)}\alpha_{\overline{\frakp}}^{(j)}$ (for $(i,j)\in (\Z/2\Z)^2$, see Proposition \ref{Proposition: p-stabilisation}).

\begin{Remark}\label{Remark: Hecke eigenvalues are all real}
    Since we have assumed $f$ is of level $1$, by \cite[Proposition 6.4]{Urban-PhD}, we know that $f$ agrees with its complex conjugation. Moreover, by applying Lemma 6.1 and Proposition 6.2 of \emph{op. cit.}, we know that the Hecke-eigenvalues of $f$ are all real. 
\end{Remark}

Several steps of preparation are required. We begin with the following lemmas.

\begin{Lemma}\label{Lemma: matrix representation of Hecke operators at p}
    For $t=1, 2$, let \[
        W_{f, t} = \mathrm{span}_{\C}\left\{ [\mu_{f, i}], \bfupsilon_{\frakp}^* *[\mu_{f, i}], \bfupsilon_{\overline{\frakp}}^* *[\mu_{f, i}], \bfupsilon_{\frakp}^* \bfupsilon_{\overline{\frakp}}^* *[\mu_{f, i}]\right\}.
    \] Then, the linear action of $U_{\frakp}$ (resp., $U_{\overline{\frakp}}$) on $W_{f, t}$ has the matrix representation  \[
        U_{\frakp} = \begin{pmatrix}\lambda_{\frakp} & p^{k+1} & & \\ -1 &&& \\ && \lambda_{\frakp} & p^{k+1}\\ && -1\end{pmatrix} \quad \left( \text{resp., } U_{\overline{\frakp}} = \begin{pmatrix} \lambda_{\overline{\frakp}} & & p^{k+1} & \\ & \lambda_{\overline{\frakp}} & & p^{k+1}\\ -1 &&& \\ & -1 &&\end{pmatrix}\right),
    \] where $\lambda_{f, \frakp} = \lambda_{\frakp}$ (resp., $\lambda_{f, \overline{\frakp}} = \lambda_{\overline{\frakp}}$) is the $T_{\bfupsilon_{\frakp}}$-eigenvalue (resp., $T_{\bfupsilon_{\overline{\frakp}}}$-eigenvalue) of $[\mu_{f, i}]$. Consequently, the matrix representation for $U_p$ is given by \[
        U_p = \begin{pmatrix} \lambda_{\frakp}\lambda_{\overline{\frakp}} & p^{k+1}\lambda_{\overline{\frakp}} & \lambda_{\frakp}p^{k+1} & p^{2(k+1)}\\ -\lambda_{\overline{\frakp}} & & -p^{k+1} & \\ -\lambda_{\frakp} & -p^{k+1} & & \\ 1 & & & \end{pmatrix}.
    \]
\end{Lemma}
\begin{proof}
    We show the result for $U_{\frakp}$ as similar computation applies to $U_{\overline{\frakp}}$. 
    
    By Remark \ref{Remark: coset decomposition for Tp} and the proof of Proposition \ref{Proposition: p-stabilisation}, we have\[
        U_{\frakp} : \begin{array}{rl}
            [\mu_{f, i}] & \mapsto \lambda_{\frakp}[\mu_{f, i}] - \bfupsilon_{\frakp}^* *[\mu_{f, i}]  \\
            \bfupsilon_{\frakp}^* * [\mu_{f, i}] & \mapsto p^{k+1}[\mu_{f, i}] 
        \end{array}.
    \] Observe that the $U_{\frakp}$-action commutes with the action by $\bfupsilon_{\overline{\frakp}}$, we have \[
        U_{\frakp}: \begin{array}{rl}
            \bfupsilon_{\overline{\frakp}}^* *[\mu_{f, i}] & \mapsto \lambda_{\frakp} \bfupsilon_{\overline{\frakp}}^* *[\mu_{f, i}] - \bfupsilon_{\overline{\frakp}}^* \bfupsilon_{\frakp}^* [\mu_{f, i}]  \\
            \bfupsilon_{\frakp}^*\bfupsilon_{\overline{\frakp}}^* * [\mu_{f, i}] & \mapsto p^{k+1}\left(\bfupsilon_{\overline{\frakp}}^* [\mu_{f, i}] \right)
        \end{array}.
    \]
    The desired matrix representation follows.
\end{proof}

\begin{Lemma}\label{Lemma: matrix representation for AL operators at p}
    Define \[
        \bfomega_{\frakp} = \left(\begin{pmatrix} & -p\\ 1\end{pmatrix}, \begin{pmatrix}1 & \\ & 1\end{pmatrix}\right) \quad \left(\text{resp., } \bfomega_{\overline{\frakp}} = \left(\begin{pmatrix}1 & \\ & 1\end{pmatrix}, \begin{pmatrix} & -p\\ 1\end{pmatrix}\right)\right).
    \] Then, the linear action of $\bfomega_{\frakp}$ (resp., $\bfomega_{\overline{\frakp}}$) on $W_{f, t}$ has the matrix representation \[
        \bfomega_{\frakp} = \begin{pmatrix} & p^k & & \\ (-1)^k &&& \\ && & p^k \\ && (-1)^k&\end{pmatrix} \quad \left(\text{resp., }\bfomega_{\overline{\frakp}} = \begin{pmatrix} & & p^k & \\  &&&p^k \\  (-1)^k  &&&\\ & (-1)^k&&\end{pmatrix} \right).
    \] Note that the Atkin--Lehner operator satisfies \[
        \bfomega_p = \bfomega_{\frakp}\bfomega_{\overline{\frakp}}.
    \] Thus, the matrix representative for $\bfomega_{p}$ is given by \[
        \bfomega_{p} = \begin{pmatrix}
            &&& p^{2k} \\
            &&(-p)^k& \\
            &(-p)^k&& \\
            1&&&
        \end{pmatrix}.
    \]
\end{Lemma}
\begin{proof}
    We again show the matrix representation for $\bfomega_{\frakp}$ as the one for $\bfomega_{\overline{\frakp}}$ is similar. 

    First of all, we have the computation \[
        \begin{pmatrix} & -p\\ 1\end{pmatrix} \begin{pmatrix} p & \\ & 1\end{pmatrix} = \begin{pmatrix} & -p\\ p & \end{pmatrix} = \begin{pmatrix} & -1\\ 1\end{pmatrix} \begin{pmatrix} p & \\ & p\end{pmatrix}.
    \] Therefore, \begin{align*}
        \bfomega_{\frakp}\bfupsilon_{\frakp}^* *[\mu_{f, i}] & = \left(\begin{pmatrix} & -1\\ 1 & \end{pmatrix} , \begin{pmatrix}1 & \\ & 1\end{pmatrix}\right) \left(\begin{pmatrix} p& \\  & p\end{pmatrix} , \begin{pmatrix}1 & \\ & 1\end{pmatrix}\right) * [\mu_{f, i}] \\
         & = p^k \left(\begin{pmatrix} & -1\\ 1 & \end{pmatrix} , \begin{pmatrix}1 & \\ & 1\end{pmatrix}\right) * [\mu_{f, i}]\\
         & = p^k [\mu_{f, i}],
    \end{align*}
    where the penultimate equation follows from that $\begin{pmatrix}p & \\ & p\end{pmatrix}$ acts on $P_k$ via $p^k$ and the last equation follows from that $\left(\begin{pmatrix} & -1\\ 1 & \end{pmatrix} , \begin{pmatrix}1 & \\ & 1\end{pmatrix}\right)\in G(\Z)$. Moreover, since \[
        \begin{pmatrix} & -p\\ 1 & \end{pmatrix}\begin{pmatrix} & -p\\ 1 & \end{pmatrix} = \begin{pmatrix}-p & \\ & -p\end{pmatrix},
    \] we have \[
        \bfomega_{\frakp}^2 * [\mu_{f, i}] = (-p)^k [\mu_{f, i}].
    \]

    Consequently, we know the matrix representation for $\bfomega_{\frakp}$ should be of the form \[
        \bfomega_{\frakp} = \begin{pmatrix} x & p^k & a & \\ y & & b & \\ z & & c & p^k \\ t & & d &  \end{pmatrix}
    \] such that \begin{align*}
         \begin{pmatrix} (-p)^k &&&\\ & (-p)^k && \\ && (-p)^k & \\ &&& (-p)^k  \end{pmatrix} & = \begin{pmatrix} x & p^k & a & \\ y & & b & \\ z & & c & p^k \\ t & & d &  \end{pmatrix} \begin{pmatrix} x & p^k & a & \\ y & & b & \\ z & & c & p^k \\ t & & d &  \end{pmatrix} \\
         & = \begin{pmatrix} x^2+p^ky + az & p^kx & xa+p^kb+ac & ap^k \\ 
        xy+bz & yp^{k} & ya + bc & bp^k\\
        zx + cz + p^kt & zp^k & za+c^2+dp^k & cp^k\\
        tx+dz & tp^k & ta+dc & dp^k\end{pmatrix}.
    \end{align*}
    Comparing the entries, one sees that \[
        x=z=t=a=b=c=0 \quad \text{ and }\quad y=d=(-1)^k.
    \]
    Hence the desired matrix representation.
\end{proof}

\begin{Corollary}\label{Corollary: adjoint operators for Hecke operators at p}
    The adjoint operator $U_{\frakp}^*$ (resp., $U_{\overline{\frakp}}^*$) for $U_{\frakp}$ (resp., $U_{\overline{\frakp}}$) on $W_{f, t}$, with respect to the pairing $\langle \cdot, \cdot \rangle_k$, has the matrix representation \[
        U_{\frakp}^* =  \begin{pmatrix} & -(-p)^k & & \\ (-1)^kp & \lambda_{\frakp} & & \\ &&& -(-p)^k\\ && (-1)^k p & \lambda_{\frakp}\end{pmatrix}\quad \left(\text{resp., }U_{\overline{\frakp}}^* = \begin{pmatrix} & & -(-p)^k & \\ & & & -(-p)^k\\ (-1)^kp & & \lambda_{\overline{\frakp}} & \\ & (-1)^kp & & \lambda_{\overline{\frakp}}\end{pmatrix}\right).
    \] Moreover, the adjoint operator $U_p^*$ for $U_p$ has the matrix representation \[
        U_p^* = \begin{pmatrix} & & & p^{2k}\\ & & -p^{k+1} & (-1)^{k+1}p^k\lambda_{\frakp}\\ & -p^{k+1} & & (-1)^{k+1}p^k\lambda_{\overline{\frakp}} \\ p^2 & (-1)^k p\lambda_{\frakp} & (-1)^kp\lambda_{\overline{\frakp}} & \lambda_{\frakp}\lambda_{\overline{\frakp}}\end{pmatrix}.
    \]
\end{Corollary}
\begin{proof}
    By \cite[Lemma 6.1]{Urban-PhD}, $U_{\frakp}^* = \bfomega_{p}^{-1} U_{\frakp} \bfomega_{p}$ (resp., $U_{\overline{\frakp}}^* = \bfomega_{p}^{-1}U_{\overline{\frakp}} \bfomega_{p}$). The result then follows from Lemma \ref{Lemma: matrix representation of Hecke operators at p} and Lemma \ref{Lemma: matrix representation for AL operators at p}.
\end{proof}

Recall that \[
    [\mu_{f, t}]_{(i,j)}^{(p)} = [\mu_{f, t}] - \alpha_{\frakp}^{(i), -1}\bfupsilon_{\frakp}^* *[\mu_{f, t}] - \alpha_{\overline{\frakp}}^{(j), -1} \bfupsilon_{\overline{\frakp}}^* * [\mu_{f, t}] + \alpha_{\frakp}^{(i), -1}\alpha_{\overline{\frakp}}^{(j), -1} \bfupsilon_{\frakp}^* \bfupsilon_{\overline{\frakp}}^* * [\mu_{f, t}].
\] To compute $\left[ [\mu_{f, 1}]_{(i,j)}^{(p)}, [\mu_{f, 2}]_{(i,j)}^{(p)}\right]_k$, note that, by Remark \ref{Remark: relation to the algebraic pairing} and Lemma \ref{Lemma: matrix representation for AL operators at p}, \begin{align*}
    \left[ [\mu_{f, 1}]_{(i,j)}^{(p)}, [\mu_{f, 2}]_{(i,j)}^{(p)}\right]_k & = \left\langle [\mu_{f, 1}]_{(i,j)}^{(p)}, \bfomega_{p}* [\mu_{f, 2}]_{(i,j)}^{(p)} \right\rangle_k \\ 
    & = \left\langle [\mu_{f, 1}]_{(i,j)}^{(p)}, \bfupsilon_{\frakp}^*\bfupsilon_{\overline{\frakp}}^* * [\mu_{f, 2}] \right\rangle_k  + \left\langle [\mu_{f, 1}]_{(i,j)}^{(p)},  (-p)^k (-\alpha_{\frakp}^{(i), -1})\bfupsilon_{\overline{\frakp}}^* * [\mu_{f, 2}]\right\rangle_k \\
    & \qquad + \left\langle [\mu_{f, 1}]_{(i,j)}^{(p)},  (-p)^k(-\alpha_{\overline{\frakp}}^{(j), -1})\bfupsilon_{\frakp}^* *[\mu_{f, 2}]\right\rangle_k  + \left\langle [\mu_{f, 1}]_{(i,j)}^{(p)},  p^{2k}[\mu_{f, 2}]\right\rangle_k .
\end{align*}
Our task is to understand the following pairings \[
    \begin{array}{cccc}
         & \left\langle [\mu_{f, 1}], \bfupsilon_{\frakp}^* *[\mu_{f, 2}] \right\rangle_k &  \left\langle [\mu_{f, 1}], \bfupsilon_{\overline{\frakp}}^* *[\mu_{f, 2}] \right\rangle_k & \left\langle [\mu_{f, 1}], \bfupsilon_{\frakp}^* \bfupsilon_{\overline{\frakp}}^* *[\mu_{f, 2}] \right\rangle_k\\ \\
         \left\langle \bfupsilon_{\frakp}^* *[\mu_{f, 1}], [\mu_{f, 2}] \right\rangle_k & \left\langle \bfupsilon_{\frakp}^* *[\mu_{f, 1}], \bfupsilon_{\frakp}^* *[\mu_{f, 2}] \right\rangle_k &  \left\langle \bfupsilon_{\frakp}^* *[\mu_{f, 1}], \bfupsilon_{\overline{\frakp}}^* *[\mu_{f, 2}] \right\rangle_k & \left\langle \bfupsilon_{\frakp}^* *[\mu_{f, 1}], \bfupsilon_{\frakp}^* \bfupsilon_{\overline{\frakp}}^* *[\mu_{f, 2}] \right\rangle_k\\  \\
         \left\langle \bfupsilon_{\overline{\frakp}}^* *[\mu_{f, 1}], [\mu_{f, 2}] \right\rangle_k & \left\langle \bfupsilon_{\overline{\frakp}}^* *[\mu_{f, 1}], \bfupsilon_{\frakp}^* *[\mu_{f, 2}] \right\rangle_k &  \left\langle \bfupsilon_{\overline{\frakp}}^* *[\mu_{f, 1}], \bfupsilon_{\overline{\frakp}}^* *[\mu_{f, 2}] \right\rangle_k & \left\langle \bfupsilon_{\overline{\frakp}}^* *[\mu_{f, 1}], \bfupsilon_{\frakp}^* \bfupsilon_{\overline{\frakp}}^* *[\mu_{f, 2}] \right\rangle_k\\ \\
         \left\langle \bfupsilon_{\frakp}^*\bfupsilon_{\overline{\frakp}}^* *[\mu_{f, 1}], [\mu_{f, 2}] \right\rangle_k &\left\langle \bfupsilon_{\frakp}^*\bfupsilon_{\overline{\frakp}}^* *[\mu_{f, 1}], \bfupsilon_{\frakp}^* *[\mu_{f, 2}] \right\rangle_k &  \left\langle \bfupsilon_{\frakp}^*\bfupsilon_{\overline{\frakp}}^* *[\mu_{f, 1}], \bfupsilon_{\overline{\frakp}}^* *[\mu_{f, 2}] \right\rangle_k & \left\langle \bfupsilon_{\frakp}^*\bfupsilon_{\overline{\frakp}}^* *[\mu_{f, 1}], \bfupsilon_{\frakp}^* \bfupsilon_{\overline{\frakp}}^* *[\mu_{f, 2}] \right\rangle_k
    \end{array}
\]
in terms of $\left\langle [\mu_{f, 1}], [\mu_{f, 2}] \right\rangle_k $.

We make the following assumptions for the computations below.

\begin{Assumption}\label{Assumption: Hecke eigenvalues at p are nonzero}
    \begin{enumerate}
        \item[(i)] The $T_{\bfupsilon_{\frakp}}$- and $T_{\bfupsilon_{\overline{\frakp}}}$-eigenvalues $\lambda_{\frakp}$ and $\lambda_{\overline{\frakp}}$ are nonzero real numbers.
        \item[(ii)] The Hecke polynomials $P_{\frakp}(X)$ and $P_{\overline{\frakp}}(X)$ (Proposition \ref{Proposition: p-stabilisation}) are irreducible over $\R$. 
    \end{enumerate}
\end{Assumption}

\begin{Remark}\label{Remark: explanation of Assumption 2}
    The first statement in Assumption \ref{Assumption: Hecke eigenvalues at p are nonzero} is a technical assumption, which is needed for computing $\left\langle \bfupsilon_{\frakp}^* *[\mu_{f, 1}],  \bfupsilon_{\overline{\frakp}}^* *[\mu_{f, 2}] \right\rangle_k$ and $\left\langle \bfupsilon_{\overline{\frakp}}^* *[\mu_{f, 1}], \bfupsilon_{\frakp}^* *[\mu_{f, 2}] \right\rangle_k$. Under (i), the second statement is equivalent to \[
        \lambda_{\frakp}^2 - 4p^{k+1}<0 \quad \text{ and }\quad \lambda_{\overline{\frakp}}^2-4p^{k+1}<0.
    \] We remark that this is a reasonable assumption as the Ramanujan conjecture predicts that \[
        p^{-\frac{k+1}{2}}|\lambda_{\frakp}| = 1\quad \text{ and }\quad p^{-\frac{k+1}{2}}|\lambda_{\overline{\frakp}}| = 1.
    \] As a consequence of Assumption \ref{Assumption: Hecke eigenvalues at p are nonzero} (ii), we have \[
        \alpha_{\frakp}^{(0)} = \overline{\alpha_{\frakp}^{(1)}} \quad \text{ and }\quad \alpha_{\overline{\frakp}}^{(0)} = \overline{\alpha_{\overline{\frakp}}^{(1)}},
    \] where $\overline{\cdot}$ stands for taking complex conjugation. Moreover, since $P_{\frakp}(X)$ and $P_{\overline{\frakp}}(X)$ are irreducible over $\R$, $\alpha_{\frakp}^{(0)}\neq \alpha_{\frakp}^{(1)}$ and $\alpha_{\overline{\frakp}}^{(0)}\neq \alpha_{\overline{\frakp}}^{(1)}$ and so $f$ is regular at both $\frakp$ and $\overline{\frakp}$ (see also \cite[Condition 2.2 \text{[K]} (C3)]{BW}) We also point out that, if $f$ is a base change form from a cuspidal eigenform, then (ii) follows from Hecke estimate (\cite[pp. 284]{Eigen}).
\end{Remark}

\paragraph{Computations for $\left\langle [\mu_{f, 1}], -\right\rangle_{k}$ and $\left\langle -, [\mu_{f, 2}]\right\rangle_{k}$.}
By the matrix representation of $U_{\frakp}^*$, $U_{\overline{\frakp}}^*$, and $U_p^*$, we have \[
    \bfupsilon_{\frakp}^* *[\mu_{f, t}] = (-1)^k p^{-1}U_{\frakp}^* [\mu_{f, t}], \quad \bfupsilon_{\overline{\frakp}}^* * [\mu_{f, t}] = (-1)^kp^{-1}U_{\overline{\frakp}}^* [\mu_{f, t}], \quad \text{ and }\quad \bfupsilon_{\frakp}^* \bfupsilon_{\overline{\frakp}}^* * [\mu_{f, t}] = p^{-2}U_p^* \left([\mu_{f, t}]\right)
\] for $t=1, 2$. \begin{enumerate}
    \item[$\bullet$] Computation for $\left\langle [\mu_{f, 1}], \bfupsilon_{\frakp}^* * [\mu_{f, 2}] \right\rangle_k$: \begin{align*}
        \left\langle [\mu_{f, 1}], \bfupsilon_{\frakp}^* * [\mu_{f, 2}] \right\rangle_k & = (-1)^k p^{-1}\left\langle [\mu_{f, 1}], U_{\frakp}^* [\mu_{f, 2}] \right\rangle_k \\
        & = (-1)^kp^{-1} \left\langle U_{\frakp}[\mu_{f, 1}], [\mu_{f, 2}] \right\rangle_k\\
        & = (-1)^kp^{-1} \left( \lambda_{\frakp}\left\langle [\mu_{f, 1}], [\mu_{f, 2}] \right\rangle_k - \left\langle \bfupsilon_{\frakp}^* *[\mu_{f, 1}], [\mu_{f, 2}]\right\rangle_k \right) \\
        & = (-1)^kp^{-1}\left( \lambda_{\frakp}\left\langle [\mu_{f, 1}], [\mu_{f, 2}] \right\rangle_k - (-1)^kp^{-1}\lambda_{\frakp} \left\langle [\mu_{f, 1}], [\mu_{f, 2}] \right\rangle_k + (-1)^kp^{-1} \left\langle [\mu_{f, 1}], \bfupsilon_{\frakp}^* *[\mu_{f, 2}] \right\rangle_k \right) \\
        & = (-1)^kp^{-1}\lambda_{\frakp}(1-(-1)^kp^{-1})\left\langle [\mu_{f, 1}], [\mu_{f, 2}] \right\rangle_k + p^{-2}\left\langle [\mu_{f, 1}], \bfupsilon_{\frakp}^* *[\mu_{f, 2}] \right\rangle_k.
    \end{align*}
    Therefore, we have \begin{equation}\label{eq: <f, frakp f> }
        \left\langle [\mu_{f, 1}], \bfupsilon_{\frakp}^* * [\mu_{f, 2}] \right\rangle_k = \frac{(-1)^kp^{-1}\lambda_{\frakp}(1-(-1)^kp^{-1})}{1-p^{-2}} \left\langle [\mu_{f, 1}], [\mu_{f, 2}] \right\rangle_k = \frac{\lambda_{\frakp}}{(-1)^k p +1} \left\langle [\mu_{f, 1}], [\mu_{f, 2}] \right\rangle_k.
    \end{equation}
    \item[$\bullet$] Computation for $\left\langle [\mu_{f, 1}], \bfupsilon_{\overline{\frakp}}^* * [\mu_{f, 2}] \right\rangle_k$: \begin{align*}
        \left\langle [\mu_{f, 1}], \bfupsilon_{\overline{\frakp}}^* * [\mu_{f, 2}] \right\rangle_k & = (-1)^kp^{-1} \left\langle [\mu_{f, 1}], U_{\overline{\frakp}}^* [\mu_{f, 2}] \right\rangle_k \\
        & = (-1)^kp^{-1}\lambda_{\overline{\frakp}}(1-(-1)^kp^{-1}) \left\langle [\mu_{f, 1}],  [\mu_{f, 2}] \right\rangle_k + p^{-2}\left\langle [\mu_{f, 1}], \bfupsilon_{\overline{\frakp}}^* * [\mu_{f, 2}] \right\rangle_k,
    \end{align*}
    where the computation is similar to the previous case. We then similarly conclude \begin{equation}\label{eq: <f, frakpbar f>}
        \left\langle [\mu_{f, 1}], \bfupsilon_{\overline{\frakp}}^* * [\mu_{f, 2}] \right\rangle_k = \frac{\lambda_{\overline{\frakp}}}{(-1)^k p+1} \left\langle [\mu_{f, 1}], [\mu_{f, 2}] \right\rangle_k.
    \end{equation}

    \item[$\bullet$] Computation for $\left\langle \bfupsilon_{\frakp}^* * [\mu_{f, 1}], [\mu_{f, 2}]\right\rangle_{k}$: From the above, we see that \begin{align*}
        \left\langle \bfupsilon_{\frakp}^* * [\mu_{f, 1}], [\mu_{f, 2}]\right\rangle_{k} & = (-1)^kp^{-1}\lambda_{\frakp} \left\langle [\mu_{f, 1}], [\mu_{f, 2}]\right\rangle_{k} - (-1)^kp^{-1}\left\langle  [\mu_{f, 1}], \bfupsilon_{\frakp}^* *[\mu_{f, 2}]\right\rangle_{k} \\
        & = (-1)^kp^{-1}\lambda_{\frakp}\left( 1- \frac{1}{(-1)^kp+1}\right)  \left\langle  [\mu_{f, 1}], [\mu_{f, 2}]\right\rangle_{k}.
    \end{align*} We conclude that \begin{equation}\label{eq: <frakp f, f>}
        \left\langle \bfupsilon_{\frakp}^* * [\mu_{f, 1}], [\mu_{f, 2}]\right\rangle_{k} = \frac{\lambda_{\frakp}}{(-1)^k p +1} \left\langle  [\mu_{f, 1}], [\mu_{f, 2}]\right\rangle_{k}.
    \end{equation} 

    \item[$\bullet$] Computation for $\left\langle \bfupsilon_{\overline{\frakp}}^* *[\mu_{f, 1}], [\mu_{f, 2}]\right\rangle_{k}$: Similar as before, we have \begin{align*}
        \left\langle \bfupsilon_{\overline{\frakp}}^* * [\mu_{f, 1}], [\mu_{f, 2}]\right\rangle_{k} & = (-1)^kp^{-1}\lambda_{\overline{\frakp}} \left\langle [\mu_{f, 1}], [\mu_{f, 2}]\right\rangle_{k} - (-1)^kp^{-1}\left\langle  [\mu_{f, 1}], \bfupsilon_{\overline{\frakp}}^* *[\mu_{f, 2}]\right\rangle_{k} \\
        & = (-1)^kp^{-1}\lambda_{\overline{\frakp}}\left( 1- \frac{1}{(-1)^kp+1}\right)  \left\langle  [\mu_{f, 1}], [\mu_{f, 2}]\right\rangle_{k}.
    \end{align*} Hence \begin{equation}\label{eq: <frakpbar f , f>}
        \left\langle \bfupsilon_{\overline{\frakp}}^* * [\mu_{f, 1}], [\mu_{f, 2}]\right\rangle_{k} = \frac{\lambda_{\overline{\frakp}}}{(-1)^kp +1} \left\langle  [\mu_{f, 1}], [\mu_{f, 2}]\right\rangle_{k}.
    \end{equation}

    \item[$\bullet$] Computation for $\left\langle [\mu_{f, 1}], \bfupsilon_{\frakp}^*\bfupsilon_{\overline{\frakp}}^* * [\mu_{f, 2}] \right\rangle_k$: \begin{align*}
        \left\langle [\mu_{f, 1}], \bfupsilon_{\frakp}^*\bfupsilon_{\overline{\frakp}}^* * [\mu_{f, 2}] \right\rangle_k & = p^{-2} \left\langle [\mu_{f, 1}], U_p^*\left( [\mu_{f, 2}] \right) \right\rangle_k\\
        & = p^{-2} \left\langle U_p\left([\mu_{f, 1}]\right),  [\mu_{f, 2}] \right\rangle_k\\
        & = p^{-2}\left( \lambda_{\frakp}\lambda_{\overline{\frakp}}\left\langle [\mu_{f, 1}],  [\mu_{f, 2}] \right\rangle_k -\lambda_{\overline{\frakp}} \left\langle \bfupsilon_{\frakp}^* *[\mu_{f, 1}],  [\mu_{f, 2}] \right\rangle_k \right. \\
        & \left.\qquad -\lambda_{\frakp} \left\langle \bfupsilon_{\overline{\frakp}}^* *[\mu_{f, 1}],  [\mu_{f, 2}] \right\rangle_k + \left\langle \bfupsilon_{\frakp}^*\bfupsilon_{\overline{\frakp}}^* *[\mu_{f, 1}],  [\mu_{f, 2}] \right\rangle_k \right) \\
        & = p^{-2}\left( \lambda_{\frakp}\lambda_{\overline{\frakp}} \left(1- \frac{2}{(-1)^kp+1}\right)\left\langle [\mu_{f, 1}],  [\mu_{f, 2}] \right\rangle_k + \left\langle \bfupsilon_{\frakp}^*\bfupsilon_{\overline{\frakp}}^* *[\mu_{f, 1}],  [\mu_{f, 2}] \right\rangle_k \right).
    \end{align*} We use the same technique and compute \begin{align*}
        \left\langle \bfupsilon_{\frakp}^*\bfupsilon_{\overline{\frakp}}^* *[\mu_{f, 1}],  [\mu_{f, 2}] \right\rangle_k & = p^{-2} \left\langle  [\mu_{f, 1}], U_p([\mu_{f, 2}])\right\rangle_k\\
        & = p^{-2}\left( \lambda_{\frakp}\lambda_{\overline{\frakp}} \left\langle [\mu_{f, 1}], [\mu_{f, 2}] \right\rangle_k - \lambda_{\overline{\frakp}} \left\langle [\mu_{f, 1}], \bfupsilon_{\frakp}^* *[\mu_{f, 2}] \right\rangle_k \right.\\
        & \qquad \left. -\lambda_{\frakp} \left\langle [\mu_{f, 1}], \bfupsilon_{\overline{\frakp}}^* *[\mu_{f, 2}] \right\rangle_k + \left\langle [\mu_{f, 1}], \bfupsilon_{\frakp}^* \bfupsilon_{\overline{\frakp}}^* *[\mu_{f, 2}] \right\rangle_k\right) \\
        & = p^{-2}\left( \lambda_{\frakp}\lambda_{\overline{\frakp}} \left( 1-\frac{2}{(-1)^kp +1}\right) \left\langle [\mu_{f, 1}], [\mu_{f, 2}] \right\rangle_k + \left\langle [\mu_{f, 1}], \bfupsilon_{\frakp}^* \bfupsilon_{\overline{\frakp}}^* *[\mu_{f, 2}] \right\rangle_k\right)
    \end{align*}
    Note that \[
        1-\frac{2}{(-1)^kp + 1} = \frac{(-1)^kp-1}{(-1)^kp +1}.
    \] Combining everything together, we see \begin{align*}
        \left\langle [\mu_{f, 1}], \bfupsilon_{\frakp}^* \bfupsilon_{\overline{\frakp}}^* *[\mu_{f, 2}] \right\rangle_k & = \frac{\lambda_{\frakp} \lambda_{\overline{\frakp}} \left((-1)^kp-1\right) }{p^2((-1)^kp+1)} \cdot\frac{1+p^{-2}}{1-p^{-4}} \left\langle [\mu_{f, 1}], [\mu_{f, 2}] \right\rangle_k
    \end{align*} and so \begin{equation}\label{eq: <f, frakpfrakpbar f>}
        \left\langle [\mu_{f, 1}], \bfupsilon_{\frakp}^* \bfupsilon_{\overline{\frakp}}^* *[\mu_{f, 2}] \right\rangle_k = \frac{\lambda_{\frakp}\lambda_{\overline{\frakp}}  }{((-1)^kp +1)^2} \left\langle [\mu_{f, 1}], [\mu_{f, 2}] \right\rangle_k.
    \end{equation}

    \item[$\bullet$] Computation for $\left\langle \bfupsilon_{\frakp}^*\bfupsilon_{\overline{\frakp}}^* * [\mu_{f, 1}],   [\mu_{f, 2}] \right\rangle_k$: From the computation above, we see that \begin{align*}
        \left\langle \bfupsilon_{\frakp}^*\bfupsilon_{\overline{\frakp}}^* * [\mu_{f, 1}],   [\mu_{f, 2}] \right\rangle_k & = p^{-2}\left( \lambda_{\frakp}\lambda_{\overline{\frakp}} \left( 1-\frac{2}{(-1)^kp +1}\right) \left\langle [\mu_{f, 1}], [\mu_{f, 2}] \right\rangle_k + \left\langle [\mu_{f, 1}], \bfupsilon_{\frakp}^* \bfupsilon_{\overline{\frakp}}^* *[\mu_{f, 2}] \right\rangle_k\right)\\
        & = p^{-2}\lambda_{\frakp}\lambda_{\overline{\frakp}}\left( \frac{\left((-1)^kp-1\right) }{(-1)^kp+1} + \frac{1}{((-1)^kp +1)^2} \right)\left\langle [\mu_{f, 1}], [\mu_{f, 2}] \right\rangle_k.
    \end{align*} Note that \[
        \frac{\left((-1)^kp-1\right) }{(-1)^kp+1} + \frac{1}{((-1)^kp +1)^2} = \frac{p^2}{((-1)^kp +1)^2}
    \] and so we conclude \begin{equation}\label{eq: <frakpfrakpbar f, f>}
        \left\langle \bfupsilon_{\frakp}^*\bfupsilon_{\overline{\frakp}}^* * [\mu_{f, 1}],   [\mu_{f, 2}] \right\rangle_k = \frac{\lambda_{\frakp}\lambda_{\overline{\frakp}}  }{((-1)^kp +1)^2} \left\langle [\mu_{f, 1}], [\mu_{f, 2}] \right\rangle_k.
    \end{equation}
\end{enumerate}

\paragraph{Computations for $\left\langle \bfupsilon_{\frakp}^* * [\mu_{f, 1}], \bfupsilon_{\frakp}^* *[\mu_{f, 2}] \right\rangle_k$ and $\left\langle \bfupsilon_{\overline{\frakp}}^* * [\mu_{f, 1}], \bfupsilon_{\overline{\frakp}}^* *[\mu_{f, 2}] \right\rangle_k$.} Using the matrix representations for $U_{\frakp}^*$ and $U_{\overline{\frakp}}^*$ again, we have the equations \begin{equation}\label{eq: <frakp f, frakp f>}
    \left\langle \bfupsilon_{\frakp}^* * [\mu_{f, 1}], \bfupsilon_{\frakp}^* *[\mu_{f, 2}] \right\rangle_k = (-p)^k \left\langle  [\mu_{f, 1}], [\mu_{f, 2}] \right\rangle_k,
\end{equation}
\begin{equation}\label{eq: <frakpbar f, frakpbar f>}
    \left\langle \bfupsilon_{\overline{\frakp}}^* * [\mu_{f, 1}], \bfupsilon_{\overline{\frakp}}^* *[\mu_{f, 2}] \right\rangle_k = (-p)^k \left\langle  [\mu_{f, 1}], [\mu_{f, 2}] \right\rangle_k,
\end{equation}
whose computations are given below.
\begin{enumerate}
    \item[$\bullet$] Computation for $\left\langle \bfupsilon_{\frakp}^* * [\mu_{f, 1}], \bfupsilon_{\frakp}^* *[\mu_{f, 2}] \right\rangle_k$: \begin{align*}
        \left\langle \bfupsilon_{\frakp}^* * [\mu_{f, 1}], \bfupsilon_{\frakp}^* *[\mu_{f, 2}] \right\rangle_k & = (-1)^k p^{-1}\left\langle U_{\frakp}^* [\mu_{f, 1}], \bfupsilon_{\frakp}^* *[\mu_{f, 2}] \right\rangle_k\\
        & = (-1)^k p^{-1} \left\langle  [\mu_{f, 1}], U_{\frakp}\left(\bfupsilon_{\frakp}^* *[\mu_{f, 2}]\right) \right\rangle_k \\
        & = (-1)^kp^k \left\langle  [\mu_{f, 1}], [\mu_{f, 2}] \right\rangle_k.
    \end{align*}

    \item[$\bullet$] Computation for $\left\langle \bfupsilon_{\overline{\frakp}}^* * [\mu_{f, 1}], \bfupsilon_{\overline{\frakp}}^* *[\mu_{f, 2}] \right\rangle_k$: \begin{align*}
        \left\langle \bfupsilon_{\overline{\frakp}}^* * [\mu_{f, 1}], \bfupsilon_{\overline{\frakp}}^* *[\mu_{f, 2}] \right\rangle_k & = (-1)^k p^{-1}\left\langle U_{\overline{\frakp}}^* [\mu_{f, 1}], \bfupsilon_{\frakp}^* *[\mu_{f, 2}] \right\rangle_k\\
        & = (-1)^k p^{-1} \left\langle  [\mu_{f, 1}], U_{\overline{\frakp}}\left(\bfupsilon_{\overline{\frakp}}^* *[\mu_{f, 2}]\right) \right\rangle_k \\
        & = (-1)^kp^k \left\langle  [\mu_{f, 1}], [\mu_{f, 2}] \right\rangle_k.
    \end{align*}
\end{enumerate}

\paragraph{Computations for $\left\langle -, \bfupsilon_{\frakp}^*\bfupsilon_{\overline{\frakp}}^* *[\mu_{f, 2}] \right\rangle_k$ and $\left\langle \bfupsilon_{\frakp}^*\bfupsilon_{\overline{\frakp}} * [\mu_{f, 1}], - \right\rangle_k$.} Recall that \[
    \bfupsilon_{\frakp}^* \bfupsilon_{\overline{\frakp}}^* *[\mu_{f, t}] = p^{-2}U_p^*([\mu_{f, t}])
\] for $t=1, 2$. \begin{enumerate}
    \item[$\bullet$] Computation for $\left\langle \bfupsilon_{\frakp}^* * [\mu_{f, 1}], \bfupsilon_{\frakp}^*\bfupsilon_{\overline{\frakp}}^* *[\mu_{f, 2}] \right\rangle_k$: \begin{align*}
        \left\langle \bfupsilon_{\frakp}^* * [\mu_{f, 1}], \bfupsilon_{\frakp}^*\bfupsilon_{\overline{\frakp}}^* *[\mu_{f, 2}] \right\rangle_k & = p^{-2} \left\langle \bfupsilon_{\frakp}^* * [\mu_{f, 1}], U_p^*([\mu_{f, 2}]) \right\rangle_k \\
         & = p^{-2} \left\langle U_p\left(\bfupsilon_{\frakp}^* * [\mu_{f, 1}]\right), [\mu_{f, 2}] \right\rangle_k \\ 
         & = p^{-2}\left( p^{k+1}\lambda_{\overline{\frakp}} \left\langle  [\mu_{f, 1}], [\mu_{f, 2}] \right\rangle_k - p^{k+1}\left\langle \bfupsilon_{\overline{\frakp}}^* * [\mu_{f, 1}], [\mu_{f, 2}] \right\rangle_k \right)  \\
         & = p^{k-1}\lambda_{\overline{\frakp}} \left\langle  [\mu_{f, 1}], [\mu_{f, 2}] \right\rangle_k - \frac{p^{k-1}\lambda_{\overline{\frakp}}}{(-1)^kp+1}\left\langle  [\mu_{f, 1}], [\mu_{f, 2}] \right\rangle_k\\
         & = p^{k-1}\lambda_{\overline{\frakp}} \left(1-\frac{1}{(-1)^kp+1}\right) \left\langle  [\mu_{f, 1}], [\mu_{f, 2}] \right\rangle_k.
    \end{align*} We conclude that \begin{equation}\label{eq: <frakp f, frakpfrakpbar f>}
        \left\langle \bfupsilon_{\frakp}^* * [\mu_{f, 1}], \bfupsilon_{\frakp}^*\bfupsilon_{\overline{\frakp}}^* *[\mu_{f, 2}] \right\rangle_k = \frac{(-p)^k\lambda_{\overline{\frakp}}}{(-1)^kp+1} \left\langle  [\mu_{f, 1}], [\mu_{f, 2}] \right\rangle_k
    \end{equation}

    \item[$\bullet$] Computation for $\left\langle \bfupsilon_{\overline{\frakp}}^* * [\mu_{f, 1}], \bfupsilon_{\frakp}^*\bfupsilon_{\overline{\frakp}}^* *[\mu_{f, 2}] \right\rangle_k$: \begin{align*}
        \left\langle \bfupsilon_{\overline{\frakp}}^* * [\mu_{f, 1}], \bfupsilon_{\frakp}^*\bfupsilon_{\overline{\frakp}}^* *[\mu_{f, 2}] \right\rangle_k & = p^{-2} \left\langle \bfupsilon_{\overline{\frakp}}^* * [\mu_{f, 1}], U_p^*([\mu_{f, 2}]) \right\rangle_k \\
         & = p^{-2} \left\langle U_p\left(\bfupsilon_{\overline{\frakp}}^* * [\mu_{f, 1}]\right), [\mu_{f, 2}] \right\rangle_k \\ 
         & = p^{-2}\left( p^{k+1}\lambda_{\frakp} \left\langle  [\mu_{f, 1}], [\mu_{f, 2}] \right\rangle_k - p^{k+1}\left\langle  \bfupsilon_{\frakp}^* *[\mu_{f, 1}], [\mu_{f, 2}] \right\rangle_k \right) \\
         & = p^{k-1}\lambda_{\frakp}\left( 1- \frac{1}{(-1)^kp+1}\right) \left\langle  [\mu_{f, 1}], [\mu_{f, 2}] \right\rangle_k.
    \end{align*} We conclude that \begin{equation}\label{eq: <frakpbar f, frakpfrakpbar f>}
        \left\langle \bfupsilon_{\overline{\frakp}}^* * [\mu_{f, 1}], \bfupsilon_{\frakp}^*\bfupsilon_{\overline{\frakp}}^* *[\mu_{f, 2}] \right\rangle_k = \frac{(-p)^k\lambda_{\frakp}}{(-1)^kp+1}\left\langle  [\mu_{f, 1}], [\mu_{f, 2}] \right\rangle_k.
    \end{equation}

    \item[$\bullet$] Computation for $\left\langle \bfupsilon_{\frakp}^*\bfupsilon_{\overline{\frakp}}^* * [\mu_{f, 1}], \bfupsilon_{\frakp}^* *[\mu_{f, 2}] \right\rangle_k$:  Similar computation as above allows us to conclude \begin{equation}\label{eq: <frakpfrakpbar f, frakp f>}
        \left\langle \bfupsilon_{\frakp}^*\bfupsilon_{\overline{\frakp}}^* * [\mu_{f, 1}], \bfupsilon_{{\frakp}}^* *[\mu_{f, 2}] \right\rangle_k = \frac{(-p)^k\lambda_{\overline{\frakp}}}{(-1)^kp + 1} \left\langle  [\mu_{f, 1}], [\mu_{f, 2}] \right\rangle_k.
    \end{equation}

    \item[$\bullet$] Computation for $\left\langle \bfupsilon_{\frakp}^*\bfupsilon_{\overline{\frakp}}^* * [\mu_{f, 1}], \bfupsilon_{\overline{\frakp}}^* *[\mu_{f, 2}] \right\rangle_k$: Similar computation as above allows us to conclude \begin{equation}\label{eq: <frakpfrakpbar f, frakpbar f>}
        \left\langle \bfupsilon_{\frakp}^*\bfupsilon_{\overline{\frakp}}^* * [\mu_{f, 1}], \bfupsilon_{\overline{\frakp}}^* *[\mu_{f, 2}] \right\rangle_k = \frac{(-p)^k\lambda_{{\frakp}}}{(-1)^kp + 1} \left\langle  [\mu_{f, 1}], [\mu_{f, 2}] \right\rangle_k.
    \end{equation} 

    \item[$\bullet$]  Computation for $\left\langle \bfupsilon_{\frakp}^*\bfupsilon_{\overline{\frakp}}^* * [\mu_{f, 1}], \bfupsilon_{\frakp}^*\bfupsilon_{\overline{\frakp}}^* *[\mu_{f, 2}] \right\rangle_k$: \begin{align*}
        \left\langle \bfupsilon_{\frakp}^*\bfupsilon_{\overline{\frakp}}^* * [\mu_{f, 1}], \bfupsilon_{\frakp}^*\bfupsilon_{\overline{\frakp}}^* *[\mu_{f, 2}] \right\rangle_k & = p^{-2}\left\langle  \bfupsilon_{\frakp}^*\bfupsilon_{\overline{\frakp}}^* *[\mu_{f, 1}], U_p^*([\mu_{f, 2}]) \right\rangle_k\\
         & = p^{-2}\left\langle  U_p\left( \bfupsilon_{\frakp}^*\bfupsilon_{\overline{\frakp}}^* *[\mu_{f, 1}]\right), [\mu_{f, 2}] \right\rangle_k.
    \end{align*} Using the matrix representation for $U_p$, we see that \[
        U_p\left( \bfupsilon_{\frakp}^*\bfupsilon_{\overline{\frakp}}^* *[\mu_{f, 1}]\right) = p^{2(k+1)}[\mu_{f, 1}]
    \] and so \begin{equation}\label{eq: <frakpfrakpbar f, frakpfrakpbar f>}
        \left\langle \bfupsilon_{\frakp}^*\bfupsilon_{\overline{\frakp}}^* * [\mu_{f, 1}], \bfupsilon_{\frakp}^*\bfupsilon_{\overline{\frakp}}^* *[\mu_{f, 2}] \right\rangle_k = p^{2k} \left\langle  [\mu_{f, 1}], [\mu_{f, 2}] \right\rangle_k.
    \end{equation}
\end{enumerate}

\paragraph{Computations for $\left\langle \bfupsilon_{\frakp}^* * [\mu_{f, 1}], \bfupsilon_{\overline{\frakp}}^* *[\mu_{f, 2}] \right\rangle_k$ and $\left\langle \bfupsilon_{\overline{\frakp}}^* * [\mu_{f, 1}], \bfupsilon_{{\frakp}}^* *[\mu_{f, 2}] \right\rangle_k$.} From the matrix representation for $U_p^*$, we have \[
    U_p^*\left( \bfupsilon_{\frakp}^* \bfupsilon_{\overline{\frakp}}^* *[\mu_{f, t}]\right) = p^{2k}[\mu_{f, t}] + (-1)^{k+1}p^k\lambda_{\frakp} \bfupsilon_{\frakp}^* [\mu_{f, t}] + (-1)^{k+1}p^k \lambda_{\overline{\frakp}} \bfupsilon_{\overline{\frakp}}^* *[\mu_{f, t}] + \lambda_{\frakp}\lambda_{\overline{\frakp}} \bfupsilon_{\frakp}^* \bfupsilon_{\overline{\frakp}}^* *[\mu_{f, t}]
\] and so \[
    \bfupsilon_{\overline{\frakp}}^* * [\mu_{f, t}] = (-1)^kp^k\lambda_{\overline{\frakp}}^{-1}[\mu_{f, t}] -\lambda_{\frakp}\lambda_{\overline{\frakp}}^{-1}\bfupsilon_{\frakp}^* [\mu_{f, t}] + (-1)^kp^{-k}\lambda_{\frakp}\bfupsilon_{\frakp}^* \bfupsilon_{\overline{\frakp}}^* *[\mu_{f, t}] - (-1)^kp^{-k}\lambda_{\overline{\frakp}}^{-1}U_p^*\left(\bfupsilon_{\frakp}^*\bfupsilon_{\overline{\frakp}}^* * [\mu_{f, t}] \right).
\] \begin{enumerate}
    \item[$\bullet$] Computation for  $\left\langle \bfupsilon_{\frakp}^* * [\mu_{f, 1}], \bfupsilon_{\overline{\frakp}}^* *[\mu_{f, 2}] \right\rangle_k$: \begin{align*}
        & \left\langle \bfupsilon_{\frakp}^* * [\mu_{f, 1}], \bfupsilon_{\overline{\frakp}}^* *[\mu_{f, 2}] \right\rangle_k\\
        & \qquad  = (-1)^kp^k \lambda_{\overline{\frakp}}^{-1} \left\langle  \bfupsilon_{\frakp}^* * [\mu_{f, 1}], [\mu_{f, 2}] \right\rangle_k - \lambda_{\frakp}\lambda_{\overline{\frakp}}^{-1}\left\langle  \bfupsilon_{\frakp}^* *[\mu_{f, 1}], \bfupsilon_{\frakp}^* * [\mu_{f, 2}] \right\rangle_k \\
        & \qquad\qquad  + (-1)^k p^{-k}\lambda_{\frakp} \left\langle  \bfupsilon_{\frakp}^* * [\mu_{f, 1}], \bfupsilon_{\frakp}^* \bfupsilon_{\overline{\frakp}}^* *[\mu_{f, 2}] \right\rangle_k - (-1)^kp^{-k}\lambda_{\overline{\frakp}}^{-1} \left\langle  \bfupsilon_{\frakp}^* *[\mu_{f, 1}], U_p^*\left( \bfupsilon_{\frakp}^* \bfupsilon_{\overline{\frakp}}^* *[\mu_{f, 2}]  \right)\right\rangle_k.
    \end{align*}
    Note that  \begin{align*}
        \left\langle  \bfupsilon_{\frakp}^* *[\mu_{f, 1}], U_p^*\left( \bfupsilon_{\frakp}^* \bfupsilon_{\overline{\frakp}}^* *[\mu_{f, 2}]  \right)\right\rangle_k & = \left\langle  U_p\left(\bfupsilon_{\frakp}^* *[\mu_{f, 1}] \right), \bfupsilon_{\frakp}^* \bfupsilon_{\overline{\frakp}}^* *[\mu_{f, 2}] \right\rangle_k \\
        & = p^{k+1}\lambda_{\overline{\frakp}} \left\langle  [\mu_{f, 1}], \bfupsilon_{\frakp}^* \bfupsilon_{\overline{\frakp}}^* * [\mu_{f, 2}] \right\rangle_k - p^{k+1}\left\langle  \bfupsilon_{\overline{\frakp}}^* *[\mu_{f, 1}], \bfupsilon_{\frakp}^* \bfupsilon_{\overline{\frakp}}^* *[\mu_{f, 2}] \right\rangle_k.
    \end{align*}
    Combining everything, we have \begin{align*}
        & \left\langle \bfupsilon_{\frakp}^* * [\mu_{f, 1}], \bfupsilon_{\overline{\frakp}}^* *[\mu_{f, 2}] \right\rangle_k \\
        & \qquad = (-1)^kp^k\lambda_{\overline{\frakp}}^{-1}\cdot \frac{\lambda_{\frakp}}{(-1)^kp+1} \left\langle  [\mu_{f, 1}], [\mu_{f, 2}] \right\rangle_k - \lambda_{\frakp}\lambda_{\overline{\frakp}}^{-1}\cdot (-1)^kp^k \left\langle  [\mu_{f, 1}], [\mu_{f, 2}] \right\rangle_k\\
        & \qquad\qquad  + (-p)^{-k}\lambda_{\frakp} \cdot \frac{(-p)^k\lambda_{\overline{\frakp}}}{(-1)^kp+1} \left\langle  [\mu_{f, 1}], [\mu_{f, 2}] \right\rangle_k -(-p)^{-k}\lambda_{\overline{\frakp}}^{-1}\cdot p^{k+1}\lambda_{\overline{\frakp}} \cdot \frac{\lambda_{\frakp}\lambda_{\overline{\frakp}}}{((-1)^kp+1)^2} \left\langle  [\mu_{f, 1}], [\mu_{f, 2}] \right\rangle_k\\
        & \qquad\qquad + (-p)^{-k}\lambda_{\overline{\frakp}}^{-1}\cdot p^{k+1}\cdot \frac{(-p)^k\lambda_{\frakp}}{(-1)^kp+1} \left\langle  [\mu_{f, 1}], [\mu_{f, 2}] \right\rangle_k\\
        & \qquad = \scalemath{0.9}{ \left( \frac{(-p)^k\lambda_{\frakp}\lambda_{\overline{\frakp}}^{-1} - p^{k+1}\lambda_{\frakp}\lambda_{\overline{\frakp}}^{-1} - (-p)^k\lambda_{\frakp}\lambda_{\overline{\frakp}}^{-1} + p^{k+1}\lambda_{\frakp}\lambda_{\overline{\frakp}}^{-1}}{(-1)^kp+1}  +\frac{(-1)^kp\lambda_{\frakp}\lambda_{\overline{\frakp}}-\lambda_{\frakp}\lambda_{\overline{\frakp}}-(-1)^kp\lambda_{\frakp}\lambda_{\overline{\frakp}}}{((-1)^kp +1)^2} \right) } \left\langle  [\mu_{f, 1}], [\mu_{f, 2}] \right\rangle_k.
    \end{align*}We then conclude \begin{equation}\label{eq: <frakp f, frakpbar f>}
        \left\langle \bfupsilon_{\frakp}^* * [\mu_{f, 1}], \bfupsilon_{\overline{\frakp}}^* *[\mu_{f, 2}] \right\rangle_k = \frac{-\lambda_{\frakp}\lambda_{\overline{\frakp}}}{((-1)^kp+1)^2}\left\langle  [\mu_{f, 1}], [\mu_{f, 2}] \right\rangle_k.
    \end{equation}

    \item[$\bullet$] Computation for $\left\langle \bfupsilon_{\overline{\frakp}}^* * [\mu_{f, 1}], \bfupsilon_{{\frakp}}^* *[\mu_{f, 2}] \right\rangle_k$: Using the same strategy as above and combining our previous computations, we also have \begin{equation}\label{eq: <frakpbar f, frakp f>}
        \left\langle \bfupsilon_{\overline{\frakp}}^* * [\mu_{f, 1}], \bfupsilon_{{\frakp}}^* *[\mu_{f, 2}] \right\rangle_k = \frac{-\lambda_{\frakp}\lambda_{\overline{\frakp}}}{((-1)^kp+1)^2}\left\langle  [\mu_{f, 1}], [\mu_{f, 2}] \right\rangle_k.
    \end{equation}
\end{enumerate}

\paragraph{Computation for $\left[ [\mu_{f, 1}]_{(i,j)}^{(p)}, [\mu_{f, 2}]_{(i,j)}^{(p)}\right]_k$.} We finally would like to combine the computations above to understand $\left[ [\mu_{f, 1}]_{(i,j)}^{(p)}, [\mu_{f, 2}]_{(i,j)}^{(p)}\right]_k$. 

Recall we have \begin{align*}
    \left[ [\mu_{f, 1}]_{(i,j)}^{(p)}, [\mu_{f, 2}]_{(i,j)}^{(p)}\right]_k &  = \left\langle [\mu_{f, 1}]_{(i,j)}^{(p)}, \bfupsilon_{\frakp}^*\bfupsilon_{\overline{\frakp}}^* * [\mu_{f, 2}] \right\rangle_k  + \left\langle [\mu_{f, 1}]_{(i,j)}^{(p)},  (-p)^k (-\alpha_{\frakp}^{(i), -1})\bfupsilon_{\overline{\frakp}}^* * [\mu_{f, 2}]\right\rangle_k \\
    & \qquad + \left\langle [\mu_{f, 1}]_{(i,j)}^{(p)},  (-p)^k(-\alpha_{\overline{\frakp}}^{(j), -1})\bfupsilon_{\frakp}^* *[\mu_{f, 2}]\right\rangle_k  + \left\langle [\mu_{f, 1}]_{(i,j)}^{(p)},  p^{2k}[\mu_{f, 2}]\right\rangle_k .
\end{align*} Let's study each term in the expansion. \begin{enumerate}
    \item[$\bullet$] Expansion for  $\left\langle [\mu_{f, 1}]_{(i,j)}^{(p)}, \bfupsilon_{\frakp}^*\bfupsilon_{\overline{\frakp}}^* * [\mu_{f, 2}] \right\rangle_k$: \begin{align*}
        & \left\langle [\mu_{f, 1}]_{(i,j)}^{(p)}, \bfupsilon_{\frakp}^*\bfupsilon_{\overline{\frakp}}^* * [\mu_{f, 2}] \right\rangle_k\\
        & \qquad = \left\langle  [\mu_{f, 1}], \bfupsilon_{\frakp}^* \bfupsilon_{\overline{\frakp}}^* *[\mu_{f, 2}] \right\rangle_k - \alpha_{\frakp}^{(i), -1}\left\langle  \bfupsilon_{\frakp}^* *[\mu_{f, 1}], \bfupsilon_{\frakp}^* \bfupsilon_{\overline{\frakp}}^* *[\mu_{f, 2}] \right\rangle_k \\ 
        & \qquad\qquad  -\alpha_{\overline{\frakp}}^{(j), -1}\left\langle  \bfupsilon_{\overline{\frakp}}^* *[\mu_{f, 1}], \bfupsilon_{\frakp}^* \bfupsilon_{\overline{\frakp}}^* *[\mu_{f, 2}] \right\rangle_k + \alpha_{\frakp}^{(i), -1}\alpha_{\overline{\frakp}}^{(j), -1}\left\langle  \bfupsilon_{\frakp}^* \bfupsilon_{\overline{\frakp}}^* *[\mu_{f, 1}], \bfupsilon_{\frakp}^* \bfupsilon_{\overline{\frakp}}^* *[\mu_{f, 2}] \right\rangle_k \\
        & \qquad = \left(\frac{\lambda_{\frakp}\lambda_{\overline{\frakp}}}{((-1)^kp+1)^2} - \alpha_{\frakp}^{(i), -1}\cdot \frac{(-p)^k\lambda_{\overline{\frakp}}}{(-1)^kp+1} - \alpha_{\overline{\frakp}}^{(j), -1}\cdot \frac{(-p)^k\lambda_{\frakp}}{(-1)^kp+1} + \alpha_{\frakp}^{(i), -1}\alpha_{\overline{\frakp}}^{(j), -1}p^{2k}\right) \left\langle  [\mu_{f, 1}], [\mu_{f, 2}] \right\rangle_k.
    \end{align*} 

    \item[$\bullet$] Expansion for $\left\langle [\mu_{f, 1}]_{(i,j)}^{(p)},  (-p)^k (-\alpha_{\frakp}^{(i), -1})\bfupsilon_{\overline{\frakp}}^* * [\mu_{f, 2}]\right\rangle_k$: \begin{align*}
        & \left\langle [\mu_{f, 1}]_{(i,j)}^{(p)},  (-p)^k (-\alpha_{\frakp}^{(i), -1})\bfupsilon_{\overline{\frakp}}^* * [\mu_{f, 2}]\right\rangle_k\\
        & \qquad = -(-p)^k\overline{\alpha_{{\frakp}}^{(i), -1}} \left( \left\langle  [\mu_{f, 1}], \bfupsilon_{\overline{\frakp}}^* *[\mu_{f, 2}]\right\rangle_k - \alpha_{\frakp}^{(i), -1} \left\langle  \bfupsilon_{\frakp}^{*}*[\mu_{f, 1}], \bfupsilon_{\overline{\frakp}}^* *[\mu_{f, 2}] \right\rangle_k \right.\\
        & \qquad \qquad  \left. - \alpha_{\overline{\frakp}}^{(j), -1} \left\langle  \bfupsilon_{\overline{\frakp}}^* *[\mu_{f, 1}], \bfupsilon_{\overline{\frakp}}^* *[\mu_{f, 2}] \right\rangle_k + \alpha_{\frakp}^{(i), -1}\alpha_{\overline{\frakp}}^{(j), -1}\left\langle  \bfupsilon_{\frakp}^* \bfupsilon_{\overline{\frakp}}^* *[\mu_{f, 1}], \bfupsilon_{\overline{\frakp}}^* *[\mu_{f, 2}] \right\rangle_k\right) \\
        & \qquad = -(-p)^k\overline{\alpha_{\frakp}^{(i), -1}}\left( \frac{\lambda_{\overline{\frakp}}}{(-1)^kp+1} -\alpha_{\frakp}^{(i), -1}\cdot \frac{-\lambda_{\frakp}\lambda_{\overline{\frakp}}}{((-1)^kp+1)^2} \right.\\
        & \qquad \qquad \left. -\alpha_{\overline{\frakp}}^{(j), -1}\cdot (-p)^k+ \alpha_{\frakp}^{(i), -1}\alpha_{\overline{\frakp}}^{(j), -1}\cdot \frac{(-p)^k\lambda_{\frakp}}{(-1)^kp+1}\right) \left\langle  [\mu_{f, 1}], [\mu_{f, 2}] \right\rangle_k
    \end{align*}

    \item[$\bullet$] Expansion for $\left\langle [\mu_{f, 1}]_{(i,j)}^{(p)},  (-p)^k(-\alpha_{\overline{\frakp}}^{(j), -1})\bfupsilon_{\frakp}^* *[\mu_{f, 2}]\right\rangle_k$: \begin{align*}
        & \left\langle [\mu_{f, 1}]_{(i,j)}^{(p)},  (-p)^k(-\alpha_{\overline{\frakp}}^{(j), -1})\bfupsilon_{\frakp}^* *[\mu_{f, 2}]\right\rangle_k\\
        & \qquad = -(-p)^k\overline{\alpha_{\overline{\frakp}}^{(j), -1}} \left( \left\langle  [\mu_{f, 1}], \bfupsilon_{\frakp}^* *[\mu_{f, 2}] \right\rangle_k  - \alpha_{\frakp}^{(i), -1} \left\langle  \bfupsilon_{\frakp}^* * [\mu_{f, 1}], \bfupsilon_{\frakp}^* *[\mu_{f, 2}] \right\rangle_k \right.\\
        & \qquad\qquad  \left. - \alpha_{\overline{\frakp}}^{(j), -1} \left\langle  \bfupsilon_{\overline{\frakp}}^* *[\mu_{f, 1}], \bfupsilon_{\frakp}^* *[\mu_{f, 2}] \right\rangle_k + \alpha_{\frakp}^{(i), -1}\alpha_{\overline{\frakp}}^{(j), -1} \left\langle  \bfupsilon_{\frakp}^* \bfupsilon_{\overline{\frakp}}^* *[\mu_{f, 1}], \bfupsilon_{\frakp}^* *[\mu_{f, 2}] \right\rangle_k\right)\\
        & \qquad  = -(-p)^k\overline{\alpha_{\overline{\frakp}}^{(j), -1}} \left(  \frac{\lambda_{\frakp}}{(-1)^kp+1} - \alpha_{\frakp}^{(i), -1}\cdot (-p)^k\right. \\
        & \qquad\qquad  \left. -\alpha_{\overline{\frakp}}^{(j), -1}\cdot \frac{-\lambda_{\frakp}\lambda_{\overline{\frakp}}}{((-1)^kp+1)^2} + \alpha_{\frakp}^{(i), -1}\alpha_{\overline{\frakp}}^{(j), -1}\cdot \frac{(-p)^k\lambda_{\overline{\frakp}}}{(-1)^k p +1} \right)\left\langle  [\mu_{f, 1}], [\mu_{f, 2}] \right\rangle_k
    \end{align*} 

    \item[$\bullet$] Expansion for $\left\langle [\mu_{f, 1}]_{(i,j)}^{(p)},  p^{2k}[\mu_{f, 2}]\right\rangle_k$: \begin{align*}
        & \left\langle [\mu_{f, 1}]_{(i,j)}^{(p)},  p^{2k}[\mu_{f, 2}]\right\rangle_k \\
        & \qquad = p^{2k}\left(\left\langle  [\mu_{f, 1}], [\mu_{f, 2}] \right\rangle_k - \alpha_{\frakp}^{(i), -1} \left\langle  \bfupsilon_{\frakp}^* *[\mu_{f, 1}], [\mu_{f, 2}] \right\rangle_k \right.\\ 
        & \qquad\qquad  \left. -\alpha_{\overline{\frakp}}^{(j), -1}\left\langle  \bfupsilon_{\overline{\frakp}}^* *[\mu_{f, 1}], [\mu_{f, 2}] \right\rangle_k + \alpha_{\frakp}^{(i), -1}\alpha_{\overline{\frakp}}^{(j), -1}\left\langle  \bfupsilon_{\frakp}^* \bfupsilon_{\overline{\frakp}}^* *[\mu_{f, 1}], [\mu_{f, 2}] \right\rangle_k\right) \\
        & \qquad  = p^{2k} \left( 1- \alpha_{\frakp}^{(i), -1}\cdot \frac{\lambda_{\frakp}}{(-1)^kp +1}   -\alpha_{\overline{\frakp}}^{(j), -1}\cdot \frac{\lambda_{\overline{\frakp}}}{(-1)^kp+1} + \alpha_{\frakp}^{(i), -1}\alpha_{\overline{\frakp}}^{(j), -1} \cdot \frac{\lambda_{\frakp}\lambda_{\overline{\frakp}}}{((-1)^kp +1)^2}\right) \left\langle  [\mu_{f, 1}], [\mu_{f, 2}] \right\rangle_k
    \end{align*} 
\end{enumerate} 

As a consequence, we have \begin{equation*}
    \begin{array}{r}
        \left[ [\mu_{f, 1}]_{(i,j)}^{(p)}, [\mu_{f, 2}]_{(i,j)}^{(p)}\right]_k = \begin{pmatrix} p^{2k}\left(1+\alpha_{\frakp}^{(i), -1}\overline{\alpha_{\overline{\frakp}}^{(j), -1}} + \overline{\alpha_{\frakp}^{(i), -1}}\alpha_{\overline{\frakp}}^{(j), -1} + \alpha_{\frakp}^{(i), -1}\alpha_{\overline{\frakp}}^{(j), -1} \right) \\ \\
        -\dfrac{(-p)^k\lambda_{\overline{\frakp}}}{(-1)^kp +1} \left(\alpha_{\frakp}^{(i), -1} + \overline{\alpha_{\frakp}^{(i), -1}} + (-p)^k\alpha_{\frakp}^{(i), -1}\alpha_{\overline{\frakp}}^{(j), -1}\overline{\alpha_{\overline{\frakp}}^{(j), -1}} + (-p)^k\alpha_{\overline{\frakp}}^{(j), -1} \right) \\ \\
        -\dfrac{(-p)^k\lambda_{{\frakp}}}{(-1)^kp +1}\left(\alpha_{\overline{\frakp}}^{(j), -1} + \overline{\alpha_{\overline{\frakp}}^{(j), -1}} + (-p)^k\alpha_{\frakp}^{(i), -1}\overline{\alpha_{\frakp}^{(i), -1}} \alpha_{\overline{\frakp}}^{(j), -1} + (-p)^k\alpha_{\frakp}^{(i), -1} \right)\\ \\
        + \dfrac{\lambda_{\frakp}\lambda_{\overline{\frakp}}}{((-1)^kp+1)^2} \left(1-(-p)^k\overline{\alpha_{\frakp}^{(i), -1}}\alpha_{\frakp}^{(i), -1} - (-p)^k\alpha_{\overline{\frakp}}^{(j), -1}\overline{\alpha_{\overline{\frakp}}^{(j), -1}} + p^{2k}\alpha_{\frakp}^{(i), -1}\alpha_{\overline{\frakp}}^{(j), -1} \right)
    \end{pmatrix} \\ \\
    \times \left\langle  [\mu_{f, 1}], [\mu_{f, 2}] \right\rangle_k
    \end{array}
\end{equation*}
Using the relations \[
    \alpha_{\frakp}^{(i)}\alpha_{\frakp}^{(1-i)} = p^{k+1}, \quad \alpha_{\frakp}^{(i)} + \alpha_{\frakp}^{(1-i)} = \lambda_{\frakp}, \quad \alpha_{\overline{\frakp}}^{(j)}\alpha_{\overline{\frakp}}^{(1-j)} = p^{k+1}, \quad \alpha_{\overline{\frakp}}^{(j)} + \alpha_{\overline{\frakp}}^{(1-j)} = \lambda_{\overline{\frakp}},
\] one deduces \begin{equation}\label{eq: [p-stab f, p-stab f]}
    \left[ [\mu_{f, 1}]_{(i,j)}^{(p)}, [\mu_{f, 2}]_{(i,j)}^{(p)}\right]_k = \Theta(\lambda_{\frakp}, \lambda_{\overline{\frakp}}, i,j) \left\langle  [\mu_{f, 1}], [\mu_{f, 2}] \right\rangle_k,
\end{equation} where \begin{align*}
     & \Theta(\lambda_{\frakp}, \lambda_{\overline{\frakp}}, i,j) \\ 
     & \quad =  \scalemath{0.7}{\frac{p^{2k+2}(p^2-1)-p\lambda_{\frakp}\lambda_{\overline{\frakp}}(p+2+2(-1)^k) + (-1)^kp((-1)^kp+1)(\alpha_{\frakp}^{(1-i)}\alpha_{\overline{\frakp}}^{(j)} + \alpha_{\frakp}^{(i)}\alpha_{\overline{\frakp}}^{(1-i)}) + \alpha_{\frakp}^{(1-i)}\alpha_{\overline{\frakp}}^{(1-j)}(\lambda_{\frakp}\lambda_{\overline{\frakp}}-(-1)^kp-1)-p((-1)^kp+1)(\alpha_{\frakp}^{(1-i), 2} + \alpha_{\overline{\frakp}}^{(1-j), 2}) }{p^2((-1)^kp+1)^2} }.
\end{align*}

\subsection{A formula for the adjoint \texorpdfstring{$L$}{L}-value}\label{subsection: adjoint L-value formula}

We keep the assumption above that $f$ is a Bianchi cuspidal eigenform of level $1$ with cohomological weight $k = (k, k)\neq (0,0)$. For any prime ideal $\frakq \subset \calO_K$, denote by $\varpi_{\frakq}$ a fixed uniformiser of $\calO_{K, \frakq}$ and let $T_{\bfupsilon_{\frakq}}$ be the Hecke operator defined by the matrix \[
    \bfupsilon_{\frakq} = \begin{pmatrix}1 & \\ & \varpi_{\frakq}\end{pmatrix}.
\] Denote by $\lambda_{\frakq}$ the $T_{\bfupsilon_{\frakq}}$-eigenvalue of $f$. We similarly consider the Hecke polynomial \[
    P_{\frakq}(X) = X^2 -\lambda_{\frakq} X + q_{\frakq}^{k+1},
\] where $q_{\frakq}$ is the cardinality of the residue field of $\calO_{K, \frakq}$.  Let $\alpha_{\frakq}^{(0)}$, $\alpha_{\frakq}^{(1)}$ be the two roots of $P_{\frakq}(X)$. Then, the adjoint $L$-function associated with $f$ is defined to be the Euler product \[
    L(\mathrm{ad}^0 f, s) := \prod_{\frakq} \left( \left(1-\frac{\alpha_{\frakq}^{(0)}}{\alpha_{\frakq}^{(1)}}q_{\frakq}^{-s}\right) \left(1-q_{\frakq}^{-s}\right) \left(1-\frac{\alpha_{\frakq}^{(1)}}{\alpha_{\frakq}^{(0)}}q_{\frakq}^{-s} \right) \right)^{-1}.
\]
Note that this product converges when $\Re(s)$ is sufficiently large.

\begin{Proposition}[$\text{\cite[Proposition 7.1]{Urban-PhD}}$]\label{Proposition: Urban's formula}
    Let $f$ be as above. Then, \[
        \langle [\mu_{f, 1}], [\mu_{f, 2}] \rangle_{G(\Z), k} = \frac{D_{\infty}(k+1, \Phi_{\infty})\mathrm{disc}(K)}{16\pi} L(\mathrm{ad}^0f, 1),
    \] where $\mathrm{disc}(K)$ is the discriminant of $K$ and  \[
        D_{\infty}(k+1, \Phi_{\infty}) = (2\pi)^{-1-2k}\times (-1)^{k+1} \times \frac{((k+1)!)^2}{2k+1}\times \sum_{|n|\leq k+1}(-1)^n a_{2n, -2n}\begin{pmatrix}2k+2\\ k+n+1\end{pmatrix}
    \] with \[
        a_{n, -n} = \left\{ \begin{array}{ll}
            \dfrac{(2k^2+k)\times (\substack{2k\\k})}{(k!)^2} \times (-1)^n\times (k-n-1)!(k+n-1)!, & \text{ if }n\in \{-k+1, ..., k-1\} \\ \\
            (2k+1)(k+1)(-1)^{k+1}\times (\substack{2k\\k})^2, & \text{ if }n= -1-k\text{ or }1+k\\ \\
            (-1)^k\times (\substack{2k\\k})^2, & \text{ if }n=k\text{ or }-k
        \end{array}\right. .
    \]
\end{Proposition}

\begin{Proposition}\label{Proposition: pairing and adjoint L-value}
    Let $f$ be as above and let $(i,j)\in (\Z/2\Z)^2$. Then, \[
        \left[ [\mu_{f, 1}]_{(i,j)}^{(p)}, [\mu_{f, 2}]_{(i,j)}^{(p)}\right]_k = \frac{(p+1)^2\Theta(\lambda_{f,\frakp}, \lambda_{f,\overline{\frakp}}, i,j)D_{\infty}(k+1, \Phi_{\infty})\mathrm{disc}(K)}{16 \pi} \times  L(\mathrm{ad}^0 f, 1).\footnote{ Recall we abbreviated $\lambda_{f, \frakp}=\lambda_{\frakp}$ and $\lambda_{f, \overline{\frakp}}=\lambda_{\overline{\frakp}}$ in the previous subsection.}
    \]
\end{Proposition}
\begin{proof}
    Note that 
    $\langle [\mu_{f, 1}], [\mu_{f, 2}] \rangle_{ k} = (p+1)^2\langle [\mu_{f, 1}], [\mu_{f, 2}] \rangle_{G(\Z), k}$ where the factor $(p+1)^2$ is the index $[G(\Z_p): \Iw_G]$. Hence, the formula follows from  Proposition \ref{Proposition: Urban's formula} and \eqref{eq: [p-stab f, p-stab f]}.
\end{proof}

\begin{Corollary}\label{Corollary: nonvanishing Euler factor}
    Let $f$ be as above and suppose one can choose $(i,j)\in (\Z/2\Z)^2$ such that $[\mu_{f, 1}]_{(i,j)}^{(p)}$ and $[\mu_{f, 2}]_{(i,j)}^{(p)}$ satisfy Assumption \ref{Assumption: multiplicity one on degree 1 and 2} and are non-critical. Then \[
        \Theta(\lambda_{f,\frakp}, \lambda_{f,\overline{\frakp}}, i,j) \neq 0.
    \]
\end{Corollary}
\begin{proof}
    Note that $D_{\infty}(k+1, \Phi_{\infty})$ is nonzero (\cite[Notation 5.5a]{Urban-PhD}). Also note that the non-degeneracy of the classical Poincaré pairing and Urban's formula implies the non-vanishing of $L(\mathrm{ad}^0 f, 1)$. The result then follows from the non-degeneracy of $\left[ [\mu_{f, 1}]_{(i,j)}^{(p)}, [\mu_{f, 2}]_{(i,j)}^{(p)}\right]_k$ in this case.  
\end{proof}

\begin{Remark}
    Immediately from the corollary, we see that the phenomenon of trivial zeros shall not happen when $f$ is regular. Although it is an interesting question to study when $\Theta(\lambda_{\frakp}, \lambda_{\overline{\frakp}}, i,j) = 0$, we do not pursue in this direction in the present paper. 
\end{Remark}

\begin{Corollary}\label{Corollary: interpolation property of p-adic adjoint L-function}
    Let $f$ be as above and suppose one can choose $(i,j)\in (\Z/2\Z)^2$ such that $[\mu_{f, 1}]_{(i,j)}^{(p)}$ and $[\mu_{f, 2}]_{(i,j)}^{(p)}$ satisfy Assumption \ref{Assumption: multiplicity one on degree 1 and 2} and are non-critical. Suppose the Hecke eigensystem of $[\mu_{f, 1}]_{(i,j)}^{(p)}$ (and $[\mu_{f, 2}]_{(i,j)}^{(p)}$) defines a point $x_{f_{(i,j)}^{(p)}}$ in $\calE$. We assume that $x_{f_{(i,j)}^{(p)}}$ varies in a family over an affinoid curve $\calU \subset \calW$ such that Assumption \ref{Assumption: vary in 1-dimensional family} is satisfied. 
    \begin{enumerate}
        \item[(i)]  Let the $p$-adic adjoint $L$-function $L_{\calU, h}^{\adj}$ be chosen as in \eqref{eq: natural choice of p-adic L-function}. We have \begin{equation}\label{eq: interpolation formula}
            L_{\calU, h}^{\adj}\left(x_{f_{(i,j)}^{(p)}}\right) = \frac{(p+1)^2\Theta(\lambda_{f,\frakp}, \lambda_{f,\overline{\frakp}}, i,j)D_{\infty}(k+1, \Phi_{\infty})\mathrm{disc}(K)}{16\pi \Omega_{f_{(i,j)}^{(p)}, 1}\Omega_{f_{(i,j)}^{(p)}, 2}} \times  L(\mathrm{ad}^0 f, 1).
        \end{equation}
        \item[(ii)] Moreover, after shrinking $\calU$ if necessary so that $\wt$ is étale on the connected component $\calX$ of $\wt^{-1}(\calU)$ containing $x_{f_{(i,j)}^{(p)}}$ (see the discussion after Corollary \ref{Corollary: non-vanishing result for p-adic adjoint L-function}). If $g$ is another Bianchi cuspform of weight $k'$ satisfying the same assumptions on $f$ as above such that $[\mu_{g,1}]_{(i,j)}^{(p)}$ defines a point $x_{g_{(i,j)}^{(p)}}$ in $\calX$, then a similar formula also applies to $L_{\calU, h}^{\adj}(x_{g_{(i,j)}^{(p)}})$ (by simply replacing the $f$'s in \eqref{eq: interpolation formula} by $g$).
    \end{enumerate}
\end{Corollary}
\begin{proof}
    This follows immediately from Proposition \ref{Proposition: pairing and adjoint L-value} and the Corollary \ref{Corollary: specific choice of p-adic adjoint L-function and its special value}. 
\end{proof}

\begin{Remark}\label{Remark: interpolation formula }
    Corollary \ref{Corollary: interpolation property of p-adic adjoint L-function} justifies the terminology `$p$-adic adjoint $L$-function' for $L_{\calU, h}^{\adj}$ in the sense that it $p$-adically interpolates the adjoint $L$-value $L(\ad^0(-), 1)$, \emph{i.e.}, the $p$-adic variables for $L_{\calU, h}^{\adj}$ are points in a $p$-adic family of Bianchi forms.
\end{Remark}

\begin{Remark}\label{Remark: comparison with Bellaïche}
    Finally, we remark that the `adjoint $L$-function' in \cite[\S 9.5.2]{Eigen} is actually the `symmetric $L$-function $L(\mathrm{Sym}f, s)$'. After a change of variable, one has \[
        L(\mathrm{Sym}f, s-k-1) = L(\ad^0 f, s).
    \] In particular, the $p$-adic adjoint $L$-function in \emph{op. cit.} should interpolate $L(\ad^0(-), 1)$ in $p$-adic families. We also remark that a precise interpolation formula ( \emph{i.e.}, a similar formula as in Corollary \ref{Corollary: interpolation property of p-adic adjoint L-function}) is not written down in \cite{Eigen}; instead, Bellaïche provided an upper- and lower-bound for the $p$-adic norm of his $p$-adic adjoint $L$-function (\cite[Theorem 9.5.2]{Eigen}). 
\end{Remark}

\begin{appendix}
    \section{Cohomological correspondence for topological spaces}\label{section: cohomological correspondences for topological spaces}
In this appendix, we set up a framework of cohomological correspondences that we used in the main body of the paper. Materials presented here are presumably well-known to the expert, but we have not been able to find a reference.

\subsection{Some recollection of the lower shriek functor}\label{subsection: recall some topological terminologies}

Let $X$ be a locally compact Hausdorff topological space. Suppose $f: Y \rightarrow X$ is a continuous map between locally compact Hausdorff topological spaces, standard sheaf theory gives rise to two functors \[
    f_*: \Ab(Y) \rightarrow \Ab(X) \quad \text{ and }\quad f^{-1}: \Ab(X) \rightarrow \Ab(Y)
\]
between categories of sheaves of abelian groups on $X$ and $Y$. 
It is well-known that $f_*$ is the right adjoint of $f^{-1}$. In particular, for any sheaf $\scrF$ on $X$, we have a natural morphism of sheaves \[
    \scrF \rightarrow f_* f^{-1}\scrF,
\]
which corresponds to the identity $f^{-1} \scrF \rightarrow f^{-1}\scrF$.

Recall that a continuous map between locally compact Hausdorff topological spaces $f: Y \rightarrow X$ is \textbf{\textit{proper}} if the inverse images of compact subsets in $X$ are still compact in $Y$. We have another functor \[
    f_!: \Ab(Y) \rightarrow \Ab(X)
\] given by, for any open subset $U\subset X$, \[
    f_!\scrF(U)  = \left\{s\in \scrF(f^{-1}(U)): f: \supp s \rightarrow U\text{ is proper}  \right\}.
\]

\begin{Lemma}\label{Lemma: lower shrieck of proper maps}
    Let $f: Y \rightarrow X$ be a proper continuous map between locally compact Hausdorff topological spaces. Then, we have an identification of functors \[
        f_* = f_!.
    \]
\end{Lemma}
\begin{proof}
    The assertion follows immediately from the construction. 
\end{proof}

\begin{Lemma}\label{Lemma: trace map}
    Let $f: Y \rightarrow X$ be a continuous map between locally compact Hausdorff spaces with the property that for any $x\in X$, there exists an open neighbourhood $x\in U_x \subset X$ such that $f^{-1}(U_x) = \bigsqcup_{\sigma: U_x \rightarrow Y} \sigma(U_x)$, where $\sigma: U_x \rightarrow Y$ are continuous sections such that $\sigma: U_x \simeq \sigma(U_x)$. Then, for any $\scrF \in \Ab(X)$ and any $\scrG \in \Ab(Y)$, we have the adjunction \[
        \Hom_{\Ab(X)}(f_!\scrG, \scrF) = \Hom_{\Ab(Y)}(\scrG, f^{-1}\scrF).
    \]
\end{Lemma}
\begin{proof}
     Let $U \subset X$ be a small enough open subset. By assumption, we may assume that $f^{-1}(U) = \bigsqcup_{\sigma: U \rightarrow Y} \sigma(U) = \bigsqcup_{\sigma: U \rightarrow Y} U$. Hence, for any $\scrG \in \Ab(Y)$,  $\scrG(f^{-1}(U)) = \prod_{\sigma: U \rightarrow Y} \scrG(\sigma(U))$. The proper-support condition in the definition of $f_!\scrF$ then yields and identification \[
        f_!\scrG(U) = \bigoplus_{\sigma: U \rightarrow Y} \scrG(\sigma(U)). 
     \] 
     
     To show the adjunction, we claim the existence of the following two morphisms: \begin{enumerate}
         \item[(i)] for any $\scrF\in \Ab(X)$, there exists a natural nonzero morphism $\tr: f_!f^{-1}\scrF \rightarrow \scrF$;
         \item[(ii)] for any $\scrG\in \Ab(Y)$, there exists a natural nonzero morphism $\mathrm{cotr}: \scrG \rightarrow f^{-1}f_!\scrG$;
     \end{enumerate}
     such that the maps \begin{align*}
         \Phi: & \Hom_{\Ab(X)}(f_!\scrG, \scrF) \xrightarrow{f^{-1}} \Hom_{\Ab(Y)}(f^{-1}f_!\scrG, f^{-1}\scrG) \xrightarrow{\circ \mathrm{cotr}} \Hom_{\Ab(Y)}(\scrG, f^{-1}\scrF)\\
         \Psi: & \Hom_{\Ab(Y)}(\scrF, f^{-1}\scrG) \xrightarrow{f_!} \Hom_{\Ab(X)}(f_!\scrG, f_!f^{-1}\scrF) \xrightarrow{\tr\circ} \Hom_{\Ab(X)}(f_!\scrG, \scrF)
     \end{align*} yield \begin{equation}\label{eq: adjunction identities}
         \Phi \circ \Psi = \id \quad \text{ and }\quad \Psi \circ \Phi = \id.
     \end{equation}
     
     We first construct $\tr$. Given $\scrF\in \Ab(X)$ and small-enough open subset $U \subset X$ as above, we have \[
        f_! f^{-1}\scrF(U) = \bigoplus_{\sigma: U \rightarrow Y} f^{-1}\scrF(\sigma(U)) = \bigoplus_{\sigma: U \rightarrow Y} \scrF(f(\sigma(U))) = \bigoplus_{\sigma: U \rightarrow Y} \scrF(U),
     \] where the penultimate equation follows from the assumption that $f$ is a local homeomorphism and each $\sigma$ is a continuous section. We then define \[
        \tr(U): f_!f^{-1}\scrF(U) \rightarrow \scrF(U), \quad (s_{\sigma})_{\sigma} \mapsto \sum_{\sigma}s_{\sigma}.
     \] 
     It is easy to see that this map glues to a morphism $\tr: f_!f^{-1}\scrF \rightarrow \scrF$. 

     Next, we construct $\mathrm{cotr}$. Let $V \subset Y$ be an open subset such that $f|_V : V \rightarrow f(V)$ is a homeomorphism. After shrinking $V$, we may further assume that $f^{-1}(f(V))$ is a disjoint union of open subsets, each of which is homeomorphic to $f(V)$ under $f$. Given $\scrG\in \Ab(Y)$, we then have \[
        f^{-1}f_!\scrG(V) = f_! \scrG(f(V)) = \bigoplus_{\sigma: f(V) \rightarrow Y} \scrG(\sigma(f(V))).
     \] Note that $V$ agrees with one of $\sigma(f(V))$. Hence, we have a natural map \[
        \mathrm{cotr}: \scrG(V) = \scrG(\sigma(f(V))) \rightarrow \bigoplus_{\lambda: f(V) \rightarrow Y} \scrG(\lambda(f(V))), \quad s \mapsto (s_{\lambda})_{\lambda}\text{, where }s_{\lambda} = \left\{ \begin{array}{ll}
            s, & \text{if }\lambda = \sigma \\
            0, & \text{else}
        \end{array}\right..
     \] This map also glue to a morphism $\mathrm{cotr}: \scrG \rightarrow f^{-1}f_!\scrG$.

     It remains to show \eqref{eq: adjunction identities}. We begin with $\alpha \in \Hom_{\Ab(X)}(f_!\scrG, \scrF)$ and we want to show $\Psi(\Phi(\alpha)) = \alpha$. We can check this locally at stalks. For any $x\in X$, unwinding exerything, we have \[
        \begin{tikzcd}[row sep = tiny]
             \Psi(\Phi(\alpha))_x : & \displaystyle \bigoplus_{y\in f^{-1}(x)}\scrG_y \arrow[r,"{\bigoplus \mathrm{cotr}_y}"] &  \displaystyle \bigoplus_{y\in f^{-1}(x)} \bigoplus_{z\in f^{-1}(x)} \scrG_y \arrow[r, "{\bigoplus \alpha_x}"] &  \displaystyle \bigoplus_{z\in f^{-1}(x)} \scrF_x \arrow[r, "{\tr_x}"] &  \scrF_x\\
             & (s_{y})_y \arrow[r, mapsto ] & \left(s_{y,z}  \right)_{y, z} \arrow[r, mapsto] & (\alpha_x(s_z))_z \arrow[r, mapsto] & \alpha_x((s_z)_z).
        \end{tikzcd}
     \] Here, the image of $(s_y)$ under the first map is given by \[
        s_{y,z} = \left\{ \begin{array}{ll}
            s_y, & \text{ if }y=z \\
            0, & \text{else}
        \end{array}\right. .
     \] We can then conclude that $\alpha = \Psi(\Phi(\alpha))$.

     On the other hand, given $\beta\in \Hom_{\Ab(Y)}(\scrG, f^{-1}\scrF)$, we show $\Phi(\Psi(\beta)) = \beta$. We again check this locally at stalks. For any $y\in Y$, unwinding everything, one obtains \[
        \begin{tikzcd}[row sep = tiny]
            \Phi(\Psi(\beta))_y: & \scrG_y \arrow[r, "\mathrm{cotr}"] & \displaystyle \bigoplus_{z\in f^{-1}(f(y))} \scrG_z \arrow[r, "\bigoplus_{z} \beta_z"] & \displaystyle \bigoplus_{z\in f^{-1}(f(y))}\scrF_{f(y)} \arrow[r, "\tr"] & \scrF_{f(y)}\\
            & s \arrow[r, mapsto] & (s_z)_z \arrow[r, mapsto] & (\beta_z(s_z))_z \arrow[r, mapsto] & \sum_{z}\beta_z(s_z)
        \end{tikzcd}.
     \] However, since $s_z = \left\{ \begin{array}{ll}
         s, & \text{ if }y=z \\
         0, & \text{else}
     \end{array}\right.$, we see that $\sum_{z}\beta_z(s_z) = \beta_y(s)$ and so $\beta = \Phi(\Psi(\beta))$. 
\end{proof}

\subsection{Cohomological correspondences for topological spaces}\label{subsection: cohomological correspondences}

\begin{Definition}\label{Definition: covering space}
    Let $X$ be a locally compact Hausdorff topological space. Let $d\in \Z_{>0}$. A \textbf{covering space of $X$ of degree $d$} is a locally compact Hausdorff topological space $Y$ together with a surjective continuous map $f: Y \rightarrow X$ with the property that for any $x\in X$, there exists an open neighbourhood $U_x\subset X$ of $x$ such that \[
        f^{-1}(U_x) = \bigsqcup_{i=1}^d V_{x, i},
    \]
    where each $V_{x, i}$ is an open subset in $Y$ and $f|_{V_{x_i}}: V_{x, i} \rightarrow U_{x}$ is a homeomorphism. 
\end{Definition}

\begin{Lemma}\label{Lemma: covering spaces are proper}
    Let $f: Y \rightarrow X$ be a covering space of degree $d$. Then, $f$ is proper. 
\end{Lemma}
\begin{proof}
    This is an easy exercise.
\end{proof}

Suppose we are given a diagram of topological spaces \[
    \begin{tikzcd}
        & C \arrow[ld, "\pr_1"']\arrow[rd, "\pr_2"]\\
        X & & X
    \end{tikzcd}
\]
such that $\pr_1: C \rightarrow X$ is a covering space of degree $d$. We also assume that we are given a sheaf of abelian groups $\scrF$ on $X$ such that \[
    \pr_1^{-1}\scrF \cong \pr_2^{-1}\scrF.
\]
Then, we have a sequence of cohomology groups \begin{equation}\label{eq: cohomological correspondence for topological spaces}
    H^i(X, \scrF) \xrightarrow{\pr_2^{-1}} H^i(C, \pr_2^{-1}\scrF) \cong H^i(C, \pr_1^{-1}\scrF) \cong H^i(X, \pr_{1, *}\pr_1^{-1}\scrF) \xrightarrow{\tr} H^i(X, \scrF),
\end{equation}
where the last trace map follows from Lemma \ref{Lemma: covering spaces are proper},  Lemma \ref{Lemma: lower shrieck of proper maps}, and  Lemma \ref{Lemma: trace map}. We term the composition the \textbf{\textit{(degree-$i$) cohomological correspondence}} with respect to the diagram $X \xleftarrow{\pr_1} C \xrightarrow{\pr_2} X$ and $\scrF$.

\begin{Remark}
    We finally remark that the construction above works for other cohomology theories. In particular, one also has a composition of morphisms \[
        H_c^i(X, \scrF) \xrightarrow{\pr_2^{-1}} H_c^i(C, \pr_2^{-1}\scrF) \cong H_c^i(C, \pr_1^{-1}\scrF) \cong H_c^i(X, \pr_{1, *}\pr_1^{-1}\scrF) \xrightarrow{\tr} H_c^i(X, \scrF),
    \] which is compatible with \eqref{eq: cohomological correspondence for topological spaces}. 
\end{Remark}

\end{appendix}

\printbibliography[heading=bibintoc]


\vspace{10mm}

\begin{tabular}[t]{l}
    P.-H.L.\\
    University of Leicester\\
    School of Computing and Mathematical Sciences\\
    Leicester, UK\\
    \textit{E-mail address: }\texttt{phl8@leicester.ac.uk}
\end{tabular}
\begin{tabular}[t]{l}
    J.-F.W.\\
    School of Mathematics and Statistics\\
    University College Dublin\\
    Belfield Dublin 4\\
    Ireland\\
    \textit{E-mail address: }\texttt{ju-feng.wu@ucd.ie}
\end{tabular}


\end{document}